\documentclass[11pt]{article}

\usepackage{amsmath,amssymb,amsthm,verbatim,amscd,eucal}
\usepackage{pictexwd}
\usepackage{latexsym}
\usepackage{graphicx}
\usepackage[usenames]{color} 
\usepackage{hyperref}

\numberwithin{equation}{section} 
\newtheorem{thm}[equation]{Theorem}
\newtheorem{prop}[equation]{Proposition}
\newtheorem{lemma}[equation]{Lemma}
\newtheorem{cor}[equation]{Corollary}
\newtheorem{example}[equation]{Example}
\newtheorem{remark}[equation]{Remark}

\newtheoremstyle{efronremark}
{6pt}{6pt}{}{}{\itshape}{\quad}{ }{\thmnote{#3}}

\theoremstyle{efronremark}   \newtheorem*{eremark}{}

\def\fsl{\mathfrak{gl}} 
\def\End{ \mathsf{End}}
\def\Res{\mathsf{Res}}
\def\min{\mathsf {min}}
\def\rank{\mathsf {rank}}
\def\id{{\rm id}}

\def\whitebull{\circ}
\def\redbull{\circ}

\newcommand{\cP}{\mathcal{P}}

\newcommand{\cM}{\mathcal{M}}
\newcommand{\Mn}{{\mathcal M}_n}
\newcommand{\Mk}{{\mathcal M}_k}
\newcommand{\Mtwok}{{\mathcal M}_{2k}}
\newcommand{\Motz}{\M_k(\mathsf{\zeta_{q}})}
  
\newcommand{\KK}{\mathbb{K}}

\newcommand{\C}{\mathsf{C}}  
\newcommand{\cb}{\mathsf{n}}
\newcommand{\G}{\mathsf{G}} 
\newcommand{\N}{\mathsf{N}} 
\newcommand{\NN}{\mathbb{N}} 
\newcommand{\LL}{\mathsf{L}} 
\newcommand{\LP}{\mathsf{LP}} 
\newcommand{\M}{\mathsf{M}} 
\newcommand{\J}{\mathsf{J}} 
\newcommand{\qq}{\mathsf{q}}

\newcommand{\RR}{\mathsf{R}} 
\newcommand{\RP}{\mathsf{RP}} 
\newcommand{\TL}{\mathsf{TL}}

\newcommand{\T}{\mathsf{T}}  

\newcommand{\UU}{\mathsf{U}}  
\newcommand{\VV}{\mathsf{V}}
\newcommand{\W}{\mathsf{W}}  

\newcommand{\uqsl}{\UU_{\mathsf q}(\fsl_2)}
\newcommand{\ranglet}{\rangle_{\mathsf t}}
\newcommand{\rangleb}{\rangle_{\mathsf b}} 
\newcommand{\zq}{\mathsf{\zeta_{q}}}

\begin{document}

\title{Motzkin Algebras} 

\author{\textsc{Georgia Benkart} \\
\textit{Department of Mathematics} \\ 
\textit{University of Wisconsin}  \\
\textit{Madison, WI 53706, USA}\\
\texttt{benkart@math.wisc.edu}
\and
\textsc{Tom Halverson}\\
\textit{Department of Mathematics}\\
\textit{Macalester College}\\
\textit{Saint Paul, MN 55105, USA}\\
\texttt{halverson@macalester.edu}} 

\date{June 18, 2013}
\maketitle

\begin{abstract}   We introduce an associative algebra $\M_k(x)$ whose dimension is the
$2k$-th Motzkin number.  The algebra $\M_k(x)$ has a basis of ``Motzkin diagrams,''
which are analogous to Brauer and Temperley-Lieb diagrams.  We show for a particular value of  $x$ that the algebra  $\M_k(x)$  is the centralizer algebra of the quantum enveloping algebra
$\uqsl$ acting on the $k$-fold tensor power of the sum of the 1-dimensional and 2-dimensional irreducible
$\uqsl$-modules.   We prove that $\M_k(x)$ is cellular in the sense of Graham and Lehrer  and 
construct  indecomposable $\M_k(x)$-modules which are the left cell modules. 
When $\M_k(x)$ is a semisimple algebra, these modules provide a complete set of representatives of isomorphism classes of irreducible
$\M_k(x)$-modules.  We compute the determinant of the Gram matrix of a bilinear form on the cell modules and use these determinants to show that $\M_k(x)$ is semisimple exactly when $x$ is not the root of certain Chebyshev polynomials. \medskip

\begin{eremark}[Keywords:] Motzkin number, Catalan number, Motzkin algebra, Temperley-Lieb algebra, $\fsl_2$ quantum enveloping algebra, Schur-Weyl duality, cellular algebra
\ 
\end{eremark}

\begin{eremark}[AMS 2010 Subject Classification:]
Primary 05E10,  Secondary 17B37, 81R50, 82B20.
\end{eremark}

\end{abstract}
 
 \bigskip

\begin{center} {\textbf{Acknowledgments}}   \end{center}
This paper was begun while the authors participated 
in the program  Combinatorial Representation Theory  
at the Mathematical Sciences Research Institute (MSRI) in Spring 2008.  They acknowledge
with gratitude the hospitality  and stimulating research
environment of MSRI.  G. Benkart is grateful to the Simons
Foundation for the support she received as  a Simons Visiting Professor at MSRI.  T. Halverson was partially supported by National Science Foundation grant DMS-0800085. The authors  also thank Anton Cox and Paul Martin for helping to clarify the connection to towers of recollement in Remark \ref{rem:CMPX}. 
\newpage

\tableofcontents

\begin{section}{Introduction}

The \emph{Motzkin  numbers} $\Mn$, $n=0,1,\dots$,   are defined by the generating function
\begin{align*}
\cM(t) & = \sum_{n \ge 0} \Mn t^n  =  \frac{1 - t - \sqrt{1-2t-3t^2}}{2t^2} \\
& = 1 + t + 2 t^2 + 4t^3 + 9t^4 + 21t^5 + 51 t^6 + 127 t^7 + 323 t^8 +  \cdots 
\end{align*}
which satisfies     
$$
\cM(t) = 1 + t \cM(t) + t^2 \cM^2(t).
$$
Although not as ubiquitous as their relatives the Catalan numbers, the Motzkin numbers  have appeared in a variety of combinatorial settings (see  \cite{M}, \cite{DS},  \cite{A}, \cite[6.38]{St}, \cite{E}).
For example, $\Mn$ counts 
\begin{itemize}
\item the number of ways of drawing any number of nonintersecting chords
among $n$ points on a circle; 
\item the number of lattice paths from (0,0)
to $(n,0)$ with steps $(1,-1)$, (1,0), (1,1), never going below the $x$-axis;  or
\item  the number of walks on $\mathbb N = \{0,1,\dots\}$
with $n$ steps from $\{-1,0,1\}$ starting and ending at 0; 
\item the number of standard tableaux of shape $\lambda$ with entries $\{1,\dots, n\}$ over all 
partitions $\lambda$ of $n$ with no more than 3 parts.    \end{itemize}

In Section 2,  we introduce a unital associative algebra $\M_k(x)$, depending on the parameter $x$,  whose dimension is the Motzkin number
$\Mtwok$.   The algebra $\M_k(x)$ is defined over
any commutative ring $\KK$ with $1$ and  has a $\KK$-basis of diagrams similar to those
for the Temperley-Lieb,  Brauer, and  partition algebras.  We refer 
to the algebras $\M_k(x)$ as  {\it Motzkin algebras}  because of their close association with the Motzkin numbers.
Although in analogy with the rook algebras in \cite{Mu1},\cite{Mu2}, \cite{So1}, \cite{So2}, \cite{FHH}, the
rook partition algebras in \cite{Gr}, and  the rook Brauer algebras in \cite{dH}, \cite{MM}, the name ``rook Temperley-Lieb algebras" or ``partial Temperley-Lieb algebras" 
would also be appropriate.

 In Section 2.3, we show that the algebra $\M_k(x)$ is  generated by certain diagrams $p_j \ (1 \leq j \leq k)$,  $\ell_i,\,r_i,\,t_i \  (1 \leq i < k)$,   where
 the subalgebra $\TL_k(x)$ generated by the $t_i \ (1 \leq i < k)$ is the well-studied Temperley-Lieb algebra  
with dimension the Catalan number ${\mathcal C}_k = \frac{1}{k+1}{{2k}\choose k}$.    The 
Temperley-Lieb algebra  first arose
in statistical mechanics \cite{TL} and has played a prominent role in Jones'
work (\cite {Jo1}, \cite{Jo3}) on
subfactors of von Neumann algebras and invariants of knots and links \cite{Jo2}.        
We show that the algebra   $\M_k(x)$ has a factorization  $\M_k(x) = \RP_k \TL_k(x)\,  \LP_k$ into three subalgebras, 
and every basis diagram of $\M_k(x)$ can be expressed as a product of diagrams, 
one from each of the three components. 
The subalgebra $\LP_k$ is  generated
by the elements $\ell_i  \  (1 \leq i < k)$ and $p_j \ (1 \leq j \leq k)$ and has dimension the Catalan number
${\mathcal C}_{k+1}$.    The  $r_i,   \  (1 \leq i < k)$ and $p_j \ (1 \leq j \leq k)$ generate the subalgebra $\RP_k$, which  is anti-isomorphic to  $\LP_k$ via 
an involution that interchanges $\RP_k$ and $\LP_k$
and restricts to an involution on the Temperley-Lieb subalgebra $\TL_k(x)$, so
it too has dimension  ${\mathcal C}_{k+1}$.  

In Section 3, we consider the  action  
of the quantum enveloping algebra $\uqsl$  on the $k$-fold tensor power ${\VV}^{\otimes k}$,
where $\VV = \VV(0) \oplus \VV(1)$,  the sum of the trivial $1$-dimensional module $\VV(0)$
and the natural $2$-dimensional module $\VV(1)$ for $\uqsl$  when $\KK$ is a field and $\qq \in \KK$ is nonzero and not a root of unity.   We show the centralizer algebra of the $\uqsl$-action is an algebra
$\M_k$ whose dimension is the Motzkin number ${\mathcal M}_{2k}$.   The algebra
$\M_k$ is semisimple and has irreducible modules $\M_k^{(r)}$ indexed by 
the integers $r = 0,1,\dots, k$.  
The dimension of $\M_k^{(r)}$ is $\mathsf{m}_{k,r}$, the number of walks on $\mathbb N$ with
$k$ steps from  $\{-1,0,1\}$ starting at 0 and ending at $r$, and ${\mathcal M}_{2k} = \sum_{r=0}^k
\mathsf{m}_{k,r}^2$.   
By comparison, the centralizer algebra
of the action of $\uqsl$  on $\VV(1)^{\otimes k}$ is the  
Temperley-Lieb algebra $\mathsf{TL}_k(\qq+\qq^{-1})$. 
Using facts about the Temperley-Lieb algebra and its irreducible modules, we 
derive a closed form expression for the numbers  $\mathsf{m}_{k,r}$.

Our main results in Section 3 are Theorem \ref{thm:main}, which shows that the centralizer algebra $\M_k$ is
isomorphic to the Motzkin algebra 
 $\Motz$, where $\zq = 1-\qq-\qq^{-1}$, and Theorem \ref{thm:tensordecomp},  which gives an explicit 
 decomposition of $\VV^{\otimes k}$ into irreducible $\uqsl$-modules.  
 Finding this decomposition was one of our original motivations for introducing the Motzkin algebras.    It follows that 
 when $\qq \rightarrow 1$, the algebra $\M_k(-1)$ is the centralizer algebra of the action of $\fsl_2$ on tensor
 powers of the sum of its natural 2-dimensional and trivial representations.  
 
In Section 4, we define an action of any Motzkin algebra $\M_k(x)$  
 on certain paths, which we term {\em Motzkin paths} and which are related to the lattice paths
and to the walks on $\mathbb N$ mentioned earlier.   The Motzkin paths enable us to construct indecomposable $\M_k(x)$-modules  $\C_k^{(r)}$ labeled by 
$r = 0,1,\dots, k$.   By Theorem \ref{thm:irredmod},   $\{\C_k^{(r)} \mid r = 0,1,\dots, k\}$ is a complete set of representatives
of isomorphism classes of irreducible $\M_k(x)$-modules when
the algebra $\M_k(x)$ is semisimple.  Section 4.2 shows that  algebra $\M_k(x)$ is cellular in the sense of Graham and Lehrer \cite{GL} and that the modules $\C_k^{(r)}$ are the left cell representations (see Theorem \ref{thm:cellmod}).   
The values of the characters of these modules on  basis diagrams  
are equal to the numbers $\mathsf{m}_{\ell,r}$ for $0 \leq \ell \leq k$ according to Proposition \ref{prop:charac}.
In Theorem \ref{thm:GramDet}, we give an explicit formula for  the determinant of the Gram matrices for the modules
$\C_k^{(r)}$ in terms of Chebyshev polynomials of the second kind.    Theorem \ref{thm:cheby} provides  a criterion for the algebra $\M_k(x)$ to be semisimple using expressions in the roots of the Chebyshev polynomials.    
 
 The results in this paper give representation-theoretic interpretations  of  some known identities \eqref{eq:MCrel}, \eqref{eq:restrictionrule} involving Motzkin numbers.   Equations \eqref{eq:sumofsqs}, \eqref{eq:bimoduleidentity}, and \eqref{eq:MT} provide additional Motzkin identities which arise naturally in this
setting.  
\end{section} 
  
\begin{section}{The Motzkin Algebra}

\begin{subsection}{Motzkin diagrams}

A \emph{Motzkin $k$-diagram} consists of two rows each with $k$ vertices such that each vertex is connected to at most one other vertex
by an edge, and the edges are {\em planar}, which is to say they can be drawn without crossing while staying inside the rectangle formed by the vertices.    For example
$$
d =  
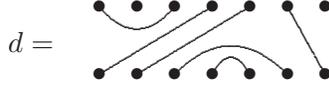
{\beginpicture
\setcoordinatesystem units <0.5cm,0.3cm>        
\setplotarea x from 0 to 8, y from -1 to 2
\put{$\bullet$} at  1 -1  \put{$\bullet$} at  1 2
\put{$\bullet$} at  2 -1  \put{$\bullet$} at  2 2
\put{$\bullet$} at  3 -1  \put{$\bullet$} at  3 2
\put{$\bullet$} at  4 -1  \put{$\bullet$} at  4 2
\put{$\bullet$} at  5 -1  \put{$\bullet$} at  5 2
\put{$\bullet$} at  6 -1  \put{$\bullet$} at  6 2
\put{$\bullet$} at  7 -1  \put{$\bullet$} at  7 2
\plot  4 2  1 -1 /
\plot 5 2 2 -1 /
\plot 6 2 7 -1  /
\setquadratic
\plot 1 2 2 1 3 2 /
\plot 3 -1 4.5 .25 6 -1 /
\plot 4 -1 4.5 -.25 5 -1 /
\endpicture}
$$
is a Motzkin $7$-diagram, and   
$$
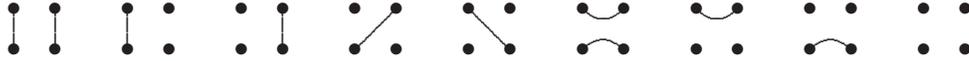
\begin{array}{ccccc}
{\beginpicture
\setcoordinatesystem units <0.55cm,0.55cm>        
\setplotarea x from 1 to 2, y from 0 to 1
\put{$\bullet$} at  1 0  \put{$\bullet$} at  1 1
\put{$\bullet$} at  2 0  \put{$\bullet$} at  2 1
\plot 1 0 1 1 /
\plot 2 0 2 1 /
\endpicture} \qquad
{\beginpicture
\setcoordinatesystem units <0.55cm,0.55cm>        
\setplotarea x from 1 to 2, y from 0 to 1
\put{$\bullet$} at  1 0  \put{$\bullet$} at  1 1
\put{$\bullet$} at  2 0  \put{$\bullet$} at  2 1
\plot 1 0 1 1 /
\endpicture}  \qquad
{\beginpicture
\setcoordinatesystem units <0.55cm,0.55cm>        
\setplotarea x from 1 to 2, y from 0 to 1
\put{$\bullet$} at  1 0  \put{$\bullet$} at  1 1
\put{$\bullet$} at  2 0  \put{$\bullet$} at  2 1
\plot 2 0 2 1 /
\endpicture} \qquad
{\beginpicture
\setcoordinatesystem units <0.55cm,0.55cm>        
\setplotarea x from 1 to 2, y from 0 to 1
\put{$\bullet$} at  1 0  \put{$\bullet$} at  1 1
\put{$\bullet$} at  2 0  \put{$\bullet$} at  2 1
\plot 1 0 2 1 /
\endpicture} \qquad
{\beginpicture
\setcoordinatesystem units <0.55cm,0.55cm>        
\setplotarea x from 1 to 2, y from 0 to 1
\put{$\bullet$} at  1 0  \put{$\bullet$} at  1 1
\put{$\bullet$} at  2 0  \put{$\bullet$} at  2 1
\plot 1 1 2 0 /
\endpicture} \qquad
{\beginpicture
\setcoordinatesystem units <0.55cm,0.55cm>        
\setplotarea x from 1 to 2, y from 0 to 1
\put{$\bullet$} at  1 0  \put{$\bullet$} at  1 1
\put{$\bullet$} at  2 0  \put{$\bullet$} at  2 1
\setquadratic
\plot 1 1 1.5 .75 2 1 /
\plot 1 0 1.5 .25 2 0 /
\endpicture} \qquad
{\beginpicture
\setcoordinatesystem units <0.55cm,0.55cm>        
\setplotarea x from 1 to 2, y from 0 to 1
\put{$\bullet$} at  1 0  \put{$\bullet$} at  1 1
\put{$\bullet$} at  2 0  \put{$\bullet$} at  2 1
\setquadratic
\plot 1 1 1.5 .75 2 1 /
\endpicture} \qquad
{\beginpicture
\setcoordinatesystem units <0.55cm,0.55cm>        
\setplotarea x from 1 to 2, y from 0 to 1
\put{$\bullet$} at  1 0  \put{$\bullet$} at  1 1
\put{$\bullet$} at  2 0  \put{$\bullet$} at  2 1
\setquadratic
\plot 1 0 1.5 .25 2 0 /
\endpicture} \qquad
{\beginpicture
\setcoordinatesystem units <0.55cm,0.55cm>        
\setplotarea x from 1 to 2, y from 0 to 1
\put{$\bullet$} at  1 0  \put{$\bullet$} at  1 1
\put{$\bullet$} at  2 0  \put{$\bullet$} at  2 1
\endpicture}
\end{array}
$$
are the nine Motzkin 2-diagrams.  The number of Motzkin $k$-diagrams  is the same as the number of ways of drawing any number of nonintersecting chords among $2k$ points on a circle, which is known \cite[6.38.a]{St} to be the Motzkin number $\mathcal{M}_{2k}$.  

 A Motzkin $k$-diagram with $k$ edges  is a \emph{Temperley-Lieb diagram} (see \cite{TL} or \cite{GHJ}).   The number of Temperley-Lieb diagrams is the number of ways of drawing $k$  nonintersecting chords on $2k$ points on a circle, which is the Catalan number
 \begin{equation}
\mathcal{C}_{k}= \frac{1}{k+1} \binom{2k}{k}.
 \end{equation}
 We can count the number of  Motzkin $k$-diagrams according to the number $n$ of edges  in the diagram. First choose, in $\binom{2k}{2n}$ ways, the $2n$ vertices to be connected by the $n$ edges. Then the number of planar ways to connect these edges is the Catalan number $\mathcal{C}_{n}$. Thus,   
  \begin{equation}\label{eq:MCrel}
\mathcal{M}_{2k}= \sum_{n=0}^k  \binom{2k}{2n} \mathcal{C}_n = \sum_{n=0}^k 
\frac{1}{n+1} \binom{2k}{2n}\binom{2n}{n}.
 \end{equation}
 
 \end{subsection}
 
 \begin{subsection}{The Motzkin algebras $\M_k(x,y)$ and $\M_k(x)$}\label{subsec:motzkinalgebra}
 
Assume $\KK$ is a commutative ring with 1 and that $x$ and $y$ are elements of $\KK$.   
Set  $\M_0(x,y) = \KK 1$, and for $k \geq 1$,  let $\M_k(x,y)$ be the free $\KK$-module with basis the
Motzkin $k$-diagrams.   
We multiply two Motzkin $k$-diagrams $d_1$ and $d_2$ as follows. Place $d_1$ above $d_2$ and identify the bottom-row vertices in $d_1$ with the corresponding top-row vertices in $d_2$. The product is $d_1 d_2 =x^{\kappa(d_1,d_2)} y^{\upsilon(d_1,d_2)} d_3$, where $d_3$ is the diagram consisting 
of the horizontal edges from the top row of $d_1$, the horizontal edges from the bottom row of $d_2$, and the vertical edges that propagate from the bottom of $d_2$ to  the top of $d_1$;   $\kappa(d_1,d_2)$ is the number of loops that arise in the middle row;  and  and $\upsilon(d_1,d_2)$ is the number of isolated vertices that remain in the middle row.  For example, if
$$
d_1 = 
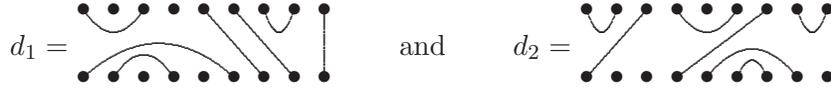
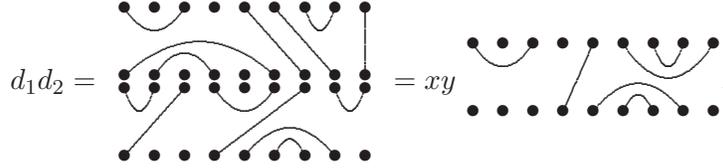
{\beginpicture
\setcoordinatesystem units <0.4cm,0.3cm>        
\setplotarea x from 1 to 9, y from -1 to 2
\put{$\bullet$} at  1 -1  \put{$\bullet$} at  1 2
\put{$\bullet$} at  2 -1  \put{$\bullet$} at  2 2
\put{$\bullet$} at  3 -1  \put{$\bullet$} at  3 2
\put{$\bullet$} at  4 -1  \put{$\bullet$} at  4 2
\put{$\bullet$} at  5 -1  \put{$\bullet$} at  5 2
\put{$\bullet$} at  6 -1  \put{$\bullet$} at  6 2
\put{$\bullet$} at  7 -1  \put{$\bullet$} at  7 2
\put{$\bullet$} at  8 -1  \put{$\bullet$} at  8 2
\put{$\bullet$} at  9 -1  \put{$\bullet$} at  9 2
\plot 5 2  7 -1 /
\plot 6 2 8 -1 /
\plot 9 2 9 -1 /
\setquadratic
\plot 1 2 2 1 3 2 /
\plot 7 2 7.5 1 8 2 /
\plot 1 -1 3.5 .5 6 -1 /
\plot 2 -1 3 0 4 -1 /
\endpicture}
\qquad\hbox{ and }\qquad
d_2 = {\beginpicture
\setcoordinatesystem units <0.4cm,0.3cm>        
\setplotarea x from 1 to 9, y from -1 to 2
\put{$\bullet$} at  1 -1  \put{$\bullet$} at  1 2
\put{$\bullet$} at  2 -1  \put{$\bullet$} at  2 2
\put{$\bullet$} at  3 -1  \put{$\bullet$} at  3 2
\put{$\bullet$} at  4 -1  \put{$\bullet$} at  4 2
\put{$\bullet$} at  5 -1  \put{$\bullet$} at  5 2
\put{$\bullet$} at  6 -1  \put{$\bullet$} at  6 2
\put{$\bullet$} at  7 -1  \put{$\bullet$} at  7 2
\put{$\bullet$} at  8 -1  \put{$\bullet$} at  8 2
\put{$\bullet$} at  9 -1  \put{$\bullet$} at  9 2
\plot 3 2  1 -1 /
\plot 7 2 4 -1 /
\setquadratic
\plot 1 2 1.5 1 2 2 /
\plot 4 2 5 1 6 2 /
\plot 8 2 8.5 1 9 2 / 
\plot 5 -1 6.5 .25 8 -1 /
\plot 6 -1 6.5 -.25 7 -1 /
\endpicture}
$$
then
$$
d_1 d_2 = 
\begin{array}{l}
{\beginpicture
\setcoordinatesystem units <0.4cm,0.3cm>        
\setplotarea x from 1 to 9, y from -1 to 2
\put{$\bullet$} at  1 -1  \put{$\bullet$} at  1 2
\put{$\bullet$} at  2 -1  \put{$\bullet$} at  2 2
\put{$\bullet$} at  3 -1  \put{$\bullet$} at  3 2
\put{$\bullet$} at  4 -1  \put{$\bullet$} at  4 2
\put{$\bullet$} at  5 -1  \put{$\bullet$} at  5 2
\put{$\bullet$} at  6 -1  \put{$\bullet$} at  6 2
\put{$\bullet$} at  7 -1  \put{$\bullet$} at  7 2
\put{$\bullet$} at  8 -1  \put{$\bullet$} at  8 2
\put{$\bullet$} at  9 -1  \put{$\bullet$} at  9 2
\plot 5 2  7 -1 /
\plot 6 2 8 -1 /
\plot 9 2 9 -1 /
\setquadratic
\plot 1 2 2 1 3 2 /
\plot 7 2 7.5 1 8 2 /
\plot 1 -1 3.5 .5 6 -1 /
\plot 2 -1 3 0 4 -1 /
\endpicture}
\\
{\beginpicture
\setcoordinatesystem units <0.4cm,0.3cm>        
\setplotarea x from 1 to 9, y from -1 to 2
\put{$\bullet$} at  1 -1  \put{$\bullet$} at  1 2
\put{$\bullet$} at  2 -1  \put{$\bullet$} at  2 2
\put{$\bullet$} at  3 -1  \put{$\bullet$} at  3 2
\put{$\bullet$} at  4 -1  \put{$\bullet$} at  4 2
\put{$\bullet$} at  5 -1  \put{$\bullet$} at  5 2
\put{$\bullet$} at  6 -1  \put{$\bullet$} at  6 2
\put{$\bullet$} at  7 -1  \put{$\bullet$} at  7 2
\put{$\bullet$} at  8 -1  \put{$\bullet$} at  8 2
\put{$\bullet$} at  9 -1  \put{$\bullet$} at  9 2
\plot 3 2  1 -1 /
\plot 7 2 4 -1 /
\setquadratic
\plot 1 2 1.5 1 2 2 /
\plot 4 2 5 1 6 2 /
\plot 8 2 8.5 1 9 2 / 
\plot 5 -1 6.5 .25 8 -1 /
\plot 6 -1 6.5 -.25 7 -1 /
\endpicture}
\end{array}
= x y\ 
{\beginpicture
\setcoordinatesystem units <0.4cm,0.3cm>        
\setplotarea x from 1 to 9, y from -1 to 2
\put{$\bullet$} at  1 -1  \put{$\bullet$} at  1 2
\put{$\bullet$} at  2 -1  \put{$\bullet$} at  2 2
\put{$\bullet$} at  3 -1  \put{$\bullet$} at  3 2
\put{$\bullet$} at  4 -1  \put{$\bullet$} at  4 2
\put{$\bullet$} at  5 -1  \put{$\bullet$} at  5 2
\put{$\bullet$} at  6 -1  \put{$\bullet$} at  6 2
\put{$\bullet$} at  7 -1  \put{$\bullet$} at  7 2
\put{$\bullet$} at  8 -1  \put{$\bullet$} at  8 2
\put{$\bullet$} at  9 -1  \put{$\bullet$} at  9 2
\plot 5 2  4 -1 /
\setquadratic
\plot 6 2 7.5 .5 9 2 /
\plot 1 2 2 1 3 2 /
\plot 7 2 7.5 1 8 2 /
\plot 5 -1 6.5 .25 8 -1 /
\plot 6 -1 6.5 -.25 7 -1 /
\endpicture}.
$$
Diagram multiplication makes $\M_k(x,y)$ into an associative algebra with identity element
$$
\mathbf{1}_k = 
{\beginpicture
\setcoordinatesystem units <0.35cm,0.15cm>        
\setplotarea x from 1 to 4, y from -1 to 2
\put{$\bullet$} at  1 -1  \put{$\bullet$} at  1 2
\put{$\bullet$} at  2 -1  \put{$\bullet$} at  2 2
\put{$\cdots$} at  3.5 -1  \put{$\cdots$} at  3.5 2
\put{$\bullet$} at  5 -1  \put{$\bullet$} at  5 2
\plot 1 2 1 -1 /
\plot 2 2 2 -1 /
\plot 5 2 5 -1 /
\endpicture}.
$$
Note that under this multiplication, two vertical 
edges can become a single horizontal edge as in the example above, 
and a vertical edge can contract to a vertex.  
The {\em rank} of a diagram $\rank(d)$ is the number of vertical edges in the diagram.  Therefore,  
\begin{equation}\label{eq:rank}
\rank( d_1 d_2) \le \min( \rank(d_1), \rank(d_2) ).
\end{equation} 

 In this paper we restrict our attention to the specialization $\M_k(x) := \M_k(x,1)$.
For $0 \le r \le k$, we let $\J_r \subseteq \M_k(x)$ be the $\KK$-span of the Motzkin $k$-diagrams of rank \emph{less than or equal} to $r$.   Then, by \eqref{eq:rank}, $\J_r$ is a two-sided ideal in $\M_k(x)$ and we have the tower of ideals,
\begin{equation}\label{eq:TowerOfIdeals}
\J_0 \subseteq \J_1 \subseteq \J_2 \subseteq \cdots \subseteq \J_k = \M_k(x).
\end{equation}  \end{subsection}

 \begin{subsection}{Generators for the Motzkin algebra $\M_k(x)$}\label{subsec:generators}
 
For $1 \le i \le k-1$,  consider the following diagrams in $\M_k(x)$,
\begin{equation}\label{eq:trlp}
t_i = {\beginpicture
\setcoordinatesystem units <0.35cm,0.15cm>        
\setplotarea x from 1 to 9, y from -1 to 4
\put{$\bullet$} at  1 -1  \put{$\bullet$} at  1 2
\put{$\bullet$} at  2 -1  \put{$\bullet$} at  2 2
\put{$\cdots$} at  3 -1  \put{$\cdots$} at  3 2
\put{$\bullet$} at  4 -1  \put{$\bullet$} at  4 2
\put{$\bullet$} at  5 -1  \put{$\bullet$} at  5 2
\put{$\bullet$} at  6 -1  \put{$\bullet$} at  6 2
\put{$\bullet$} at  7 -1  \put{$\bullet$} at  7 2
\put{$\cdots$} at  8 -1  \put{$\cdots$} at  8 2
\put{$\bullet$} at  9 -1  \put{$\bullet$} at  9 2
\plot 1 2 1 -1 /
\plot 2 2 2 -1 /
\plot 4 2 4 -1 /
\plot 7 2 7 -1 /
\plot 9 2 9 -1 /
\setquadratic
\plot 5 2 5.5 1 6 2 /
\plot 5 -1 5.5 0 6 -1 /
\put{$\scriptstyle{i\,}$} at 5 3.85 
\put{$\scriptstyle{i\!+\!1}$} at 6 3.75 
\endpicture}, \quad 
r_i = {\beginpicture
\setcoordinatesystem units <0.35cm,0.15cm>        
\setplotarea x from 1 to 9, y from -1 to 4
\put{$\bullet$} at  1 -1  \put{$\bullet$} at  1 2
\put{$\bullet$} at  2 -1  \put{$\bullet$} at  2 2
\put{$\cdots$} at  3 -1  \put{$\cdots$} at  3 2
\put{$\bullet$} at  4 -1  \put{$\bullet$} at  4 2
\put{$\bullet$} at  5 -1  \put{$\bullet$} at  5 2
\put{$\bullet$} at  6 -1  \put{$\bullet$} at  6 2
\put{$\bullet$} at  7 -1  \put{$\bullet$} at  7 2
\put{$\cdots$} at  8 -1  \put{$\cdots$} at  8 2
\put{$\bullet$} at  9 -1  \put{$\bullet$} at  9 2
\plot 1 2 1 -1 /
\plot 2 2 2 -1 /
\plot 4 2 4 -1 /
\plot 7 2 7 -1 /
\plot 9 2 9 -1 /
\plot 5 -1  6 2 /
\put{$\scriptstyle{i\,}$} at 5 3.85 
\put{$\scriptstyle{i\!+\!1}$} at 6 3.75 
\endpicture}, \quad 
\ell_i = {\beginpicture
\setcoordinatesystem units <0.35cm,0.15cm>        
\setplotarea x from 1 to 9, y from -1 to 4
\put{$\bullet$} at  1 -1  \put{$\bullet$} at  1 2
\put{$\bullet$} at  2 -1  \put{$\bullet$} at  2 2
\put{$\cdots$} at  3 -1  \put{$\cdots$} at  3 2
\put{$\bullet$} at  4 -1  \put{$\bullet$} at  4 2
\put{$\bullet$} at  5 -1  \put{$\bullet$} at  5 2
\put{$\bullet$} at  6 -1  \put{$\bullet$} at  6 2
\put{$\bullet$} at  7 -1  \put{$\bullet$} at  7 2
\put{$\cdots$} at  8 -1  \put{$\cdots$} at  8 2
\put{$\bullet$} at  9 -1  \put{$\bullet$} at  9 2
\plot 1 2 1 -1 /
\plot 2 2 2 -1 /
\plot 4 2 4 -1 /
\plot 7 2 7 -1 /
\plot 9 2 9 -1 /
\plot 6 -1  5 2 /
\put{$\scriptstyle{i}$\,} at 5 3.85 
\put{$\scriptstyle{i\!+\!1}$} at 6 3.75 
\endpicture}.
\end{equation}
It is well known that  the $t_i$'s generate the Temperley-Lieb algebra $\TL_k(x)$ (see for example \cite{GHJ}).  
For $1 \le i \le k$, let
\begin{equation}
p_i = {\beginpicture
\setcoordinatesystem units <0.35cm,0.15cm>        
\setplotarea x from 1 to 9, y from -1 to 4
\put{$\bullet$} at  1 -1  \put{$\bullet$} at  1 2
\put{$\bullet$} at  2 -1  \put{$\bullet$} at  2 2
\put{$\cdots$} at  3 -1  \put{$\cdots$} at  3 2
\put{$\bullet$} at  4 -1  \put{$\bullet$} at  4 2
\put{$\bullet$} at  5 -1  \put{$\bullet$} at  5 2
\put{$\bullet$} at  6 -1  \put{$\bullet$} at  6 2
\put{$\bullet$} at  7 -1  \put{$\bullet$} at  7 2
\put{$\cdots$} at  8 -1  \put{$\cdots$} at  8 2
\put{$\bullet$} at  9 -1  \put{$\bullet$} at  9 2
\plot 1 2 1 -1 /
\plot 2 2 2 -1 /
\plot 4 2 4 -1 /
\plot 6 2 6 -1 /
\plot 7 2 7 -1 /
\plot 9 2 9 -1 /
\put{$\scriptstyle{i}$} at 5 3.75  
\endpicture}.
\end{equation}
Diagram multiplication shows that
\begin{equation}\label{eq: projs}  p_1 = r_1 \ell_1, \quad 
p_i = r_i \ell_i = \ell_{i-1} r_{i-1}, \ \hbox{\rm  for} \ \ 1 < i  < k, \ \ \  \hbox{\rm and} \ 
\  p_k = \ell_{k-1} r_{k-1}.  \end{equation}
In this section, we prove that the diagrams $t_i, \ell_i, r_i$ ($1 \leq i < k$) generate $\M_k(x)$.

Let  $\RP_k \subseteq \M_k(x)$ denote the subalgebra with identity element
$\mathbf{1}_k$  spanned by the diagrams containing no horizontal edges and no vertical edges directed to the northwest;  and analogously, let  $\LP_k \subseteq \M_k(x)$ denote the subalgebra with $\mathbf{1}_k$  spanned by the diagrams containing no horizontal edges and no vertical edges directed 
to the northeast.    For example, 
$$
{\beginpicture
\setcoordinatesystem units <0.4cm,0.2cm>        
\setplotarea x from 1 to 7, y from -1.5 to 2.5
\put{$\bullet$} at  1 -1  \put{$\bullet$} at  1 2
\put{$\bullet$} at  2 -1  \put{$\bullet$} at  2 2
\put{$\bullet$} at  3 -1  \put{$\bullet$} at  3 2
\put{$\bullet$} at  4 -1  \put{$\bullet$} at  4 2
\put{$\bullet$} at  5 -1  \put{$\bullet$} at  5 2
\put{$\bullet$} at  6 -1  \put{$\bullet$} at  6 2
\put{$\bullet$} at  7 -1  \put{$\bullet$} at  7 2
\put{$\bullet$} at  8 -1  \put{$\bullet$} at  8 2
\put{$\bullet$} at  9 -1  \put{$\bullet$} at  9 2
\put{$\bullet$} at  10 -1  \put{$\bullet$} at  10 2
\plot 2 2  1  -1 /
\plot 4 2  2  -1 /
\plot 5 2  3  -1 /
\plot 7 2  4  -1 /
\plot 8 2 8 -1 / 
\plot 10 2 9 -1 /
\endpicture} \in \RP_{10} \quad\text{ and } \quad
{\beginpicture
\setcoordinatesystem units <0.4cm,0.2cm>        
\setplotarea x from 1 to 7, y from -1.5 to 2.5
\put{$\bullet$} at  1 -1  \put{$\bullet$} at  1 2
\put{$\bullet$} at  2 -1  \put{$\bullet$} at  2 2
\put{$\bullet$} at  3 -1  \put{$\bullet$} at  3 2
\put{$\bullet$} at  4 -1  \put{$\bullet$} at  4 2
\put{$\bullet$} at  5 -1  \put{$\bullet$} at  5 2
\put{$\bullet$} at  6 -1  \put{$\bullet$} at  6 2
\put{$\bullet$} at  7 -1  \put{$\bullet$} at  7 2
\put{$\bullet$} at  8 -1  \put{$\bullet$} at  8 2
\put{$\bullet$} at  9 -1  \put{$\bullet$} at  9 2
\put{$\bullet$} at  10 -1  \put{$\bullet$} at  10 2
\plot 1 2  1  -1 /
\plot 2 2  3  -1 /
\plot 3 2  5  -1 /
\plot 5 2  6  -1 /
\plot 6 2 8 -1 / 
\plot 7 2 9 -1 /
\endpicture} \in \LP_{10}.
$$

For $1 \le j < i  \le k$, set  
\begin{equation}\label{eq:rijdef}
r_{i,j} = r_{i-1} r_{i-2} \cdots r_j = 
{\beginpicture
\setcoordinatesystem units <0.35cm,0.2cm>         
\setplotarea x from 1 to 13, y from -1.5 to 2.5
\put{$\bullet$} at  1 -1  \put{$\bullet$} at  1 2
\put{$\bullet$} at  2 -1  \put{$\bullet$} at  2 2
\put{$\cdots$} at  3 -1  \put{$\cdots$} at  3 2
\put{$\bullet$} at  4 -1  \put{$\bullet$} at  4 2
\put{$\bullet$} at  5 -1  \put{$\bullet$} at  5 2
\put{$\bullet$} at  6 -1  \put{$\bullet$} at  6 2
\put{$\cdots$} at  7 -1  \put{$\cdots$} at  7 2
\put{$\bullet$} at  8 -1  \put{$\bullet$} at  8 2
\put{$\bullet$} at  9 -1  \put{$\bullet$} at  9 2
\put{$\bullet$} at  10 -1  \put{$\bullet$} at  10 2
\put{$\bullet$} at  11 -1  \put{$\bullet$} at  11 2
\put{$\cdots$} at  12 -1  \put{$\cdots$} at  12 2
\put{$\bullet$} at  13 -1  \put{$\bullet$} at  13 2
\plot 1 2  1  -1 /
\plot 2 2  2  -1 /
\plot 4 2  4  -1 /
\plot 5 -1  9  2 /
\plot 10 2  10  -1 /
\plot 10 2  10  -1 /
\plot 11 2  11  -1 /
\plot 13 2  13  -1 /
\put{$\scriptstyle{j}$} at 5 3.75
\put{$\scriptstyle{i}$} at 9 3.75
\endpicture}
\end{equation}
and let $r_{i,i} = \ell_{i,i} = \mathbf{1}_k$ for $1 \leq i \leq k$.   

Let $d$ be a diagram in $\RP_k$.    Thus, $d$ has no horizontal edges. 
Let ${ I} = {I}(d) = \{i_1 < i_2 < \dots < i_s\}$ denote the vertices on the
top row of $d$ which connect to vertices in the bottom row of $d$, and
let ${J} = {J}(d) = \{j_1 < j_2 < \dots < j_s\}$ denote the corresponding
vertices in the bottom row of $d$ so that $i_q$ and $j_q$
are connected by an edge for $q=1,\dots,s$. 
In order for $d$ to belong to $\RP_k$,  we must have $i_q  \ge j_q$ for each $q = 1,\dots, s$. It is easy to see (by multiplying diagrams) that
$$
d = \bigg(\prod_{i \not\in I} p_i\bigg)   r_{i_1,j_1} r_{i_2,j_2} \cdots r_{i_s,j_s} \bigg(\prod_{j \not\in J} p_j \bigg).
$$
This shows that 
$$\RP_k = \langle \mathbf{1}_k, r_1, \ldots, r_{k-1}, p_1, \ldots, p_k \rangle.$$

There is an involution  ``$\ast$''  (a $\KK$-linear anti-automorphism of order 2) on $\M_k(x)$ 
which  interchanges the vertices in the top and bottom rows  of a diagram while maintaining all the edge connections.   Thus, for Motzkin diagrams $d_1,d_2$,  the involution satisfies $(d_1 d_2)^\ast = d_2^\ast d_1^\ast$ and $(d^\ast)^\ast = d$.
It  has the following effect on
the elements in \eqref{eq:trlp}:

\begin{equation}\label{eq:invol} t_i^\ast = t_i, \quad  \ell_i^\ast = r_i, \quad  r_i^\ast= \ell_i  \quad (1 \leq i < k),
\qquad p_j^\ast = p_j  \quad (1 \leq j \leq k). \end{equation}
As a result,  the subalgebra $\LP_k$  is anti-isomorphic to $\RP_k$ and 
$$\LP_k  = \langle \mathbf{1}_k, \ell_1, \ldots, \ell_{k-1}, p_1, \ldots, p_k \rangle.$$
Similarly, the subalgebra $\LL_k$ 
with $\mathbf{1}_k$ generated by the $\ell_i$ ($1 \leq i < k$)  is anti-isomorphic to  $\RR_k$,
the subalgebra with $\mathbf{1}_k$ generated by the $r_i$ ($1 \leq i < k$).

Every diagram in $d\in \M_k(x)$ can be factored as
\begin{equation}\label{Factorization}
d = r t \ell, \qquad \hbox{with $r \in \RP_k$, \  $t \in \TL_k(x)$,  and \ $\ell \in \LP_k$. }
\end{equation}
For example,
\begin{equation}\label{FactorizationExample}
d =  {\beginpicture
\setcoordinatesystem units <0.4cm,0.2cm>        
\setplotarea x from 0 to 7, y from -1.5 to 2.5
\put{$\bullet$} at  0 -1  \put{$\bullet$} at  0 2
\put{$\bullet$} at  1 -1  \put{$\bullet$} at  1 2
\put{$\bullet$} at  2 -1  \put{$\bullet$} at  2 2
\put{$\bullet$} at  3 -1  \put{$\bullet$} at  3 2
\put{$\bullet$} at  4 -1  \put{$\bullet$} at  4 2
\put{$\bullet$} at  5 -1  \put{$\bullet$} at  5 2
\put{$\bullet$} at  6 -1  \put{$\bullet$} at  6 2
\put{$\bullet$} at  7 -1  \put{$\bullet$} at  7 2
\plot 1 2  0 -1 /
\plot 5 2 2 -1 /
\plot 6 2 7 -1  /
\setquadratic
\plot 2 2 3 1 4 2 /
\plot 3 -1 4.5 .25 6 -1 /
\plot 4 -1 4.5 -.25 5 -1 /
\endpicture}
=
 \begin{array}{l}
 {\beginpicture
\setcoordinatesystem units <0.4cm,0.2cm>        
\setplotarea x from 0 to 7, y from -1.5 to 2.5
\put{$\bullet$} at  0 -1  \put{$\bullet$} at  0 2
\put{$\bullet$} at  1 -1  \put{$\bullet$} at  1 2
\put{$\bullet$} at  2 -1  \put{$\bullet$} at  2 2
\put{$\bullet$} at  3 -1  \put{$\bullet$} at  3 2
\put{$\bullet$} at  4 -1  \put{$\bullet$} at  4 2
\put{$\bullet$} at  5 -1  \put{$\bullet$} at  5 2
\put{$\bullet$} at  6 -1  \put{$\bullet$} at  6 2
\put{$\bullet$} at  7 -1  \put{$\bullet$} at  7 2
\plot 1 2  0  -1 /
\plot 2 2  1  -1 /
\plot 4 2  2  -1 /
\plot 5 2  3  -1 /
\plot 6 2  4  -1 /
\endpicture} = r\\
 {\beginpicture
\setcoordinatesystem units <0.4cm,0.2cm>        
\setplotarea x from 0 to 7, y from -1.5 to 2.5
\put{$\bullet$} at  0 -1  \put{$\bullet$} at  0 2
\put{$\bullet$} at  1 -1  \put{$\bullet$} at  1 2
\put{$\bullet$} at  2 -1  \put{$\bullet$} at  2 2
\put{$\bullet$} at  3 -1  \put{$\bullet$} at  3 2
\put{$\bullet$} at  4 -1  \put{$\bullet$} at  4 2
\put{$\bullet$} at  5 -1  \put{$\bullet$} at  5 2
\put{$\bullet$} at  6 -1  \put{$\bullet$} at  6 2
\put{$\bullet$} at  7 -1  \put{$\bullet$} at  7 2
\plot 0 2  0 -1 /
\plot 3 2 1 -1 /
\plot 4 2 6 -1  /
\plot 7 2 7 -1 /
\setquadratic
\plot 1 2 1.5 1 2 2 /
\plot 5 2 5.5 1 6 2 /
\plot 2 -1 3.5 .25 5 -1 /
\plot 3 -1 3.5 -.25 4 -1 /
\endpicture} = t\\
 {\beginpicture
\setcoordinatesystem units <0.4cm,0.2cm>        
\setplotarea x from 0 to 7, y from -1.5 to 2.5
\put{$\bullet$} at  0 -1  \put{$\bullet$} at  0 2
\put{$\bullet$} at  1 -1  \put{$\bullet$} at  1 2
\put{$\bullet$} at  2 -1  \put{$\bullet$} at  2 2
\put{$\bullet$} at  3 -1  \put{$\bullet$} at  3 2
\put{$\bullet$} at  4 -1  \put{$\bullet$} at  4 2
\put{$\bullet$} at  5 -1  \put{$\bullet$} at  5 2
\put{$\bullet$} at  6 -1  \put{$\bullet$} at  6 2
\put{$\bullet$} at  7 -1  \put{$\bullet$} at  7 2
\plot 0 2  0 -1 /
\plot 1 2 2 -1 /
\plot 2 2 3 -1 /
\plot 3 2 4 -1 /
\plot 4 2 5 -1 /
\plot 5 2 6 -1 /
\plot 6 2 7 -1 /
\endpicture} = \ell
\end{array}
\end{equation}
This factorization can be done  in the following way.   To obtain $t$,  first shift the isolated vertices of $d$ to the 
right of the diagram, preserving connections on the other vertices, to produce the diagram $d_1$.  Second, fill in with enough non-nested
horizontal edges so that the top and bottom rows have the same number of horizontal edges to produce the diagram $d_2$. Finally,  finish off with identity edges. This produces the Temperley-Lieb diagram  $t \in \TL_k(x)$
that is associated with $d$. These steps are illustrated  below.
$$
\underbrace{\beginpicture
\setcoordinatesystem units <0.3cm,0.2cm>        
\setplotarea x from 0 to 7, y from -1.5 to 2.5
\put{$\bullet$} at  0 -1  \put{$\bullet$} at  0 2
\put{$\bullet$} at  1 -1  \put{$\bullet$} at  1 2
\put{$\bullet$} at  2 -1  \put{$\bullet$} at  2 2
\put{$\bullet$} at  3 -1  \put{$\bullet$} at  3 2
\put{$\bullet$} at  4 -1  \put{$\bullet$} at  4 2
\put{$\bullet$} at  5 -1  \put{$\bullet$} at  5 2
\put{$\bullet$} at  6 -1  \put{$\bullet$} at  6 2
\put{$\bullet$} at  7 -1  \put{$\bullet$} at  7 2
\plot 1 2  0 -1 /
\plot 5 2 2 -1 /
\plot 6 2 7 -1  /
\setquadratic
\plot 2 2 3 1 4 2 /
\plot 3 -1 4.5 .25 6 -1 /
\plot 4 -1 4.5 -.25 5 -1 /
\endpicture}_{d}
\mapsto
\underbrace{\beginpicture
\setcoordinatesystem units <0.3cm,0.2cm>        
\setplotarea x from 0 to 7, y from -1.5 to 2.5
\put{$\bullet$} at  0 -1  \put{$\bullet$} at  0 2
\put{$\bullet$} at  1 -1  \put{$\bullet$} at  1 2
\put{$\bullet$} at  2 -1  \put{$\bullet$} at  2 2
\put{$\bullet$} at  3 -1  \put{$\bullet$} at  3 2
\put{$\bullet$} at  4 -1  \put{$\bullet$} at  4 2
\put{$\bullet$} at  5 -1  \put{$\bullet$} at  5 2
\put{$\bullet$} at  6 -1  \put{$\bullet$} at  6 2
\put{$\bullet$} at  7 -1  \put{$\bullet$} at  7 2
\plot 0 2  0 -1 /
\plot 3 2 1 -1 /
\plot 4 2 6 -1  /
\setquadratic
\plot 1 2 1.5 1 2 2 /
\plot 2 -1 3.5 .25 5 -1 /
\plot 3 -1 3.5 -.25 4 -1 /
\endpicture}_{d_1}
\mapsto
\underbrace {\beginpicture
\setcoordinatesystem units <0.3cm,0.2cm>        
\setplotarea x from 0 to 7, y from -1.5 to 2.5
\put{$\bullet$} at  0 -1  \put{$\bullet$} at  0 2
\put{$\bullet$} at  1 -1  \put{$\bullet$} at  1 2
\put{$\bullet$} at  2 -1  \put{$\bullet$} at  2 2
\put{$\bullet$} at  3 -1  \put{$\bullet$} at  3 2
\put{$\bullet$} at  4 -1  \put{$\bullet$} at  4 2
\put{$\bullet$} at  5 -1  \put{$\bullet$} at  5 2
\put{$\bullet$} at  6 -1  \put{$\bullet$} at  6 2
\put{$\bullet$} at  7 -1  \put{$\bullet$} at  7 2
\plot 0 2  0 -1 /
\plot 3 2 1 -1 /
\plot 4 2 6 -1  /
\setquadratic
\plot 1 2 1.5 1 2 2 /
\plot 5 2 5.5 1 6 2 /
\plot 2 -1 3.5 .25 5 -1 /
\plot 3 -1 3.5 -.25 4 -1 /
\endpicture}_{d_2}
\mapsto
\underbrace{\beginpicture
\setcoordinatesystem units <0.3cm,0.2cm>        
\setplotarea x from 0 to 7, y from -1.5 to 2.5
\put{$\bullet$} at  0 -1  \put{$\bullet$} at  0 2
\put{$\bullet$} at  1 -1  \put{$\bullet$} at  1 2
\put{$\bullet$} at  2 -1  \put{$\bullet$} at  2 2
\put{$\bullet$} at  3 -1  \put{$\bullet$} at  3 2
\put{$\bullet$} at  4 -1  \put{$\bullet$} at  4 2
\put{$\bullet$} at  5 -1  \put{$\bullet$} at  5 2
\put{$\bullet$} at  6 -1  \put{$\bullet$} at  6 2
\put{$\bullet$} at  7 -1  \put{$\bullet$} at  7 2
\plot 0 2  0 -1 /
\plot 3 2 1 -1 /
\plot 4 2 6 -1  /
\plot 7 2 7 -1 /
\setquadratic
\plot 1 2 1.5 1 2 2 /
\plot 5 2 5.5 1 6 2 /
\plot 2 -1 3.5 .25 5 -1 /
\plot 3 -1 3.5 -.25 4 -1 /
\endpicture}_t.
$$
Now, $r \in \RP_k$ and $\ell \in \LP_k$ are the unique diagrams that move the
edges to their appropriate positions as seen in \eqref{FactorizationExample}.
This algorithm proves that
\begin{equation}
\M_k(x) = \RP_k \TL_k(x)\,  \LP_k.
\end{equation} 
 Thus, we have established the following result.

\begin{prop} The Motzkin algebra $\M_k(x)$ is generated by $\mathbf{1}_k$ and the diagrams $t_i, \ell_i, r_i$ for $1 \le i  < k$.
\end{prop}

\begin{remark} {\rm These diagrams form a natural set of generators which obey nice relations, and 
a presentation for $\M_k(x)$ using these generators is given in \cite[Thm.~4.1]{HLP}.

The elements $\ell_i, r_i \ (1 \leq i < k)$  and  $p_j \ (1 \leq j \leq k)$
generate a subalgebra isomorphic  to  the planar rook
algebra (see \cite{He}, \cite{FHH}), which has irreducible modules with dimensions given by the binomial
coefficients ${k \choose n}$,  for $n=0,1,\dots, k$.   
The subalgebra $\RR_k$  
generated by the elements $r_i$ \ ($1 \leq i < k$)  (hence also the subalgebra $\LL_k$ generated by the $\ell_i$)  has dimension equal to  the $(k+1)$st term  in the  sequence $1,1,1,3,10,31,98, \cdots$ (sequence \#A114487 in \cite{Sl}).    This sequence   
has appeared in the work \cite{STT},   where the $n$th term in the sequence is shown to be the number of lattice paths (Dyck paths) from $(0,0)$ to $(2n,0)$ with
steps $u=(1,1)$ or $d=(1,-1)$, never falling below the $x$-axis and avoiding
the pattern $uudd$.   To keep the paper a reasonable length, we do not
include the computations of the dimensions of $\RP_k$ and $\RR_k$,  as
they are not needed for any results here, and they have appeared in an earlier preprint of ours.}   
\end{remark}

\end{subsection}

 \begin{subsection}{Motzkin paths and 1-factors}  \label{subsec:MotzkinPaths}
 
 A \emph{Motzkin path} of length $k$ is a sequence $p = (a_1,  \ldots, a_k)$ with $a_i \in \{-1, 0,1\}$ such that  $a_1 + \cdots + a_s \ge 0$ for all $1 \le s \le k$.  Define,
  $
\rank\left((a_1,  \ldots, a_k)\right) = a_1 + \cdots + a_k,
 $
 so for example, $$
p = (1,1,1,-1,1,-1,-1,0,1,-1,1,0,1,1,0,1,-1,-1,0,1)
$$
is a Motzkin path of  length 20 and rank 4.  
Let $\mathcal{P}_k$ denote the set of Motzkin paths of length $k$,
and 
\begin{equation} \label{eq:rankrpaths}  \cP_k^r  = \{ p \in \cP_k \mid \rank(p) = r \} \ \ 
\hbox{\rm for} \ \ r =0,1\dots, k. \end{equation}

  It follows from \cite[Exercise 6.38\,(d), p.~238]{St}  that the Motzkin 
number $\Mk$ is the number of  lattice paths from (0,0) to $(k,0)$ with steps $(1,-1), (1,0), (1,1)$, 
never going below the $x$-axis.   These correspond precisely to our Motzkin paths  
in $\mathcal P_k^0$,  as can be readily seen by taking the second coordinates 
of the steps.   
More generally,  the Motzkin paths in $\cP_k^r$  are in bijection with
lattice paths from (0,0) to $(k,r)$ with steps $(1,-1), (1,0), (1,1)$, 
never going below the $x$-axis.  Assume 
$\mathsf {m}_{k,r}$ counts those paths so that $| \cP_k^r | = \mathsf {m}_{k,r}$.  
These paths are also in bijection
with the walks with $k$ steps from $\{-1,0,1\}$ on $\mathbb{N}$ 
starting at 0 and ending at $r$.  

We claim that  
\begin{equation}\label{eq:sumofsqs}   \sum_{r=0}^k \mathsf{m}_{k,r}^2 =\Mtwok. \end{equation} This follows from the simple observation that $\Mtwok$ counts the number of Motzkin paths of length $2k$ and rank $0$,
hence also the number of lattice paths from $(0,0)$ to $(2k,0)$ with steps $(1,0), (1,1), (1,-1)$.
The first $k$ steps in that path form a lattice path from $(0,0)$ to $(k,r)$ for some $r\in \{0,1,\dots, k\}$ and hence correspond to a Motzkin path of length $k$ and rank $r$.  The last
$k$ steps in the path form a lattice path from $(2k,0)$ to $(k,r)$ when it is read in reverse order and with a change of signs,
hence also  correspond to a Motzkin path of length $k$ and rank $r$.   Equation \eqref{eq:sumofsqs}
follows.

For each index $i$ with $a_i = 1$ in the Motzkin path $p = (a_1, \ldots, a_k) \in \mathcal P_k$, let $j$ be the smallest index (if it exists)  such that $i < j \le k$ and  $a_i + a_{i+1} + \cdots + a_j = 0$.   Then $(a_i,a_j) = (1,-1)$ are said to be \emph{paired} in $p$; otherwise $a_i$ is said to be unpaired in $p$.   
For example, in the Motzkin path above, if we connect paired indices by an edge, and we associate unpaired 1s with $\redbull$, we have
$$
p = 
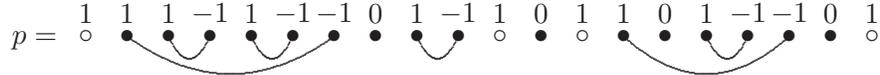
{\beginpicture
 \setcoordinatesystem units <0.55cm,0.2cm>        
\setplotarea x from .5 to 20.5, y from -1 to 2.5
\setquadratic
\plot 3 .5      3.5  -1       4 .5 /
\plot 5 .5      5.5  -1       6 .5 /
\plot 9 .5      9.5  -1       10 .5 /
\plot 16 .5     16.5  -1       17 .5 /
\plot 2 .5      4.5  -2       7 .5 /
\plot 14 .5      16  -2       18 .5 /
\put{$\redbull $} at 1 .5  \put{${1}$} at 1 2
\put{$\bullet$} at 2 .5 \put{${1}$} at 2 2
\put{$\bullet$} at 3 .5 \put{${1}$} at 3 2
\put{$\bullet$} at 4 .5 \put{${-1}$} at 4 2
\put{$\bullet$} at 5 .5 \put{${1}$} at 5 2
\put{$\bullet$} at 6 .5 \put{${-1}$} at 6 2
\put{$\bullet$} at 7 .5 \put{${-1}$} at 7 2
\put{$\bullet$} at 8 .5  \put{${0}$} at 8 2
\put{$\bullet$} at 9 .5 \put{${1}$} at 9 2
\put{$\bullet$} at 10 .5 \put{${-1}$} at 10 2
\put{$\redbull $} at 11 .5 \put{${1}$} at 11 2
\put{$\bullet$} at 12 .5 \put{${0}$} at 12 2
\put{$\redbull$} at 13 .5 \put{${1}$} at 13 2
\put{$\bullet$} at 14 .5 \put{${1}$} at 14 2
\put{$\bullet$} at 15 .5 \put{${0}$} at 15 2
\put{$\bullet$} at 16 .5 \put{${1}$} at 16 2
\put{$\bullet$} at 17 .5 \put{${-1}$} at 17 2
\put{$\bullet$} at 18 .5 \put{${-1}$} at 18 2
\put{$\bullet$} at 19 .5 \put{${0}$} at 19 2
\put{$\redbull$} at 20 .5 \put{${1}$} at 20 2
\endpicture}
$$
Then, since paired vertices cancel one another in the sum,
$$
\rank(p)  = \hbox{ the number of white vertices in $p$.} 
$$

Define a Motzkin \emph{1-factor} to be a diagram on  a single row of $k$ vertices, colored white or black, such that 
\begin{enumerate}
\item black vertices are allowed to be connected  by non-crossing horizontal edges drawn below the diagram, and
\item white vertices are not allowed between two vertices that are connected by an edge.
\end{enumerate}
The pairing procedure above provides a map from Motzkin paths to Motzkin 1-factors.  This process never allows white vertices between paired vertices or edge crossings.   Furthermore, it is easy to find the labeling on a Motzkin 1-factor that produces its path $p$,  so the procedure is reversible, and we have a bijection between paths and 1-factors.  Therefore, in what follows, we  use the terms (Motzkin) path and 1-factor interchangeably.

If $p$ and $q$ are Motzkin paths each of length $k$ and rank $r$, define a Motzkin diagram $d_p^q \in \M_k(x)$ such that
the horizontal edges in the bottom row of $d_p^q$ are from $p$ (but are reflected to be above the vertices), the horizontal edges in the top row of $d_p^q$ are from $q$, and the 
vertical edges in $d_p^q$ connect the $r$ white vertices in $p$ to the $r$ white vertices in $q$ in the one, and only one,  planar way.  For example,  
\begin{equation}\label{eigenfunction}
\begin{array}{rcl}
q &=& { \beginpicture
 \setcoordinatesystem units <0.5cm,0.2cm>        
\setplotarea x from .5 to 17.5, y from -2 to 2.5
\setquadratic
\plot 1 .5      1.5  -1      2 .5 /
\plot 6 .5      6.5  -1      7 .5 /
\plot 11 .5      11.5 -.75       12 .5 /
\plot 16 .5     16.5  -1      17  .5 /
\plot 9 .5 11.5 -1.5 14 .5 / 
\put{$\bullet$} at 1 .5
\put{$\bullet$} at 2 .5
\put{$\redbull$} at 3 .5
\put{$\bullet$} at 4 .5
\put{$\redbull$} at 5 .5
\put{$\bullet$} at 6 .5
\put{$\bullet$} at 7 .5
\put{$\redbull $} at 8 .5
\put{$\bullet$} at 9 .5
\put{$\bullet$} at 10 .5
\put{$\bullet$} at 11 .5
\put{$\bullet$} at 12 .5
\put{$\bullet$} at 13 .5
\put{$\bullet$} at 14 .5
\put{$\redbull$} at 15 .5
\put{$\bullet$} at 16 .5
\put{$\bullet$} at 17 .5
\endpicture}  \\
p &=&  { \beginpicture
 \setcoordinatesystem units <0.5cm,0.2cm>        
\setplotarea x from .5 to 17.5, y from -5 to 3
\setquadratic
\plot 3 .5      4  -1       5 .5 /
\plot 9 .5      9.5  -1      10 .5 /
\plot 8 .5      10  -2       12 .5 /
\plot 6 .5     9.5  -3      13  .5 /
\plot 15 .5      15.5  -1       16 .5 /
\put{$\redbull $} at 1 .5
\put{$\redbull $} at 2 .5
\put{$\bullet$} at 3 .5
\put{$\bullet$} at 4 .5
\put{$\bullet$} at 5 .5
\put{$\bullet$} at 6 .5
\put{$\bullet$} at 7 .5
\put{$\bullet$} at 8 .5
\put{$\bullet$} at 9 .5
\put{$\bullet$} at 10 .5
\put{$\bullet$} at 11 .5
\put{$\bullet$} at 12 .5
\put{$\bullet$} at 13 .5
\put{$\redbull $} at 14 .5
\put{$\bullet$} at 15 .5
\put{$\bullet$} at 16 .5
\put{$\redbull $} at 17 .5
\endpicture} \\
d_p^q &=& 
{ \beginpicture
 \setcoordinatesystem units <0.5cm,0.35cm>        
\setplotarea x from .5 to 17.5, y from -1 to 1
\put{$\bullet$} at 1 2  \put{$\bullet$} at 2 2 \put{$\bullet$} at 3 2  \put{$\bullet$} at 4 2 
\put{$\bullet$} at 5 2  \put{$\bullet$} at 6 2  \put{$\bullet$} at 7 2  \put{$\bullet$} at 8 2 
\put{$\bullet$} at 9 2  \put{$\bullet$} at 10 2  \put{$\bullet$} at 11 2  \put{$\bullet$} at 12 2 
\put{$\bullet$} at 13 2 
\put{$\bullet$} at 14 2 
\put{$\bullet$} at 15 2 
\put{$\bullet$} at 16 2 
\put{$\bullet$} at 17 2 
\put{$\bullet$} at 1 -1 
\put{$\bullet$} at 2 -1
\put{$\bullet$} at 3 -1 
\put{$\bullet$} at 4 -1 
\put{$\bullet$} at 5 -1 
\put{$\bullet$} at 6 -1 
\put{$\bullet$} at 7 -1 
\put{$\bullet$} at 8 -1 
\put{$\bullet$} at 9 -1 
\put{$\bullet$} at 10 -1 
\put{$\bullet$} at 11 -1 
\put{$\bullet$} at 12 -1 
\put{$\bullet$} at 13 -1 
\put{$\bullet$} at 14 -1 
\put{$\bullet$} at 15 -1 
\put{$\bullet$} at 16 -1 
\put{$\bullet$} at 17 -1 
\plot 3 2 1 -1 /
\plot 5 2 2 -1 /
\plot 15 2 17 -1 /
\plot 8 2 14 -1 /
\setquadratic
\plot 1 2 1.5 1.25 2 2 /
\plot 6 2 6.5 1.25 7 2 /
\plot 9 2 11.5 1 14 2 /
\plot 11 2 11.5 1.5 12 2 / 
\plot 16 2 16.5 1.25 17 2 /
\plot 3 -1      4  0       5 -1 /
\plot 9 -1      9.5  -.5     10 -1 /
\plot 8 -1      10  0       12 -1 /
\plot 6 -1     9.5  .5      13  -1 /
\plot 15 -1      15.5  0       16 -1 / 
\endpicture}
\end{array}
\end{equation}
In this way each Motzkin $k$-diagram of rank $r$ is uniquely identified with a pair $(p,q)$ of Motzkin 1-factors of rank $r$, and this can be used to give a second realization of  identity \eqref{eq:sumofsqs}.

In Section \ref{sec:quantumgroup}, we use Motzkin paths to count the multiplicity of irreducible quantum $\fsl_2$-modules in tensor space;  in Section \ref{subsec:decomp} we use Motzkin paths to explicitly decompose tensor space into irreducible  $\mathsf{U}_\qq(\fsl_{\mathsf{2}})$-modules;  and in Section \ref{subsec: motzact}, we use Motzkin paths as a basis for the cell modules of $\M_k(x)$.

\end{subsection}

\begin{subsection}{Basic construction}
\label{Sec:BasicConstruction}  

Assume now that $x \neq 0$.   The  natural embedding of $\M_k(x) \subseteq \M_{k+1}(x)$, 
given by adding to a Motzkin diagram in $\M_k(x)$  a vertical edge connecting the $(k+1)$st vertex in each row, allows one to study 
\begin{equation}\label{eq:tower}
\M_0(x) \subseteq \M_1(x) \subseteq \M_2(x) \subseteq \cdots 
\end{equation}
collectively as a tower of algebras using the ``Jones basic construction" (see \cite{GHJ} or \cite{HR}) and the methods of recollement (see \cite{CPS}).   Define the idempotent $e_k = \frac{1}{x} t_{k-1} \in \M_k(x)$.  Then, for $k \ge 2$, there is an algebra isomorphism 
\begin{equation}\label{eq:basicconstructionisomorphism}
\M_{k-2}(x) \xrightarrow{\cong}  e_k \M_k(x) e_k
\end{equation}
given by sending a diagram $d \in \M_{k-2}(x) \subseteq \M_k(x)$ to $e_k d e_k$.  This is easily confirmed by considering diagram multiplication.  Note that under this map $\mathbf{1}_{k-2} \mapsto e_k$,  which is the unit in $e_k \M_k(x) e_k$.

The algebras $\M_k(x) e_k \M_k(x)$ and $e_k \M_k(x) e_k$ are full centralizers of each other on $\M_k(x) e_k$ under left and right multiplication, respectively.  Therefore, by double centralizer theory, the irreducible modules for $\M_k(x) e_k \M_k(x)$ and $e_k \M_k(x)  e_k$ are in bijection.   Thus, $\M_k(x) e_k \M_k(x)$ and $\M_{k-2}(x)$ have the same irreducible modules, and they are labeled by the integers $0, 1, 2, \ldots, k-2$,  as we shall see in Section 4. The ideal $\M_k(x) e_k \M_k(x)$ is referred to as the ``basic construction" for $\M_k(x)$ determined by the idempotent $e_k$.  

It is easy to see by multiplying diagrams that $\M_k(x) e_k \M_k(x) = \J_{k-2}$, the ideal defined in Section \ref{subsec:motzkinalgebra} as the span of Motzkin diagrams of rank $\le$ $k-2$.  Then $\M_k(x)$ is spanned modulo $\mathsf{J}_{k-2}$ by the unique diagram of rank $k$ (the identity element $\mathbf{1}_k$) and the diagrams $d_q^p$ of rank $k-1$ where $p,q \in \cP_{k}^{k-1}$.      Now   
\begin{equation}\label{eq:mat-k-1}d_q^p  d_t^s \equiv  \delta_{q,s}  d_t^p \mod \J_{k-2},\end{equation}   
as the rank goes down when $q \neq s$; 
from this we see that the cosets $\{d_q^p + \mathsf{J}_{k-2} \mid p,q \in  \cP_{k}^{k-1}\}$
form a basis of a $k \times k$ matrix algebra.    Thus, when $\KK$ is a field,  \begin{equation}\label{eq:basicconstructionsemisimple}
\M_k(x)/ \M_k(x) e_k \M_k(x) = \M_k(x)/\J_{k-2} \cong \KK \mathbf{1}_k \oplus \mathsf{Mat}_k(\KK) \quad \text{is semisimple.} 
\end{equation} 

The basic construction 
 tells us that when $x \neq 0$, the irreducible modules of 
$\M_k(x)$ are in bijection with the union of the irreducible modules for  $\KK \mathbf{1}_k, \mathsf{Mat}_k(\KK)$, and $\M_{k-2}(x)$.  This fact is explicitly realized when we construct the cell modules for $\M_k(x)$ in Section \ref{sec:Cellularity}.

\end{subsection}

\end{section}

\begin{section}{Schur-Weyl duality}

\begin{subsection}{Quantum $\fsl_2$}\label{sec:quantumgroup}

 {\it We assume throughout this section that $\mathbb{K}$ is a field, and $\qq \in \mathbb{K}$ is nonzero and not a root of unity}. 
The Lie algebra $\fsl_2$ of  $2 \times 2$ matrices
over a field $\KK$ has standard basis elements  
$$
e = \begin{pmatrix} 0 & 1 \\ 0 & 0 \end{pmatrix}, \qquad
f = \begin{pmatrix} 0 & 0 \\ 1 & 0 \end{pmatrix}, \qquad
h_1 = \begin{pmatrix} 1 & 0 \\ 0 & 0 \end{pmatrix}, \qquad 
h_2 = \begin{pmatrix} 0 & 0 \\ 0 & 1 \end{pmatrix},
$$ 
which satisfy the relations  
\begin{equation} \label{eq:gl2def} [h_1,h_2] = 0, \ [h_1,e] = e, \ [h_1,f] = -f, \ [h_2,e] = -e,\ [h_2,f] = f, \  \hbox{\rm and} \ 
[e,f] = h,\end{equation}  where $[a,b] = ab-ba$ and $h = h_1-h_2$. 
Its universal enveloping algebra $\mathsf{U}(\fsl_{\mathsf{2}})$ is the unital associative $\KK$-algebra  generated by $e,f,h_1,h_2$ with the defining relations  in \eqref{eq:gl2def}.
The quantum enveloping algebra $\uqsl$ is the unital associative algebra generated by $E,F,K_i^{\pm 1}$, $i=1,2$,  subject to the relations

\begin{equation}
\begin{array}{ll}
K_1K_2 = K_2 K_1  &  K_iK_i^{-1} = K_i^{-1} K_i = 1,  \ \ i=1,2,  \\
K_1EK_1^{-1} = \qq E,  & 
K_2 E K_2^{-1} = \qq^{-1} E \\  K_1FK_1^{-1} = \qq^{-1} F, &  K_2 F K_2^{-1} = \qq F \\
EF-FE = \displaystyle{\frac{K - K^{-1}}{\qq - \qq^{-1}}}, &  \hbox{\rm where} \ \ K = K_1K_2^{-1}.
\end{array}
\end{equation} 
In the classical limit $\qq \to 1$,  the algebra $\uqsl$ specializes to $\mathsf{U}(\fsl_{\mathsf{2}})$. Furthermore,  $\uqsl$ is a non-cocommutative Hopf algebra with coproduct given by
\begin{equation}\label{coproduct}
\begin{array}{l}
\Delta(E) = E \otimes K + 1 \otimes E, \\
\Delta(F) = F \otimes 1 + K^{-1} \otimes F,
\end{array} \qquad
\Delta(K_i) = K_i \otimes K_i,  \ \ i=1,2,
\end{equation}
and counit $\mathsf{u}$ given by 
\begin{equation}
\mathsf{u}(E) =\mathsf{u}(F) = 0, \qquad \mathsf{u}(K_i) = 1, \  \ \ i = 1,2.
\end{equation}
The subalgebra of $\uqsl$ generated by $E,F,K^{\pm 1}$ is the
quantum enveloping algebra $\mathsf{U}_\qq(\mathfrak{sl}_{\mathsf{2}})$ corresponding
the Lie algebra $\mathfrak{sl}_2$ of traceless matrices in $\fsl_2$. 

For each $r=0,1,\dots$, the algebra $\uqsl$ has an irreducible module $\VV(r)$  of dimension $r+1$ generated by a highest weight vector $v_r$ such that $Ev_r = 0$, 
$K_1 v_r = \qq^{r}v_r$, and $K_2v_r =  v_r$.    In particular, 
the 1-dimensional irreducible $\uqsl$-module   $\VV(0) =  \mathsf{span}_\KK\{ v_0 \}$ has  $\uqsl$-action given by the counit,
\begin{equation}
E v_0 = 0, \qquad F v_0 = 0, \qquad K_i v_0 = v_0,   \ \ \ i =1,2,
\end{equation}
and the 2-dimensional irreducible $\uqsl$-module $\VV(1)$ has action specified by 
\begin{equation} \label{QuantumGroupAction}
\begin{array}{cccc}
 E v_1 = 0, & F v_1 = v_{-1}, & K_1  v_1 = \qq  v_1,  & K_2 v_1 = v_1,   \\
E v_{-1} = v_1, & F v_{-1} = 0, & K_1 v_{-1} = v_{-1}, & K_2 v_{-1} = 
\qq v_{-1},
\end{array}
\end{equation}
so that the matrices of the representing transformations on $\VV(1)$ relative to the
basis $\{v_1, v_{-1}\}$  are given by 
\begin{equation}
E \rightarrow  \begin{pmatrix} 0 & 1 \\ 0 & 0 \end{pmatrix}, \quad
F \rightarrow  \begin{pmatrix} 0 & 0 \\ 1 & 0 \end{pmatrix} \quad
K_1 \rightarrow  \begin{pmatrix} \qq & 0 \\ 0 & 1 \end{pmatrix} \quad
K_2 \rightarrow  \begin{pmatrix} 1 & 0 \\ 0 & \qq \end{pmatrix}.
\end{equation}

Set 
\begin{equation}
\VV = \VV(0) \oplus \VV(1) = \mathsf{span}_\KK \left\{ v_{-1}, v_0, v_{1} \right\}. 
\end{equation}
Then the $k$-fold tensor product space $\VV^{\otimes k}$ has dimension $3^k$ and  a basis of simple tensors,
\begin{equation}
\VV^{\otimes k} =  \mathsf{span}_\KK\left\{\ v_{i_1} \otimes v_{i_2} \otimes \cdots \otimes v_{i_k} \ \big\vert \  {i_j} \in \{ -1, 0, 1\}
\  \hbox{\rm for all} \ j \right\}.
\end{equation} The coproduct \eqref{coproduct} affords an action of $\uqsl$ on $\VV^{\otimes k}$.

Under the above assumptions on
$\KK$ and $\qq$,  finite-dimensional modules for $\uqsl$ are completely
reducible.  In particular, tensor products
decompose according to the Clebsch-Gordan formulas \begin{align}
\VV(r) \otimes \VV(0)  &=  \VV(r), \\
\VV(r) \otimes \VV(1)  &=  \VV(r-1) \oplus \VV(r+1) \label{eq:CG2},
\end{align}
and thus,
\begin{equation}\label{eq:TensorRule}
\VV(r) \otimes \VV  =  \VV(r-1) \oplus \VV(r) \oplus \VV(r+1), 
\end{equation}
where $\VV(-1) = 0$. 

 Iterating \eqref{eq:TensorRule} creates the tower shown in Figure 1, which displays the decomposition of $\VV^{\otimes k}$ into irreducible $\uqsl$-modules (where $(r)$
stands for $\VV(r)$).  By induction,  the multiplicity of $\VV(r)$ in $\VV^{\otimes k}$ is equal to the number of paths of length $k$ from the top of the diagram to $(r)$.  These multiplicities are shown as the subscripts.   We record the steps in such a path,  by writing  $-1$ for a northeast-to-southwest edge, 0 for a vertical edge, and 1 for a northwest-to-southeast edge.  
Thus a path from $(0)$ at the top to $(r)$ at the $k$th level  corresponds to a walk on
$\NN$ from $0$ to $r$ with $k$ steps of $-1,0$, or $1$. These are exactly the Motzkin paths of length $k$ and rank $r$ from  
Section \ref{subsec:MotzkinPaths}.  Then  
\begin{equation}
\VV^{\otimes k} \cong \bigoplus_{r = 0}^{k}  \mathsf{m}_{k,r}  \VV(r), \qquad \hbox{ as a  $\uqsl$-module,}
\end{equation}

\begin{figure}
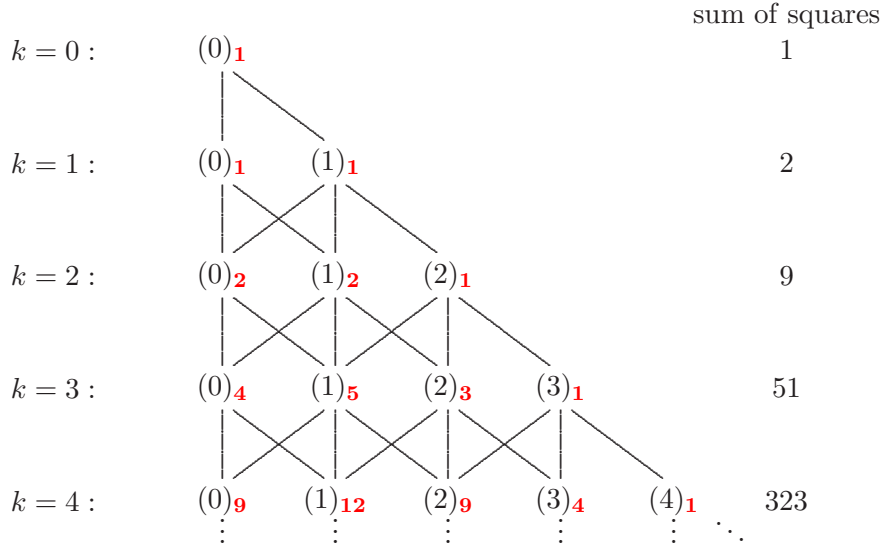
\label{MotzkinBratteli}
$$
 {\beginpicture
\setcoordinatesystem units <0.15cm,0.15cm>        
\setplotarea x from 0 to 40, y from 0 to 40
\put{$k=0:$} at -15 40   \put{$k=1:$} at -15 30  \put{$k=2:$} at -15 20
\put{$k=3:$} at -15 10  \put{$k=4:$} at -15 0
\put{$(0)_{\color{red}{\bf 9}\color{black}}$} at 0 0 
\put{$(1)_{\color{red}{\bf 12}\color{black}}$} at 10 0 
\put{$(2)_{\color{red}{\bf 9}\color{black}}$} at 20 0
\put{$(3)_{\color{red}{\bf 4}\color{black}}$} at 30 0
\put{$(4)_{\color{red}{\bf 1}\color{black}}$} at 40 0
\put{$(0)_{\color{red}{\bf 4}\color{black}}$} at 0 10 
\put{$(1)_{\color{red}{\bf 5}\color{black}}$} at 10 10 
\put{$(2)_{\color{red}{\bf 3}\color{black}}$} at 20 10
\put{$(3)_{\color{red}{\bf 1}\color{black}}$} at 30 10
\put{$(0)_{\color{red}{\bf 2}\color{black}}$} at 0 20 
\put{$(1)_{\color{red}{\bf 2}\color{black}}$} at 10 20 
\put{$(2)_{\color{red}{\bf 1}\color{black}}$} at 20 20
\put{$(0)_{\color{red}{\bf 1}\color{black}}$} at 0 30 
\put{$(1)_{\color{red}{\bf 1}\color{black}}$} at 10 30 
\put{$(0)_{\color{red}{\bf 1}\color{black}}$} at 0 40 
\plot 0 38  0 32 /
\plot 1 38  9 32 /
\plot 0 28  0 22 /
\plot 1 28  9 22 /
\plot 9 28 1 22 / 
\plot 11 28 19 22 /
\plot 10 22 10 28 /
\plot 0 18  0 12 /
\plot 1 18  9 12 /
\plot 9 18 1 12 / 
\plot 11 18 19 12 /
\plot 10 12 10 18 /
\plot 19 18 11 12 /
\plot 20 18 20 12 /
\plot 21 18 29 12 /
\plot 0 8  0 2 /
\plot 1 8  9 2 /
\plot 9 8 1 2 / 
\plot 11 8 19 2 /
\plot 10 2 10 8 /
\plot 19 8 11 2 /
\plot 20 8 20 2 /
\plot 21 8 29 2 /
\plot 29 8 21 2 /
\plot 30 8 30 2 /
\plot 31 8 39 2 /
\put{1} at 50 40
\put{2} at 50 30 
\put{9} at 50 20
\put{51} at 50 10
\put{323} at 50 0
\put{$\vdots$} at 0 -2
\put{$\vdots$} at 10 -2
\put{$\vdots$} at 20 -2
\put{$\vdots$} at 30 -2
\put{$\vdots$} at 40 -2
\put{$\ddots$} at 45 -2
\put{sum of squares} at 50 43 
\endpicture}
$$
\caption{\small The Bratteli diagram for the decomposition of $\VV^{\otimes k}$ and for the representations of $\M_k(x)$.}  
\end{figure}
\noindent where  as before $\mathsf{m}_{k,r}$ equals  the number of  Motzkin paths of length $k$ and rank $r$.  
Since the Motzkin number $\mathcal M_k$ counts the number of walks on $\NN$  with $k$ steps of $-1,0,$ or $1$ beginning and ending at 0,    
$${\mathcal M}_k = \mathsf{m}_{k,0} = \hbox{\rm the multiplicity of $\VV(0)$ in $\VV^{\otimes k}$ }.$$
Moreover, this leads to a third realization of \eqref{eq:sumofsqs},
${\mathcal M}_{2k} = \sum_{r = 0}^k \mathsf{m}_{k,r}^2,$  
because walks of length $2k$ from (0) at the top to (0)
at level $2k$ exactly correspond, by reflection about level $k$, with all walks of length $2k$ from the top of the diagram to level $k$ and then back up to the top.

\end{subsection}

\begin{subsection}{The centralizer algebra of the $\uqsl$ tensor action}

Let $\M_k = \End_{\uqsl}(\VV^{\otimes k})$ be the centralizer of $\uqsl$ acting on $\VV^{\otimes k}$, so that  
\begin{equation}
\M_k = \left\{\  \phi \in \End(\VV^{\otimes k}) \  \vert \   \phi (y w) = y\phi(w) \hbox{ for all $y \in \uqsl$, $w \in \VV^{\otimes k}$} \ \right\}.
\end{equation}
Then by the classical double-centralizer theory (see for example \cite[Secs. 3B and 68]{CR}), we know the
following:
\begin{itemize}
\item $\M_k$ is a semisimple associative $\KK$-algebra with irreducible representations labeled by $0, 1, \ldots, k$. We let
$
\left\{\  \M_k^{(r)}  \ \vert \ 0 \le r \le k \right\}
$
denote the set of irreducible $\M_k$-modules.

\item $\dim(\M_k^{(r)}) = \mathsf{m}_{k,r}$.
\item We can naturally embed the algebras  $\M_0 \subseteq \M_1 \subseteq \M_2 \cdots $,  and the edges from level $k+1$
to level $k$ in Figure 1 represent the restriction rule for $\M_k \subseteq \M_{k+1}$.  Therefore, Figure 1 gives
the Bratteli diagram for the tower of semisimple algebras $\M_k$. 
\item The tensor space $\VV^{\otimes k}$ decomposes as \begin{equation}
\begin{array}{rll}
\VV^{\otimes k} & \cong \displaystyle{\bigoplus_{r = 0}^k}\,\, \mathsf{m}_{k,r} \VV(r) & \hbox{ as a $\uqsl$-module}, \\
& \cong \displaystyle{\bigoplus_{r = 0}^k}\,  (r+1) \M_k^{(r)} & \hbox{ as an $\M_k$-module}, \\
& \cong \displaystyle{\bigoplus_{r = 0}^k}\left(\VV(r) \otimes \M_k^{(r)} \right) & \hbox{ as a $(\uqsl,\M_k)$-bimodule}. \\
\end{array}
\end{equation}
\noindent Note that it follows from these expressions that 
\begin{equation} \label{eq:bimoduleidentity}
3^k = \sum_{r = 0}^k (r+1) \mathsf{m}_{k,r}.
\end{equation}

\item By general Wedderburn theory, the dimension of $\M_k$ is the sum of the squares of the dimensions of its irreducible modules,  
\begin{equation}\label{eq:evid}
\dim(\M_k) = \sum_{r = 0}^k \mathsf{m}_{k,r}^2 = {\mathcal M}_{2k}.
\end{equation}
\end{itemize}

\end{subsection}

\begin{subsection} {The centralizer of the Temperley-Lieb action and a formula
for Motzkin numbers}  

The centralizer algebra of $\uqsl$ acting on the $k$th tensor power
$\VV(1)^{\otimes k}$ of the $2$-dimensional irreducible module  $\VV(1)$  
is known to be the Temperley-Lieb algebra $\TL_k(\qq+\qq^{-1})$.    
The irreducible modules for $\TL_k(\qq+\qq^{-1})$ are labeled by
the integers $k-2\ell$, $\ell=0,1, \dots, \lfloor k/2 \rfloor$ (see for example \cite{GHJ} or \cite{W}), and if
$\T_k^{(k-2\ell)}$ is the corresponding irreducible $\TL_k(\qq+\qq^{-1})$-module, then
\begin{equation}\label{eq:tempdims} 
\dim\big(\T_k^{(k-2\ell)}\big) = \left \{ {{k}\atop {\ell}} \right\}  :
= {k \choose \ell}-{k \choose \ell-1}.
\end{equation}
\begin{remark} {\rm The notation $\left \{ {{k}\atop {\ell}} \right\}$   has sometimes been used for Stirling numbers
of the second kind.  In this paper, it will always mean the difference of binomial coefficients as in \eqref{eq:tempdims}.
It has been customary to use $\{\ \}$ to denote dimensions of irreducible $\TL_k(x)$-modules (see [GHJ], for example),
and we follow that practice here. }
\end{remark}

Now  
\begin{equation} \VV(1)^{\otimes k} = \bigoplus_{\ell=0}^{ \lfloor k/2 \rfloor} \VV(k-2\ell)\otimes \T_k^{(k-2\ell)} \end{equation}
as an  $(\uqsl,\TL_k(\qq+\qq^{-1}))$-bimodule.  
Then using the fact that $\VV(0)$ is a trivial 1-dimensional $\uqsl$-module
isomorphic to the field $\KK$,   we have for $\VV = \VV(0) \oplus \VV(1)$,

\begin{eqnarray}\label{eq:TL}
\VV^{\otimes k} &=& \bigoplus_{n=0}^k  {k \choose n} \VV(1)^{\otimes n}  
= \bigoplus_{n=0}^k  {k \choose n} \left( \bigoplus_{\ell=0}^{\lfloor n/2 \rfloor}\VV(n-2\ell) \otimes \T_n^{(n-2\ell)}\right). 
\end{eqnarray}
Let us fix $r$,  and as above, let $\mathsf{m}_{k,r}$ denote the multiplicity of $\VV(r)$
in the decomposition of  $\VV^{\otimes k}$ as a $\uqsl$-module (which is also the  
dimension of the irreducible module $\M_k^{(r)}$ for $\M_k$).  Then by
examining the above equation for when $r = n-2\ell$,  we obtain
 \begin{eqnarray}  \label{eq:MT}
 \mathsf{m}_{k,r}&=& \sum_{\ell=0}^{\lfloor (k-r)/2 \rfloor}  {k \choose {r+2\ell}} \dim \left(\T_{r+2\ell}^{(r)}\right) 
= \sum_{\ell=0}^{\lfloor (k-r)/2 \rfloor}   {k \choose {r+2\ell}} \left \{ {{r+2\ell}\atop {\ell}} \right\}.
\end{eqnarray}

\begin{example}  {\rm When $k = 4$,  this formula says that  
$\mathsf{ m}_{4,0} = 1 \cdot 1 + 6 \cdot 1 + 1 \cdot 2 = 9$, 
$ \mathsf{m}_{4,1} = 4 \cdot 1 + 4 \cdot 2= 12,$
$ \mathsf{m}_{4,2} = 6 \cdot 1 + 1 \cdot 3= 9,$
$ \mathsf{m}_{4,3} = 4 \cdot 1 = 4,$ and
$ \mathsf{m}_{4,4} = 1 \cdot 1= 1$
(compare with line $k=4$ in Figure 1). Thus  $\dim(\M_4) = \sum_{r=0}^4 \mathsf{m}_{k,r}^2= 81 + 144 +81 + 16 + 1  = 323 = {\mathcal M}_8$,  as expected.} 
\end{example}

 \end{subsection}

\begin{subsection}{The action of the Motzkin algebra on tensor space}

As the notation and equation \eqref{eq:evid} suggest,   $\M_k$ is
a Motzkin algebra  $\M_k(x)$ for
a specific choice of the parameter $x$.  In this section, we show that when $x =1 - \qq - \qq^{-1}$ there is an action of $\M_k(1 - \qq - \qq^{-1})$ on the tensor space $\VV^{\otimes k}$ such that $\M_k(1 - \qq - \qq^{-1}) \cong \M_k$. 

 Define nondegenerate bilinear forms  $\langle \cdot , \cdot\ranglet$ and $\langle\cdot, \cdot\rangleb$ on $\VV$ by decreeing
\begin{equation}\label{eq:biform}
\begin{array}{lll}
\langle v_{-1}, v_1 \ranglet = \qq^{-1/2}, \quad &  \langle v_{0}, v_0 \ranglet  = 1, \quad&  \langle v_{1}, v_{-1} \ranglet = -\qq^{1/2},  \\
\langle v_{-1}, v_1 \rangleb = -\qq^{-1/2}, \quad &  \langle v_{0}, v_0 \rangleb  = 1, \quad&  \langle v_{1}, v_{-1} \rangleb = \qq^{1/2},  \\
\end{array}
\end{equation}
and
$$
\langle v_{i}, v_j \ranglet =  \langle v_{i}, v_j \rangleb = 0, \quad \text{ for all other $i, j$.}
$$
Notice that
\begin{equation}\label{ParameterSum}
\begin{array}{lll}
\sum_{i,j} \langle v_{i}, v_j \ranglet \langle v_{i}, v_j \rangleb & = &  \langle v_{-1}, v_{1} \ranglet  \langle v_{-1}, v_{1} \rangleb
+ \langle v_{0}, v_{0} \ranglet \langle v_{0}, v_{0} \rangleb  \\
& & \hskip.2truein  +  \langle v_{1}, v_{-1} \ranglet \langle v_{1}, v_{-1} \rangleb \\
& = & - \qq^{-1} + 1 - \qq.
\end{array}
\end{equation}

For $d$ in the set $\mathfrak{ M}_{k}$ of all Motzkin $k$-diagrams, we define an action of $d$ on the basis of simple tensors
in $\VV^{\otimes k}$  by
\begin{equation}\label{ActionOnTensorSpace}
d (v_{i_1} \otimes \cdots \otimes v_{i_k}) = \sum_{j_1, \ldots, j_k}  (d)_{i_1, \ldots, i_k}^{j_1, \ldots, j_k}\  v_{j_1} \otimes \cdots \otimes v_{j_k},
\end{equation}
where $(d)_{i_1, \ldots, i_k}^{j_1, \ldots, j_k}$ is computed by labeling the vertices in the bottom row of $d$ with $v_{i_1}, \ldots, v_{i_k}$ and the vertices in the top row of $d$ with $v_{j_1}, \ldots, v_{j_k}$.  Then
$$
(d)_{i_1, \ldots, i_k}^{j_1, \ldots, j_k} = \prod_{\varepsilon \in d} (\varepsilon)_{i_1, \ldots, i_k}^{j_1, \ldots, j_k},
$$
where the product is over the weights of all connected components $\varepsilon$ (edges and isolated vertices) in the diagram $d$, where by the weight of $\varepsilon$ we mean  
$$
(\varepsilon)_{i_1, \ldots, i_k}^{j_1, \ldots, j_k} = \begin{cases}
\delta_{a,0}, & \text{ if $\varepsilon$ is an isolated vertex labeled by $v_a$,} \\
\delta_{a,b}, &  \text{ if $\varepsilon$ is a  vertical edge connecting $v_a$ and $v_b$,} \\
\langle v_a, v_b \ranglet, &  \text{ if $\varepsilon$ is a horizontal edge in the top row of $d$} \\
&  \text{\quad connecting $v_a$ (left) and $v_b$ (right).} \\
\langle v_a, v_b \rangleb &  \text{ if $\varepsilon$ is a horizontal edge in the bottom row of $d$} \\
&  \text{\quad connecting $v_a$ (left) and $v_b$ (right),} \\
\end{cases}
$$
where $\delta_{a,b}$ is the Kronecker delta.
 For  example, for this labeled diagram
$$
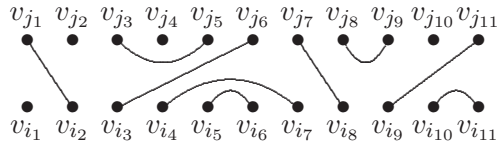
{\beginpicture
\setcoordinatesystem units <0.6cm,0.3cm>        
\setplotarea x from 0 to 8, y from -2 to 3
\put{$\bullet$} at  0 -1  \put{$\bullet$} at  0 2
\put{$\bullet$} at  1 -1  \put{$\bullet$} at  1 2
\put{$\bullet$} at  2 -1  \put{$\bullet$} at  2 2
\put{$\bullet$} at  3 -1  \put{$\bullet$} at  3 2
\put{$\bullet$} at  4 -1  \put{$\bullet$} at  4 2
\put{$\bullet$} at  5 -1  \put{$\bullet$} at  5 2
\put{$\bullet$} at  6 -1  \put{$\bullet$} at  6 2
\put{$\bullet$} at  7 -1  \put{$\bullet$} at  7 2
\put{$\bullet$} at  8 -1  \put{$\bullet$} at  8 2
\put{$\bullet$} at  9 -1  \put{$\bullet$} at  9 2
\put{$\bullet$} at  10 -1  \put{$\bullet$} at  10 2
\put{$v_{j_1}$} at 0 3  
\put{$v_{j_2}$} at 1 3 
\put{$v_{j_3}$} at 2 3  
\put{$v_{j_4}$} at 3 3  
\put{$v_{j_5}$} at 4 3  
\put{$v_{j_6}$} at 5 3  
\put{$v_{j_7}$} at 6 3  
\put{$v_{j_8}$} at 7 3
\put{$v_{j_9}$} at 8 3
\put{$v_{j_{10}}$} at 9 3
\put{$v_{j_{11}}$} at 10 3
\put{$v_{i_1}$} at 0 -2  
\put{$v_{i_2}$} at 1 -2 
\put{$v_{i_3}$} at 2 -2  
\put{$v_{i_4}$} at 3 -2  
\put{$v_{i_5}$} at 4 -2  
\put{$v_{i_6}$} at 5 -2  
\put{$v_{i_7}$} at 6 -2  
\put{$v_{i_8}$} at 7 -2
\put{$v_{i_9}$} at 8 -2
\put{$v_{i_{10}}$} at 9 -2
\put{$v_{i_{11}}$} at 10 -2

\plot 10 2 8 -1 /
\plot 0 2  1 -1 /
\plot 5 2 2 -1 /
\plot 6 2 7 -1  /
\setquadratic
\plot 2 2 3 1 4 2 /
\plot 3 -1 4.5 .25 6 -1 /
\plot 4 -1 4.5 -.25 5 -1 /
\plot 9 -1 9.5 -.25 10 -1 /
\plot 7 2 7.5 1 8 2 /
\endpicture}
$$
we have 
\begin{align*}
(d)_{i_1, \ldots, i_k}^{j_1, \ldots, j_k}\ =& \ \langle v_{j_3},v_{j_5} \ranglet \langle v_{j_8},v_{j_9} \ranglet  
\langle v_{i_4},v_{i_7} \rangleb \langle v_{i_5},v_{i_6} \rangleb \langle v_{i_{10}}, v_{i_{11}} \rangleb \  \times  \\
&  \qquad \qquad  \delta_{j_1,i_2} \delta_{j_6,i_3} \delta_{j_7,i_8} \delta_{j_{11}, i_9} 
\delta_{j_2,0} \delta_{j_4,0} \delta_{j_{10},0} \delta_{i_1,0}.
\end{align*} 
The representing transformation of $d$ is obtained by extending this action linearly to all of $\VV^{\otimes k}$. 

Let $\mathfrak{ t,l,r} \in \mathfrak{M}_2$ and $\mathsf{id} \in \mathfrak{M}_1$ be the diagrams given by
$$
\mathfrak{t} = 
{\beginpicture
\setcoordinatesystem units <0.5cm,0.2cm>        
\setplotarea x from 0.5 to 2.5, y from -1 to 2
\put{$\bullet$} at  1 -1  \put{$\bullet$} at  1 2
\put{$\bullet$} at  2 -1  \put{$\bullet$} at  2 2
\setquadratic
\plot 1 2 1.5 1 2 2 /
\plot 1 -1 1.5 0 2 -1 /
\endpicture}, \qquad
\mathfrak{l}= 
{\beginpicture
\setcoordinatesystem units <0.5cm,0.2cm>        
\setplotarea x from 0.5 to 2.5, y from -1 to 2
\put{$\bullet$} at  1 -1  \put{$\bullet$} at  1 2
\put{$\bullet$} at  2 -1  \put{$\bullet$} at  2 2
\plot 1 2  2 -1 /
\endpicture}, \qquad
\mathfrak{r}= 
{\beginpicture
\setcoordinatesystem units <0.5cm,0.2cm>        
\setplotarea x from 0.5 to 2.5, y from -1 to 2
\put{$\bullet$} at  1 -1  \put{$\bullet$} at  1 2
\put{$\bullet$} at  2 -1  \put{$\bullet$} at  2 2
\plot 2 2  1 -1 /
\endpicture}, \qquad
\mathsf{id} = 
{\beginpicture
\setcoordinatesystem units <0.5cm,0.2cm>        
\setplotarea x from 0.5 to 1.5, y from -1 to 2
\put{$\bullet$} at  1 -1  \put{$\bullet$} at  1 2
\plot 1 2  1 -1 /
\endpicture}.
$$
Then under the action defined in \eqref{ActionOnTensorSpace}, $\mathsf{id}\,v_i = v_i$ for all $i$, so  
$\mathsf{id}$  is represented by the identity map
$\id_\VV$ on $\VV$;  the representing transformations $R$
and $L$ of $\mathfrak{r}$ and $\mathfrak{l}$ on $\VV^{\otimes 2}$  move simple tensors right and left,
$$
R (v_i \otimes v_j)  = \delta_{j,0}\ v_0 \otimes v_i, \qquad
L (v_i \otimes v_j)  =  \delta_{i,0}\ v_j \otimes v_0;
$$
and  the representing transformation $T$ of $\mathfrak{t}$ on $\VV^{\otimes 2}$  acts as a ``contraction" map,
\begin{align*}
T (v_i \otimes v_j) 
& =   \langle v_i, v_j \rangleb \big(   \langle v_{-1}, v_{1} \ranglet ( v_{-1} \otimes v_{1})  \\
& \hspace{1 truein} +  \langle v_0, v_0 \ranglet (v_0 \otimes v_0)  
+ \langle v_{1}, v_{-1} \ranglet (v_{1} \otimes v_{-1}) \big),  \\
& =   \langle v_i, v_j \rangleb  \left( \qq^{-1/2} v_{-1} \otimes v_1 + v_0 \otimes v_0   -\qq^{1/2} v_{1} \otimes v_{-1}     \right). 
\end{align*}
Using these maps, we define the representing transformations of the generators $t_i$, $\ell_i$, $r_i,  (1 \leq i < k)$ acting on $\VV^{\otimes k}$ to be
 \begin{equation}\label{eq:TLRact}
\begin{array}{rclll}
T_i = \id_\VV^{\otimes i-1} \otimes T \otimes \id_\VV^{\otimes k - (i+1)}, \qquad &1 \le i < k, \\ \\
L_i  = \id_\VV^{\otimes i-1} \otimes L \otimes \id_\VV^{\otimes k - (i+1)},\qquad &1 \le i < k, \\ \\
R_i = \id_\VV^{\otimes i-1} \otimes R \otimes \id_\VV^{\otimes k - (i+1)},\qquad &1 \le i < k.
\end{array}
\end{equation}
Since $p_1 = \ell_1 r_1$, $p_i = \ell_i r_i = r_i \ell_i$ for $1 < i  < k$, and $p_k = r_k \ell_k$, we see that the representing transformation $P_i$ of $p_i$ 
is the projection onto $\VV(0)$ in the $i$th tensor position.

Note that in order for this action to satisfy the relation $T_i^2 = x T_i$, the parameter for the Motzkin algebra must be  $1-\qq-\qq^{-1}$, since
\begin{align*}
T^2 (v_a \otimes v_b) & = \langle v_a, v_b \rangleb  \sum_{i, j} \langle v_i, v_j \ranglet T (v_i \otimes v_j) \\
& = \langle v_a, v_b \rangleb  \sum_{i, j} \langle v_i, v_j \ranglet \langle v_i, v_j \rangleb   \sum_{\ell,m} \langle v_\ell, v_m \ranglet  (v_\ell\otimes v_m) \\
& =
 \left(\sum_{i, j}  \langle v_i, v_j \ranglet \langle v_i, v_j \rangleb \right)   \langle v_a, v_b \rangleb    \sum_{\ell,m} \langle v_\ell, v_m \ranglet  (v_\ell\otimes v_m) \\
& =  ( 1-\qq-\qq^{-1})T(v_a\otimes v_b),
\end{align*} 
by \eqref{ParameterSum}. 

Set 
\begin{equation}\label{eq:qdef}  \zq = 1 - \qq- \qq^{-1}. \end{equation}  
For example,   $\zeta_{1} = -1$  and $\zeta_{-1} = 3$.
  
\begin{prop} There is a representation  $\pi_k: \Motz \to \End(\VV^{\otimes k})$ such that 
$\pi_k(t_i) = T_i,\, \pi_k(\ell_i) = L_i,$ and $\pi_k(r_i) = R_i$ for $1 \leq i < k$.  \end{prop}

\begin{proof}  For diagrams $d_1, d_2 \in \mathfrak{M}_{k}$ it suffices to show that
$$
(d_1 d_2)_{i_1, \ldots, i_k}^{j_1, \ldots, j_k} = \sum_{n_1, \ldots, n_k} 
(d_1)^{j_1, \ldots, j_k}_{n_1, \ldots, n_k}
(d_2)^{n_1, \ldots,n_k}_{i_1, \ldots,i_k}.
$$
We analyze the edges of $d_1 d_2$ case by case.   Let $v_{-1}^\ast = v_1, v_0^\ast = v_0,$ and $v_1^\ast = v_{-1}$.

\textbf{Case 1}:  \ Vertical edges in $d_1 d_2$.  

These arise from a vertical edge in $d_1$ connected to a vertical edge in 
$d_2$ via a chain of an  even number (possibly 0) of horizontal edges in the middle of $d_1 d_2$.   See the example below.  The only way to achieve a nonzero matrix entry for $d_1 d_2$ is to have the labeling follow the pattern $v_a, v_a, v_a^\ast, v_a, v_a^\ast, \ldots, v_a^\ast, v_a, v_a$ as illustrated in the example below. Thus,  the horizontal edge in the product $d_1 d_2$ will require that the top vertex have the same subscript as the bottom vertex.
$$
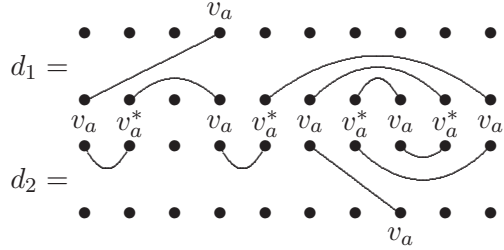
{\beginpicture
\setcoordinatesystem units <0.6cm,0.3cm>        
\setplotarea x from 0 to 8, y from -2 to 3
\put{$d_1 = $} at 0 .5 
\put{$d_2 = $} at 0 -4.5
\put{$\bullet$} at  1 -1  \put{$\bullet$} at  1 2
\put{$\bullet$} at  2 -1  \put{$\bullet$} at  2 2
\put{$\bullet$} at  3 -1  \put{$\bullet$} at  3 2
\put{$\bullet$} at  4 -1  \put{$\bullet$} at  4 2
\put{$\bullet$} at  5 -1  \put{$\bullet$} at  5 2
\put{$\bullet$} at  6 -1  \put{$\bullet$} at  6 2
\put{$\bullet$} at  7 -1  \put{$\bullet$} at  7 2
\put{$\bullet$} at  8 -1  \put{$\bullet$} at  8 2
\put{$\bullet$} at 9 -1  \put{$\bullet$} at  9 2
\put{$\bullet$} at  10 -1  \put{$\bullet$} at  10 2
\put{$\bullet$} at  1 -3  \put{$\bullet$} at  1 -6
\put{$\bullet$} at  2 -3  \put{$\bullet$} at  2  -6
\put{$\bullet$} at  3 -3  \put{$\bullet$} at  3  -6
\put{$\bullet$} at  4 -3  \put{$\bullet$} at  4  -6
\put{$\bullet$} at  5 -3  \put{$\bullet$} at  5  -6
\put{$\bullet$} at  6 -3  \put{$\bullet$} at  6  -6
\put{$\bullet$} at  7 -3  \put{$\bullet$} at  7  -6
\put{$\bullet$} at  8 -3  \put{$\bullet$} at  8  -6
\put{$\bullet$} at 9 -3  \put{$\bullet$} at  9  -6
\put{$\bullet$} at  10 -3  \put{$\bullet$} at  10  -6
\put{$v_{a}$} at 4 3  
\put{$v_{a}$} at 1 -2 
\put{$v_{a}^\ast$} at 2 -2 
\put{$v_{a}$} at 4 -2 
\put{$v_{a}^\ast$} at 5 -2 
\put{$v_{a}$} at 6 -2 
\put{$v_{a}$} at 10 -2 
\put{$v_{a}^\ast$} at 7 -2 
\put{$v_{a}$} at 8 -2
\put{$v_{a}^\ast$} at 9 -2
\put{$v_{a}$} at 8 -7 
\plot 4 2  1 -1 /
\plot 6 -3 8 -6 /
\setquadratic
\plot 1 -3 1.5 -4 2 -3 /
\plot 2 -1 3 0 4 -1 /
\plot 4 -3 4.5 -4 5 -3 /
\plot 5 -1 7.5 1 10 -1 /
\plot 7 -3 8.5 -4.5 10 -3 /
\plot 8 -3 8.5 -3.5 9 -3 /
\plot 7 -1 7.5 0 8 -1 /
\plot 6 -1 7.5 .5 9 -1 /
\endpicture}
$$
We claim that the product of the edge weights corresponding to the $2h$ horizontal edges in the middle of $d_1 d_2$ is 1.   To see this, direct the edges in the path so that the path travels from the top of $d_1$ to the bottom of $d_2$ as we have done here in our example,
$$
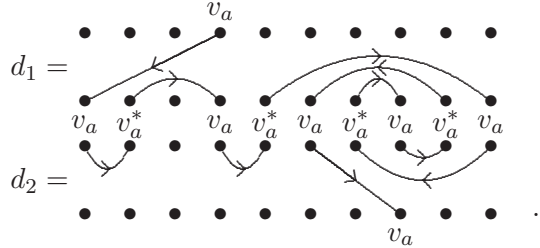
{\beginpicture
\setcoordinatesystem units <0.6cm,0.3cm>        
\setplotarea x from 0 to 8, y from -2 to 3
\put{$d_1 = $} at 0 .5 
\put{$d_2 = $} at 0 -4.5 
\put{$\bullet$} at  1 -1  \put{$\bullet$} at  1 2
\put{$\bullet$} at  2 -1  \put{$\bullet$} at  2 2
\put{$\bullet$} at  3 -1  \put{$\bullet$} at  3 2
\put{$\bullet$} at  4 -1  \put{$\bullet$} at  4 2
\put{$\bullet$} at  5 -1  \put{$\bullet$} at  5 2
\put{$\bullet$} at  6 -1  \put{$\bullet$} at  6 2
\put{$\bullet$} at  7 -1  \put{$\bullet$} at  7 2
\put{$\bullet$} at  8 -1  \put{$\bullet$} at  8 2
\put{$\bullet$} at 9 -1  \put{$\bullet$} at  9 2
\put{$\bullet$} at  10 -1  \put{$\bullet$} at  10 2
\put{$\bullet$} at  1 -3  \put{$\bullet$} at  1 -6
\put{$\bullet$} at  2 -3  \put{$\bullet$} at  2  -6
\put{$\bullet$} at  3 -3  \put{$\bullet$} at  3  -6
\put{$\bullet$} at  4 -3  \put{$\bullet$} at  4  -6
\put{$\bullet$} at  5 -3  \put{$\bullet$} at  5  -6
\put{$\bullet$} at  6 -3  \put{$\bullet$} at  6  -6
\put{$\bullet$} at  7 -3  \put{$\bullet$} at  7  -6
\put{$\bullet$} at  8 -3  \put{$\bullet$} at  8  -6
\put{$\bullet$} at 9 -3  \put{$\bullet$} at  9  -6
\put{$\bullet$} at  10 -3  \put{$\bullet$} at  10  -6
\put{$v_{a}$} at 4 3  
\put{$v_{a}$} at 1 -2 
\put{$v_{a}^\ast$} at 2 -2 
\put{$v_{a}$} at 4 -2 
\put{$v_{a}^\ast$} at 5 -2 
\put{$v_{a}$} at 6 -2 
\put{$v_{a}$} at 10 -2 
\put{$v_{a}^\ast$} at 7 -2 
\put{$v_{a}$} at 8 -2
\put{$v_{a}^\ast$} at 9 -2
\put{$v_{a}$} at 8 -7 
\put{.} at 11 -6
\plot 4 2  1 -1 /
\plot 6 -3 8 -6 /
\arrow <3 pt> [1,2] from 4 2 to 2.5 .5
\arrow <3 pt> [1,2] from 1.5 -4 to 1.55 -4
\arrow <3 pt> [1,2] from 3 0 to 3.05 0
\arrow <3 pt> [1,2] from 4.5 -4 to 4.55 -4
\arrow <3 pt> [1,2] from 7.5 1 to 7.55 1
\arrow <3 pt> [1,2] from 8.5 -3.5 to 8.55 -3.5
\arrow <3 pt> [1,2] from 7.5 0 to 7.55 0
\arrow <3 pt> [1,2] from 8.55 -4.5 to 8.5 -4.5
\arrow <3 pt> [1,2] from 7.55 .5 to 7.5 .5
\arrow <3 pt> [1,2] from 6 -3 to 7 -4.5
\setquadratic
\plot 1 -3 1.5 -4 2 -3 /
\plot 2 -1 3 0 4 -1 /
\plot 4 -3 4.5 -4 5 -3 /
\plot 5 -1 7.5 1 10 -1 /
\plot 7 -3 8.5 -4.5 10 -3 /
\plot 8 -3 8.5 -3.5 9 -3 /
\plot 7 -1 7.5 0 8 -1 /
\plot 6 -1 7.5 .5 9 -1 /
\endpicture}
$$
Observe that right arrows in the bottom of $d_1$ correspond to the weight $\langle v_a^\ast, v_a \rangleb$,  and left arrows correspond to the weight $\langle v_a, v_a^\ast \rangleb.$
Similarly, right arrows in the top of $d_2$ correspond to the weight $\langle v_a, v_a^\ast \ranglet$,  and left arrows correspond to the weight $\langle v_a^\ast, v_a \ranglet.$  Since our forms satisfy
$$
\langle v_a^\ast, v_a \rangleb \langle v_a, v_a^\ast \ranglet = \langle v_a, v_a^\ast \rangleb \langle v_a^\ast, v_a \ranglet = 1,
$$
the problem amounts to showing that there are the same number of right (and thus left) edges in $d_1$ as there are in  $d_2$.

 Our proof of this last statement is by induction on $h$. The case $h = 0$ is trivially true, so assume $h >0$.  Pick a horizontal edge $\{i,j\}$, with $i < j$,  in the top row of $d_2$ with no edge nested above it. That is, there are no edges in $d_2$ connected to vertices in  $\{i+1, \ldots, j-1\}$.  Now consider the edges $\{s,i\}$ and $\{j, t\}$ in $d_1$. Then $s,t \not\in \{i+1, \ldots, j-1\}$ and, by the planarity of $d_1$, we have $s < t$.   This means that $s < i$ or $t > j$ (or both).  Suppose that $t > j$ then the edges $\{i,j\}$ and $\{j,t\}$ have the same orientation in the path. We pair these two edges, delete them, and replace $\{s,i\}$ with $\{s,t\}$.  Because there are no edges connecting with $\{i+1, \ldots, j-1\}$, the new diagram is planar and $\{s,t\}$ has the same direction as $\{s,i\}$.  Now we apply induction to pair the rest of the edges in this new graph.  The case $s < i$ is entirely similar, and the 
 special case where $s$ or $t$ is in the top row is also easy.  This proves that the edges can be paired, and thus the product of the weights is 1,  resulting in a single vertical edge of weight 1 in  $d_1 d_2.$

\textbf{Case 2}:  \  A loop in the middle row of $d_1 d_2$.   

When there is a loop, the middle row consists of a cycle  of  $h$ horizontal edges in the bottom row of $d_1$ and $h$ horizontal edges in the top row of $d_2$ as in the picture below, 
where only the edges in the loop are displayed.  The only nonzero labeling will be a cycle $v_a, v_a^\ast, v_a, v_a^\ast, \ldots, v_a$ for $a = -1,0,1$ as illustrated here, 
$$
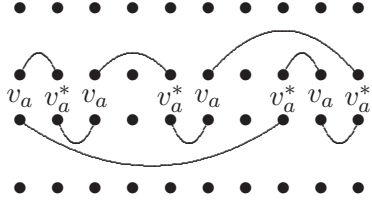
{\beginpicture
\setcoordinatesystem units <0.5cm,0.3cm>        
\setplotarea x from 0 to 8, y from -2 to 3
\put{$\bullet$} at  1 -1  \put{$\bullet$} at  1 2
\put{$\bullet$} at  2 -1  \put{$\bullet$} at  2 2
\put{$\bullet$} at  3 -1  \put{$\bullet$} at  3 2
\put{$\bullet$} at  4 -1  \put{$\bullet$} at  4 2
\put{$\bullet$} at  5 -1  \put{$\bullet$} at  5 2
\put{$\bullet$} at  6 -1  \put{$\bullet$} at  6 2
\put{$\bullet$} at  7 -1  \put{$\bullet$} at  7 2
\put{$\bullet$} at  8 -1  \put{$\bullet$} at  8 2
\put{$\bullet$} at 9 -1  \put{$\bullet$} at  9 2
\put{$\bullet$} at  10 -1  \put{$\bullet$} at  10 2
\put{$\bullet$} at  1 -3  \put{$\bullet$} at  1 -6
\put{$\bullet$} at  2 -3  \put{$\bullet$} at  2  -6
\put{$\bullet$} at  3 -3  \put{$\bullet$} at  3  -6
\put{$\bullet$} at  4 -3  \put{$\bullet$} at  4  -6
\put{$\bullet$} at  5 -3  \put{$\bullet$} at  5  -6
\put{$\bullet$} at  6 -3  \put{$\bullet$} at  6  -6
\put{$\bullet$} at  7 -3  \put{$\bullet$} at  7  -6
\put{$\bullet$} at  8 -3  \put{$\bullet$} at  8  -6
\put{$\bullet$} at 9 -3  \put{$\bullet$} at  9  -6
\put{$\bullet$} at  10 -3  \put{$\bullet$} at  10  -6
\put{$v_{a}$} at 1 -2  
\put{$v_{a}^\ast$} at 2 -2 
\put{$v_{a}$} at 3 -2 
\put{$v_{a}^\ast$} at 5 -2 
\put{$v_{a}$} at 6 -2 
\put{$v_{a}^\ast$} at 10 -2 
\put{$v_{a}$} at 9 -2 
\put{$v_{a}^\ast$} at 8 -2 
\setquadratic
\plot 1 -1 1.5 0 2 -1 /
\plot 2 -3 2.5 -4 3 -3 /
\plot 3 -1 4 0 5 -1 /
\plot 5 -3 5.5 -4 6 -3 /
\plot 6 -1 8 1 10 -1 /
\plot 9 -3 9.5 -4 10 -3 /
\plot 8 -1 8.5 0 9 -1 /
\plot 1 -3 4.5 -5 8 -3 /
\endpicture}
$$
Removing the edges connected to the leftmost $v_a$ leaves a path from $v_a^\ast$ (starting from the top of $d_2$) to $v_a^\ast$ (ending in the bottom of $d_1$).  The argument from the previous case shows that the product of the weights of this path is 1.  The product of the weights on the remaining two edges in the cycle is 
$$
\langle v_a, v_a^\ast \ranglet  \langle v_a, v_a^\ast \rangleb = \begin{cases}
-\qq^{-1} \ & \text{if $a = -1$}, \\
\ \ \, 1 \  & \text{if $a = 0$}, \\
-\qq \ & \text{if $a = 1$}.
\end{cases}
$$
Summing  over $a = -1,0,1$, as in \eqref{ParameterSum}, gives the weight $1 - \qq- \qq^{-1} = \zq$.

\textbf {Case 3}:  \ Horizontal edges in the top row of $d_1 d_2$.   

A horizontal edge in the top row of $d_1$ will also appear in the top row of $d_1 d_2$ with the appropriate weight.  It is also possible to gain a horizontal edge in $d_1 d_2$ through two vertical edges in $d_1$ connected by a path in the middle row of $d_1 d_2$ as pictured in the example below.  The path forces $v_a$ to be connected to $v_a^\ast$ in the top row.  The path in the middle row from the first $v_a$ to the last $v_a^\ast$ will have a weight of 1 by the argument in Case 1.  The remaining edge will be weighted $\langle v_a, v_a^\ast \ranglet$, as desired.
$$
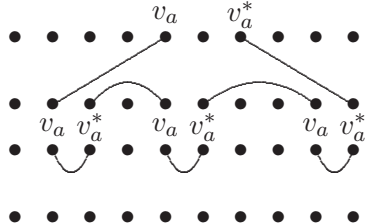
{\beginpicture
\setcoordinatesystem units <0.5cm,0.3cm>        
\setplotarea x from 0 to 8, y from -2 to 3
\put{$\bullet$} at  1 -1  \put{$\bullet$} at  1 2
\put{$\bullet$} at  2 -1  \put{$\bullet$} at  2 2
\put{$\bullet$} at  3 -1  \put{$\bullet$} at  3 2
\put{$\bullet$} at  4 -1  \put{$\bullet$} at  4 2
\put{$\bullet$} at  5 -1  \put{$\bullet$} at  5 2
\put{$\bullet$} at  6 -1  \put{$\bullet$} at  6 2
\put{$\bullet$} at  7 -1  \put{$\bullet$} at  7 2
\put{$\bullet$} at  8 -1  \put{$\bullet$} at  8 2
\put{$\bullet$} at 9 -1  \put{$\bullet$} at  9 2
\put{$\bullet$} at  10 -1  \put{$\bullet$} at  10 2
\put{$\bullet$} at  1 -3  \put{$\bullet$} at  1 -6
\put{$\bullet$} at  2 -3  \put{$\bullet$} at  2  -6
\put{$\bullet$} at  3 -3  \put{$\bullet$} at  3  -6
\put{$\bullet$} at  4 -3  \put{$\bullet$} at  4  -6
\put{$\bullet$} at  5 -3  \put{$\bullet$} at  5  -6
\put{$\bullet$} at  6 -3  \put{$\bullet$} at  6  -6
\put{$\bullet$} at  7 -3  \put{$\bullet$} at  7  -6
\put{$\bullet$} at  8 -3  \put{$\bullet$} at  8  -6
\put{$\bullet$} at 9 -3  \put{$\bullet$} at  9  -6
\put{$\bullet$} at  10 -3  \put{$\bullet$} at  10  -6
\put{$v_{a}$} at 5 3 
\put{$v_{a}^\ast$} at 7 3 
\put{$v_{a}$} at 2 -2 
\put{$v_{a}^\ast$} at 3 -2 
\put{$v_{a}$} at 5 -2 
\put{$v_{a}^\ast$} at 6 -2 
\put{$v_{a}^\ast$} at 10 -2 
\put{$v_{a}$} at 9 -2 
\plot 5 2  2 -1 /
\plot 7 2 10 -1 /
\setquadratic
\plot 2 -3 2.5 -4 3 -3 /
\plot 3 -1 4 0 5 -1 /
\plot 5 -3 5.5 -4 6 -3 /
\plot 6 -1 7.5 0 9 -1 /
\plot 9 -3 9.5 -4 10 -3 /
\endpicture}$$

\textbf {Case 4}: \  Horizontal edges in the bottom row of $d_1 d_2$.  

This case is completely analogous to Case 3.    \end{proof}

Using the representation of $\Motz$ on $\VV^{\otimes k}$, we can now show the following Schur-Weyl duality result  holds in this setting.

\begin{thm}\label{thm:commutingact}  Let $\pi_k: \Motz \rightarrow \End(\VV^{\otimes k})$
and $\rho_{\VV^{\otimes k}}: \uqsl \rightarrow  \End(\VV^{\otimes k})$ 
be the representations afforded by the action of  $\Motz$ and
$\uqsl$  on
$\VV^{\otimes k}.$   Then  $\pi_k(\Motz)$ and   $\rho_{\VV^{\otimes k}} (\uqsl)$ commute.
Thus, $\pi_k(\Motz)\subseteq \End_{\uqsl}(\VV^{\otimes k})$ and  
$\rho_{\VV^{\otimes k}}(\uqsl) \subseteq \End_{\Motz}(\VV^{\otimes k}).$   
\end{thm}
\begin{proof} The elements $t_i, r_i, \ell_i$ $(1 \leq i < k)$  generate $\Motz$.  A straightforward matrix multiplication shows that their representing
transformations, which are given in \eqref{eq:TLRact},  commute with those of the generators $E, F, K_i^{\pm 1}$, $i=1,2$,  of $\uqsl$
on $\VV^{\otimes k}$.   We provide one illustrative calculation here.  Consider $T E$ acting on the simple tensor $v_{-1} \otimes v_{-1} \in \VV \otimes \VV$.  By \eqref{QuantumGroupAction} and \eqref{coproduct} we have
\begin{align*}
TE (v_{-1} \otimes v_{-1}) 
& = T \left( E v_{-1} \otimes K v_{-1} + v_{-1} \otimes E  v_{-1}\right) \\
& = T \left( q^{-1} v_{1} \otimes  v_{-1} + v_{-1} \otimes  v_{1}\right) \\
& = \left( q^{-1} \langle v_1, v_{-1}\rangleb +  \langle v_{-1}, v_{1}\rangleb \right)  
\left( \qq^{-1/2} v_{-1} \otimes v_1 + v_0 \otimes v_0   -\qq^{1/2} v_{1} \otimes v_{-1}     \right).
\end{align*}
On the other hand $E T (v_{-1} \otimes v_{-1}) = 0$. Thus $E$ and $T$ commute as operators on $v_{-1} \otimes v_{-1}$,  since $q^{-1} \langle v_{1}, v_{-1}\rangleb +  \langle v_{-1}, v_{1}\rangleb =0$.
\end{proof} 
\medskip

Now we turn our attention to proving that the representation 
$$\pi_k:  \Motz  \rightarrow 
\End_{\uqsl}(\VV^{\otimes k})$$  is faithful. 
The proof of this fact,  as well as the proof of Theorem \ref{thm:tensordecomp}, will involve the following order $\le$.
Suppose $\underline{\alpha} = (\alpha_1, \dots, \alpha_\ell)$
is a sequence of positive integers with $1 \leq \alpha_1 < \alpha_2 < \cdots < \alpha_\ell \leq k$,
and assume  $\underline{\alpha'} = (\alpha_1', \dots, \alpha_\ell')$
is another such sequence.  Then we say $\underline \alpha < \underline \alpha'$
if $\alpha_i = \alpha_i'$ for $i < s$, but $\alpha_s < \alpha_s'$. 
When the $\alpha_i$ in such a sequence $\underline \alpha$ correspond to positions of the 
vertices which are 
the left ends of the horizontal edges in the bottom row of a $k$-diagram, then we
say $\underline \alpha$ is the left-end sequence of the bottom row.

Suppose there exists a nonzero  $y = \sum_{d \in \mathfrak{M}_k} a_{d} d \in \ker \pi_k$.
Choose  a diagram $d'$ so  that
\begin{itemize}
\item[{\rm (i)}]   $a_{d'} \neq 0$; 
\item[{\rm (ii)}]  among the diagrams satisfying (i),  $d'$ has a maximum number, say $m$,  of vertical edges; 
\item[{\rm (iii)}] among the diagrams satisfying (i) and (ii), $d'$ has a maximum
number, say $\ell$,  of horizontal edges in its bottom row;
\item[{\rm (iv)}] among the diagrams satisfying (i), (ii), and (iii), $d'$ has the minimal
left-end sequence, say $\underline \alpha$,   in its bottom row.  
\end{itemize}

Let $u$ in $\VV^{\otimes k}$ be the simple tensor having $v_{-1}$ as the factor
in the positions in $\underline \alpha$; \  $v_0$ in the positions corresponding
to isolated vertices in the bottom
row of $d'$;  and $v_1$ in all the other tensor slots. 

Now suppose $d''$ satisfies $a_{d''} \neq 0$ and has $d''u \neq 0$.   Consider the edges
in the bottom row of  $d''$ which match up with the $\ell+m$  $v_1$'s in $u$.  There can be at most $\ell$
such horizontal edges in $d''$, as there are  $\ell$ factors $v_{-1}$ in $u$.  
There can be at most $m$ vertical edges in $d''$ by (ii).  But for $d''u$ to
be nonzero, there must be an edge in the bottom row of $d''$ in the same positions
as the $\ell+m$  $v_1$'s in $u$.   So there must be exactly $\ell$ horizontal and $m$
vertical edges in the bottom row of $d''$ which meet the $\ell+m$ factors  $v_1$  in $u$.   
This accounts for all the horizontal edges in the bottom row of $d''$ by (iii).    

Assume  $\underline{\alpha'} = (\alpha_1', \dots, \alpha'_\ell)$, where
$\alpha_1' < \cdots < \alpha'_\ell$, 
is the left-end sequence of the bottom row of $d''$.      The first horizontal edge in $d''$
has one of its endpoints at $\alpha_1$.  Thus,  $\alpha_1' \leq \alpha_1$, and
by the minimality of $\underline \alpha$,  it must be that $\alpha_1' = \alpha_1$.
The next  horizontal edge in $d''$ has one of its endpoints at $\alpha_2$, and
again by minimality, $\alpha_2' = \alpha_2$.   Proceeding in this way, we 
determine that $\underline \alpha' = \underline \alpha$.   

For $j = 1,\dots, \ell$,  adding the subscripts of the factors of $u$ occurring in slots $\alpha_j, \alpha_j+1, \dots$,
we see by the  planar property of diagrams,  that the subscript sum first becomes 0
at the right-hand endpoint of the horizontal edge in $d'$ (and also $d''$) with left endpoint at 
$\alpha_j$.   At the corresponding tensor slot in $u$,  there must be the vector $v_1$.    The other $v_1$'s in
$u$ occur where $d'$ has a vertex in its bottom row with a vertical edge emanating
from it.   But since $d'' u \neq 0$, $d''$ must have a vertical edge in its bottom row
at those positions also.    By the maximality of the number of vertical edges
in the bottom row of $d'$,  this forces $d''$ to have exactly the same bottom row as $d'$.  

Now consider all the diagrams having a nonzero coefficient in $y$ and having precisely
the same bottom row as $d'$.    Look at  the simple tensors produced by applying those
diagrams to $u$, and choose one of them, call it  $t$, with a maximal number, say $n$, of $v_{-1}$'s and
having them   in positions $\gamma_1 < \cdots < \gamma_n$,
with $\underline \gamma = (\gamma_1, \dots, \gamma_n)$ minimal.  The top
row of any diagram which produces $t$ when applied to $u$ must have exactly
$m$ vertical edges, and $t$ must have $v_1$'s in the corresponding positions.  Arguing
as we have done with the bottom rows, we see that the top row of
any such diagram producing $t$ with nonzero coefficient when it is applied to $u$  is uniquely determined.
Hence there is a unique  diagram $\tilde d$ with $a_{\tilde d} \neq 0$ which
produces the simple tensor $t$ when $y$ is applied to $u$.   But since
$y$ is in $\ker \pi_k$,  and $t$ occurs in $y u = 0$ with coefficient $a_{\tilde d}$, 
we force $a_{\tilde d} = 0$, contrary to assumption.
This shows that $\ker \pi_k = 0$.       Thus, we have the following:

\begin{thm}\label{thm:main}  The representation $\pi_k:  \Motz \rightarrow 
\End_{\uqsl}(\VV^{\otimes k})$ is faithful.   Hence $\End_{\uqsl}(\VV^{\otimes k})
\cong \Motz$.  \end{thm}

The second assertion in the theorem follows from the faithfulness of $\pi_k$
and the fact that these two algebras have the same dimension. 
\medskip

\begin{remark} {\rm It is also possible to derive Theorem \ref{thm:main}, with some work, from the canonical bases of Frenkel and Khovanov \cite{FK}.  However, in  this approach, the identity diagram in the Motzkin algebra, which acts as the identity transformation on tensor space,  corresponds to a sum of canonical basis elements, and a transition matrix is required to get the correspondence between the bases.}\end{remark}

\begin{remark} {\rm It follows from the double centralizer theory that 
$$\End_{\Motz}(\VV^{\otimes k})
\cong  \rho_{\VV^{\otimes k}}(\uqsl),$$ which is the {\it Schur algebra} in this setting. 
For simplicity we denote the semisimple algebra $\rho_{\VV^{\otimes k}}(\uqsl)$ by $\mathcal{S}(\VV^{\otimes k})$  and observe that
$$\dim \left( \mathcal{S}(\VV^{\otimes k})\right) = \sum_{r=0}^k  (r+1)^2 = \frac{1}{6}(k+1)(k+2)(2k+3).$$
If  $\mathcal{S}(\VV(1)^{\otimes k})= \rho_{\VV(1)^{\otimes k}}(\uqsl)$, the subalgebra 
of $\End(\VV(1)^{\otimes k})$ generated by
the representing transformations of $\uqsl$ on $\VV(1)^{\otimes k}$ (that is, the Schur algebra
for $\VV(1)^{\otimes k}$), then it follows
from the Clebsch-Gordan formula \eqref{eq:CG2}  that 
$$\dim \left( \mathcal{S}(\VV^{\otimes k})\right) = \dim \left(\mathcal{S}(\VV(1)^{\otimes k}) \right) +  \dim \left( \mathcal{S}(\VV(1)^{\otimes k-1})\right). $$} 
\end{remark}
\end{subsection}  

\begin{subsection} {An explicit decomposition of $\VV^{\otimes k}$ into irreducible  
$\uqsl$-modules.}  \label{subsec:decomp}
 
Recall the notion of a Motzkin path of length $k$ as in  Section \ref{subsec:MotzkinPaths}.
  To each such path $p\in  \mathcal P_k$,  we associate a simple tensor $u_p$ as follows.
Suppose that the pairs $(\alpha_i, \beta_i)$ for $1 \leq i \leq \ell$  denote the left-hand and right-hand endpoints   of the horizontal edges in $p$,  where $\alpha_1 < \dots < \alpha_\ell$.   Construct the
simple tensor $u_p$  having 
$v_{-1}$ at the positions $\alpha_i$, $v_1$ at the positions $\beta_i$ and also at  the positions of the white vertices in $p$, and $v_0$ at all other positions.

 Now for each path $p$,  let $d_p^p$ be the symmetric diagram in $\Motz$ 
 determined by $p$  as in Section \ref{subsec:MotzkinPaths},   and set 
 $$w_p = d_p^p u_p.$$
 We make a number of claims about the vectors $w_p$. 
 \medskip
 
\noindent \textbf{Claim 1.}   $w_p \neq 0$.   
\medskip
If $p$ has $n$ horizontal edges,  then in  $w_p$ the simple tensor $u_p$ occurs with coefficient 
equal to 
$$\langle v_{-1}, v_1 \rangle_{\mathsf b}^n \langle v_{-1}, v_1 \rangle_{\mathsf t}^n = (-\qq^{-1/2})^n (\qq^{-1/2})^n= (-\qq^{-1})^n \neq 0.$$

 \medskip
 
\noindent \textbf{Claim 2.}   $K_i u_p = \qq^{k_i} u_p$ and $K_i w_p = \qq^{k_i} w_p$ for $i=1,2$,  where  $k_1$ is the number of $v_1$ in $u_p$ and $k_2$ is the number of $v_{-1}$ in $u_p$. 
Hence, $K u_p =  \qq^{\rank(p)}u_p$ and $K w_p = \qq^{\rank(p)} w_p$, where $\rank(p)$ is the number of white vertices
in $p$.  

\medskip  That $w_p$ and $u_p$ have the same weights relative to the $K_i$  follows from the fact that
the diagram $d_p^p$ lives in the centralizer algebra of the $\uqsl$-action, hence commutes with $K_i$.      The rest is clear from the fact that $K_i$ acts on
$\VV^{\otimes k}$ by $K_i \otimes \cdots \otimes K_i$. \medskip  

\noindent \textbf{Claim 3.}  $\{w_p \mid p \in \mathcal P_k\}$ is a linearly independent set.
\medskip

Since we are assuming $\qq$ is not a root of unity, and since vectors of different weights are linearly independent, we may assume
that we have a nontrivial dependence relation  $\sum_{p'} a_{p'} w_{p'} = 0$, where 
the sum ranges over the paths $p'$ of rank $r$.   

Now suppose among the terms occurring in this sum with nonzero coefficient,  $p$ is chosen so that it  has a maximal number of horizontal
edges, say $\ell$, and minimal left-end sequence 
$\underline \alpha = (\alpha_1, \dots, \alpha_\ell)$,
where $\alpha_1 < \cdots < \alpha_\ell$.       Suppose $p'$ is another path with $a_{p'} \neq 0$
such that $u_p$ occurs in $w_{p'}$ with nonzero coefficient.  In each tensor summand of $w_{p'}$ the vector $v_{-1}$ occurs only in positions which correspond to some horizontal edge of $p'$.  
Let $\underline \alpha'$ be the left-end sequence of $p'$.   The positions $\alpha_j$ with the $v_{-1}$  in $u_p$, must correspond to some horizontal edge endpoint  in $p'$.  Thus, $p'$ also has $\ell$ horizontal edges by the maximality of $\ell$.   Moreover, $\alpha_1' \leq \alpha_1$ has to hold, and  by the minimality of $\underline \alpha$,  we must have $\alpha_1' = \alpha_1$.  Proceeding in order with the  $\alpha_j$, we determine that $\underline \alpha' = \underline \alpha$.   

Now $w_{p'} = d_{p'}^{p'}u_{p'}$ has $3^\ell$ summands, but only one of them has
the vector $v_{-1}$ in all the positions corresponding to $\underline \alpha$, namely
$u_{p'}$.   Since we are assuming $u_p$ occurs in $w_{p'}$ with nonzero coefficient,
and since it has the same property,  $u_{p'} = u_p$.    We claim this forces 
$p' = p$.   Indeed, $p$ has a horizontal edge with left-end  position $\alpha_j$.  Starting at slot $\alpha_j$ and  summing the subscripts on the tensor factors of $u_p = u_{p'}$, the  first place where the subscript sum is 0 must be right-hand end position $\beta_j$ of the corresponding horizontal edge for each $j=1,\dots, \ell$.   The other positions where $v_1$ occurs correspond to white vertices of $p$
and $p'$, and the positions of the
 $v_0$ factors  of $u_p = u_{p'}$ are the locations of the isolated black vertices in $p$ and $p'$.  Hence
 $p' = p$, and therefore $u_p$ occurs in $\sum_{p'} a_{p'} w_{p'}$ exactly once, with coefficient
 $(1-\qq-\qq^{-1})^\ell a_p$, which implies $a_p = 0$.  We have reached a contradiction,
 so no such nontrivial dependence relation can occur.   
 \medskip
 
 \noindent \textbf{Claim 4.}   $w_p$ is a highest weight vector for $\uqsl$
 of weight $\qq^{k_i}$ relative to $K_i$,  where the values of $k_i$  are as in Claim 2.  \medskip
 
We know already that $w_p$ is a weight vector for the $K_i$ of the appropriate weight.
We need to argue it is annihilated by $E$.    Using the expression for the coproduct, we see that $E$
acts on $\VV^{\otimes k}$ by 
$$\sum_{j=0}^{k-1}  \hbox{\rm id}_{\VV}^{\otimes j} \otimes E \otimes K^{\otimes (k-1-j)}.$$

\noindent  Since $Ew_p = Ed_p^p u_p = d_p^p Eu_p$, we first consider the vector
$Eu_p$.     Now a term in the above sum has a nonzero action on $u_p$ only when
$E$ acts on a tensor slot containing the vector $v_{-1}$, which is a position in $p$ 
that is a left-end node of a horizontal  edge.   It changes $v_{-1}$ to $v_1$ at that
position.  But then when $d_p^p$ acts on the changed tensor, it finds $v_1$ at
both ends of a horizontal edge of $p$, and so gives 0 since $\langle v_1, v_1 \rangle_{\mathsf b} = 0$.    
Thus, $Ew_p = 0$.

 \medskip
 
  \noindent \textbf{Claim 5.}  $\uqsl w_p$ is an irreducible $\uqsl$-module of dimension $\rank(p) + 1$.
  \medskip
  
 From Claims 1,2,4, we know that when  $\uqsl w_p$ is viewed as a module
 for the subalgebra $\mathsf{U}_\qq(\mathfrak{sl}_{\mathsf{2}})$ it is  a nonzero homomorphic
 image of the Verma module for $\mathsf{U}_\qq(\mathfrak{sl}_{\mathsf{2}})$ with highest weight 
 $\qq^{\rank(p)}$ relative to $K$.   But the only quotients of that Verma module
 are $0$, the whole Verma module, and the irreducible $\mathsf{U}_\qq(\mathfrak{sl}_{\mathsf{2}})$-module of 
 dimension $\rank(p) + 1$ (see \cite[Prop.~2.5]{Ja}).    Since  $\uqsl w_p \subseteq \VV^{\otimes k}$  is finite dimensional and nonzero, it is the desired irreducible quotient.
 
 \medskip
 
  \noindent \textbf{Claim 6.}  The sum $\sum_{p', \rank(p') = r}   \uqsl w_{p'}$
  is direct.    \medskip
  
  If not,  one of the summands, say $\uqsl w_{p}$,  intersects the sum of the
  others nontrivially, so by irreducibility, it must be contained in the sum of the others.    
  The number of linearly independent highest weight vectors of weight
  $\qq^r$ relative to $K$  is exactly $| \mathcal P^r_k |$, the number of Motzkin paths of length $k$ and rank $r$.
    Since $w_p$ is in $\sum_{p' \neq p, \rank(p') = r}  \uqsl w_{p'}$,
  that would force it to be a linear combination of the $w_{p'}$, $p' \neq p$, which
  is impossible by Claim 3.   \medskip
  
 As a consequence of these claims we have the following result:
 
 \begin{thm} \label{thm:tensordecomp}
  $\VV^{\otimes k} = \bigoplus_{p \in \mathcal P_k}  \uqsl w_p$
 is a decomposition of $\VV^{\otimes k}$ into irreducible $\uqsl$-modules.
 \end{thm}
 
 \begin{cor}  $\bigoplus_{p \in \mathcal P_k^r}  \uqsl w_p$ is an irreducible
 $\big(\uqsl,  \Motz\big)$-bimodule under the action,  
 $(a \times d).(b.w_p) =  ab. w_{d\cdot p}$ for $a,b \in  \uqsl$ and 
$d \in \mathfrak{M}_k$, where $d\cdot p$ is as in (\ref{ActionOnPaths2}) below. \end{cor}

\end{subsection}

\end{section}

 \begin{section}{The Action on Motzkin Paths and Cellularity}\label{sec:Cellularity}
 
In this section,  we define a representation of the Motzkin algebra $\M_k(x)$ on Motzkin paths.  This action is graded by the rank of the path, and so by taking quotients, we produce, in the case that $x$ is chosen so that $\M_k(x)$ is a semisimple algebra,  a complete set of irreducible modules for the Motzkin algebra.  We show that the Motzkin algebra is cellular in the sense of Graham and Lehrer \cite{GL} and that our Motzkin path modules are isomorphic to the (left) cell modules.

 \subsection{The action of $\M_k(x)$ on Motzkin paths} \label{subsec: motzact}

Let $\KK$ denote a commutative ring with 1, and let $\M_k(x)$ be the Motzkin
algebra over $\KK$.  Thus, $\M_k(x)$ is a free $\KK$-module with basis the
Motzkin $k$-diagrams in $\mathfrak{M}_k$, and $x$ is assumed to be an element of $\KK$.

Recall from Section \ref{subsec:MotzkinPaths} that $\cP_k$ is the set of Motzkin paths, or 1-factors, of length $k$ and that $\cP_k^r \subseteq \cP_k$ is the subset of Motzkin paths of length $k$ and rank $r$.  We define an action of 
a Motzkin diagram $d \in \M_k(x)$ on a Motzkin path $p \in \cP_k$ as follows:
\begin{enumerate}

\item  Place  $d$ above  $p$ and connect the vertices in the bottom row of $d$ with the vertices in $p$ to create a new Motzkin diagram $dp$.  

\item Color the vertices in the top row of $dp$  so that vertices connected by an edge have the same color. 

\item  Let $q$ be the 1-factor that is formed from the top row of $dp$; the planarity of $d$ ensures that $q$ is a Motzkin 1-factor.

\item  Let $\kappa(d,p)$ equal the number of closed loops formed in the bottom row of $dp$.  Then set
\begin{equation}\label{ActionOnPaths}
dp = x^{\kappa(d,p)} q.
\end{equation}
\end{enumerate}
For example if $d$ and $p$ are given by
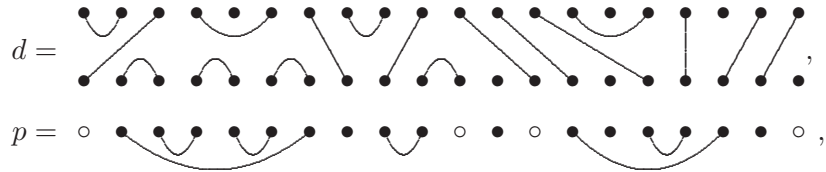
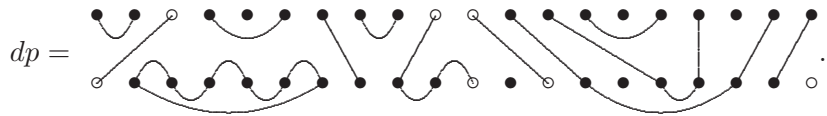
\begin{align*}  d & ={ \beginpicture
 \setcoordinatesystem units <0.5cm,0.3cm>        
\setplotarea x from .5 to 20, y from -2 to 2
\plot 3 2 1 -1 /
\plot 7 2 8 -1 /
\plot 10 2 9 -1 /
\plot 11 2 13 -1 /
\plot 12 2 14 -1 /
\plot 13 2 16 -1 /
\plot 17 2 17 -1 /
\plot 19 2 18 -1 /
\plot 20 2 19 -1 /
\setquadratic
\plot 2 -1 2.5 0 3 -1 /
\plot 4 -1 4.5 0 5 -1 /
\plot 6 -1 6.5 0 7 -1 /
\plot 1 2 1.5 1 2 2 /
\plot 4 2 5 1 6 2 /
\plot 8 2   8.5 1.0   9 2 /
\plot 14 2 15 1 16 2 /
\plot 10 -1  10.5 0   11 -1  /
\setlinear
\put{$\bullet$} at 1 2  \put{$\bullet$} at 2 2 \put{$\bullet$} at 3 2  \put{$\bullet$} at 4 2 
\put{$\bullet$} at 5 2  \put{$\bullet$} at 6 2  \put{$\bullet$} at 7 2  \put{$\bullet$} at 8 2 
\put{$\bullet$} at 9 2  \put{$\bullet$} at 10 2  \put{$\bullet$} at 11 2  \put{$\bullet$} at 12 2 
\put{$\bullet$} at 13 2  \put{$\bullet$} at 14 2  \put{$\bullet$} at 15 2 
\put{$\bullet$} at 16 2 \put{$\bullet$} at 17 2  \put{$\bullet$} at 18 2  \put{$\bullet$} at 19 2 
\put{$\bullet$} at 20 2 
\put{$\bullet$} at 1 -1 
\put{$\bullet$} at 2 -1
\put{$\bullet$} at 3 -1 
\put{$\bullet$} at 4 -1 
\put{$\bullet$} at 5 -1 
\put{$\bullet$} at 6 -1 
\put{$\bullet$} at 7 -1 
\put{$\bullet$} at 8 -1 
\put{$\bullet$} at 9 -1 
\put{$\bullet$} at 10 -1 
\put{$\bullet$} at 11 -1 
\put{$\bullet$} at 12 -1 
\put{$\bullet$} at 13 -1 
\put{$\bullet$} at 14 -1 
\put{$\bullet$} at 15 -1 
\put{$\bullet$} at 16 -1 
\put{$\bullet$} at 17 -1 
\put{$\bullet$} at 18 -1 
\put{$\bullet$} at 19 -1 
\put{$\bullet$} at 20 -1
\endpicture}, \\
 p & = {\beginpicture
 \setcoordinatesystem units <0.5cm,0.2cm>        
\setplotarea x from .5 to 20.5, y from -1 to 1
\setquadratic
\plot 3 .5      3.5  -1       4 .5 /
\plot 5 .5      5.5  -1       6 .5 /
\plot 9 .5      9.5  -1       10 .5 /
\plot 16 .5     16.5  -1       17 .5 /
\plot 2 .5      4.5  -2       7 .5 /
\plot 14 .5      16  -2       18 .5 /
\put{$\redbull $} at 1 .5 
\put{$\bullet$} at 2 .5 
\put{$\bullet$} at 3 .5
\put{$\bullet$} at 4 .5 
\put{$\bullet$} at 5 .5
\put{$\bullet$} at 6 .5 
\put{$\bullet$} at 7 .5 
\put{$\bullet$} at 8 .5  
\put{$\bullet$} at 9 .5 
\put{$\bullet$} at 10 .5 
\put{$\redbull $} at 11 .5
\put{$\bullet$} at 12 .5
\put{$\redbull$} at 13 .5 
\put{$\bullet$} at 14 .5
\put{$\bullet$} at 15 .5
\put{$\bullet$} at 16 .5 
\put{$\bullet$} at 17 .5 
\put{$\bullet$} at 18 .5 
\put{$\bullet$} at 19 .5 
\put{$\redbull$} at 20 .5
\endpicture},
\end{align*}
then 
$$
dp = 
{ \beginpicture
 \setcoordinatesystem units <0.5cm,0.3cm>        
\setplotarea x from .5 to 20, y from -2 to 2
\plot 3 2 1 -1 /
\plot 7 2 8 -1 /
\plot 10 2 9 -1 /
\plot 11 2 13 -1 /
\plot 12 2 14 -1 /
\plot 13 2 16 -1 /
\plot 17 2 17 -1 /
\plot 19 2 18 -1 /
\plot 20 2 19 -1 /
\setquadratic
\plot 2 -1 2.5 0 3 -1 /
\plot 4 -1 4.5 0 5 -1 /
\plot 6 -1 6.5 0 7 -1 /
\plot 1 2 1.5 1 2 2 /
\plot 4 2 5 1 6 2 /
\plot 8 2   8.5 1.0   9 2 /
\plot 14 2 15 1 16 2 /
\plot 10 -1  10.5 0   11 -1  /
\plot 3 -1     3.5  -1.75       4 -1 /
\plot 5 -1      5.5  -1.75       6 -1 /
\plot 9 -1     9.5  -1.75       10 -1 /
\plot 16 -1     16.5  -1.75       17 -1 /
\plot 2 -1      4.5  -2.3       7 -1 /
\plot 14 -1      16  -2.3       18 -1 /
\setlinear
\put{$\bullet$} at 1 2  \put{$\bullet$} at 2 2 \put{$\redbull$} at 3 2  \put{$\bullet$} at 4 2 
\put{$\bullet$} at 5 2  \put{$\bullet$} at 6 2  \put{$\bullet$} at 7 2  \put{$\bullet$} at 8 2 
\put{$\bullet$} at 9 2  \put{$\redbull$} at 10 2  \put{$\redbull$} at 11 2  \put{$\bullet$} at 12 2 
\put{$\bullet$} at 13 2  \put{$\bullet$} at 14 2  \put{$\bullet$} at 15 2 
\put{$\bullet$} at 16 2 \put{$\bullet$} at 17 2  \put{$\bullet$} at 18 2  \put{$\bullet$} at 19 2 
\put{$\bullet$} at 20 2 
\put{$\redbull$} at 1 -1 
\put{$\bullet$} at 2 -1
\put{$\bullet$} at 3 -1 
\put{$\bullet$} at 4 -1 
\put{$\bullet$} at 5 -1 
\put{$\bullet$} at 6 -1 
\put{$\bullet$} at 7 -1 
\put{$\bullet$} at 8 -1 
\put{$\bullet$} at 9 -1 
\put{$\bullet$} at 10 -1 
\put{$\redbull$} at 11 -1 
\put{$\bullet$} at 12 -1 
\put{$\redbull$} at 13 -1 
\put{$\bullet$} at 14 -1 
\put{$\bullet$} at 15 -1 
\put{$\bullet$} at 16 -1 
\put{$\bullet$} at 17 -1 
\put{$\bullet$} at 18 -1 
\put{$\bullet$} at 19 -1 
\put{$\redbull$} at 20 -1
\endpicture}.
$$
There is one closed loop in the bottom row of $dp$, so we get $dp = x q$ where
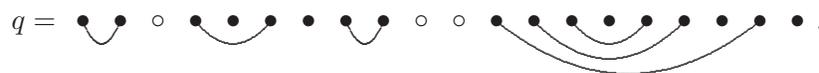
$$
\quad q = { \beginpicture
 \setcoordinatesystem units <0.5cm,0.2cm>        
\setplotarea x from .5 to 20.5, y from -3 to 2
\setquadratic
\plot  1 .5 1.5 -1 2 .5 / 
\plot 4 .5 5 -1 6 .5 / 
\plot  8 .5  8.5  -1  9  .5  /
\plot 14 .5 15 -1 16 .5 / 
\plot 13 .5 15 -2 17 .5 /
\plot 12 .5 15.5 -3 19 .5 / 
\put{$\redbull$} at 3 .5
\put{$\bullet$} at 7 .5
\put{$\bullet $} at 9 .5
\put{$\redbull $} at 11 .5
\put{$\bullet $} at 20 .5
\put{$\bullet$} at 13 .5
\put{$\bullet$} at 17 .5
\put{$\bullet$} at 12 .5
\put{$\bullet$} at 19 .5
\put{$\bullet$} at 8 .5
\put{$\bullet$} at 1 .5
\put{$\bullet$} at 2 .5
\put{$\bullet$} at 4 .5
\put{$\bullet$} at 6 .5
\put{$\bullet$} at 14 .5
\put{$\bullet$} at 16 .5
\put{$\bullet$} at 5 .5
\put{$\redbull$} at 10 .5
\put{$\bullet$} at 15 .5
\put{$\bullet$} at 18 .5
\endpicture}.
$$

Let 
$\W_k$ be the free $\KK$-module with basis  the Motzkin paths in $\cP_k$.   The action of diagrams on paths defined in \eqref{ActionOnPaths} comes from the concatenation of diagrams, which is associative.  Furthermore, vertex colors are simply pulled along the edges in the diagram, so we have $(d_1 d_2)p = d_1(d_2p)$ for all Motzkin diagrams $d_1, d_2 \in \M_k(x)$ and all paths $p \in {\mathcal P}_k$.  This action extends linearly to an action of $\M_k(x)$ on $\W_k$ making $\W_k$ a module for $\M_k(x)$.

Recall from Section \ref{subsec:MotzkinPaths} the definition of the Motzkin diagram  $d_p^q \in \M_k(x)$ of rank $r$  associated with the pair $(p,q)$  of rank $r$ Motzkin paths.  We have
\begin{equation}\label{eq:Eigenvalue}
d_p^q p = x^{\varepsilon(p)} q, \qquad \hbox{ where $\varepsilon(p)$ is the number of edges in $p$.  }
\end{equation}
In particular, $p$ is an eigenvector of   $d_p^p$ with eigenvalue  $x^{\varepsilon(p)}$.
The action of a diagram $d \in \M_k(x)$ on a path $p \in \cP_k$ satisfies
\begin{equation}\label{rank.property}
\rank( dp ) \le \min\left( \rank(d), \rank(p) \right),
\end{equation}
since a white vertex in $p$ survives only if it passes along a vertical edge in $d$.  This means that
\begin{equation}\label{eq:Wkrdef} 
\W^{(r)}_k  = \mathsf{span}_\KK\left\{ \ p \in \cP_k  \  \big\vert \ \rank(p) \le r \  \right\}
\end{equation}
is a submodule of $\W_k$, and $\W_k^{(0)} \subseteq \W_k^{(1)} \subseteq \cdots \subseteq \W_k^{(k)} = \W_k$ is a filtration of the $\M_k(x)$-module $\W_k$.  Let 
\begin{equation}\label{eq:Cmoddef} \C_k^{(r)} = \W_k^{(r)}/\W_k^{(r-1)} \cong   \mathsf{span}_\KK \left\{ \ p \mid p  \in \cP_k^r \  \right\}, \end{equation} 
where $\cP_k^r  = \left\{ \ p \in \cP_k  \  \big\vert \ \rank(p) =  r \  \right\}$ as 
in \eqref{eq:rankrpaths}.   Under this isomorphism, if the Motzkin diagram $d \in \M_k(x)$ acts on a path $p \in \cP_k$ by
$dp = x^{\kappa(d,p)} q$, then 
(abusing coset notation) we write the action in $\C_k^{(r)}$ as follows:
\begin{equation}\label{ActionOnPaths2}
d\cdot p = 
\begin{cases}
x^{\kappa(d,p)} q, & \text{if $\rank(q) = \rank(p)$} \\
0, & \text{if $\rank(q) < \rank(p)$}.
\end{cases}
\end{equation}
Let $\C_0^{(0)} = \KK$ be the trivial module for $\M_0(x) = \KK 1$.

The next theorem can also be obtained from general results on cellular algebras, once
we establish, in Section \ref{subsec:cellular}, that the algebra $\M_k(x)$ is cellular.  Here we
give an elementary proof of this result without evoking that machinery.

\begin{thm}\label{thm:irredmod}  Assume the Motzkin algebra $\M_k(x)$ is defined over
a field $\KK$.   For $k \ge 0$ and $0 \le r \le k$, the module 
$\C_k^{(r)}$ is an indecomposable $\M_k(x)$-module under the action given in \eqref{ActionOnPaths2}.  When $x$ is chosen so that $\M_k(x)$ is
a semisimple $\KK$-algebra, then  $\{\ \C_k^{(r)} \ \vert \ 0 \le r \le k \ \}$ is a complete set of pairwise nonisomorphic irreducible $\M_k(x)$-modules.  
\end{thm}

\begin{proof} First we prove that $\C_k^{(r)}$ is indecomposable.  Let $p$ be a Motzkin path in $\mathcal P_k^r$ and let 
$e_p =  {x^{-\varepsilon(p)}}d_p^p \in \M_k(x)$.
The image of $e_p$ on $\C_k^{(r)}$
 is $\KK p$,
so the kernel has codimension one.    Thus, $e_p$
acts with eigenvalues $0$ and $1$ on $\C_k^{(r)}$,  
and  the eigenspace corresponding to the eigenvalue 1 is
1-dimensional with $p$ as a basis by \eqref{eq:Eigenvalue}.   

Let $\mathsf{Z}_k^r = \End_{\M_k(x)} ( \C_k^{(r)} )$, 
and assume that  $z \in \mathsf{Z}_k^r$.   Since $z$ and $e_p$ commute,  $z$ maps the eigenspaces
of $e_p$ into themselves.  However,
then $z(p) = \lambda p$ for some $\lambda \in \KK$,  since
the eigenspace corresponding to the eigenvalue 1 is 1-dimensional.
Now let $q$ be another path in $\mathcal P_k^r$,
and set $d =  x^{-\varepsilon(p)} d_p^q$ so that  $d\cdot p = q$.
Then  
$$z(q) = z (d\cdot p) = d\cdot z(p) = \lambda d \cdot p = \lambda q.
$$
Consequently  $z$ is a scalar multiple of the identity map. 
But $z$ is an arbitrary  transformation in $\mathsf{Z}_k^r$,  which
must then be 1-dimensional.   This implies that $\C_k^{(r)}$ is indecomposable. 

Now when $x$ is chosen so that $\M_k(x)$ is semisimple, then since
modules for $\M_k(x)$ are completely reducible, and since the module $\C_k^{(r)}$
is indecomposable, $\C_k^{(r)}$ must be irreducible in this case.

To  prove that these modules are pairwise nonisomorphic,   let 
\begin{equation}\label{eq:1ellk}
{\bf 1}_{\ell,k}  = 
\underbrace{\beginpicture
\setcoordinatesystem units <0.4cm,0.12cm>        
\setplotarea x from .5 to 4.5, y from -1 to 2
\put{$\bullet$} at  1 -1  \put{$\bullet$} at  1 2
\put{$\bullet$} at  2 -1  \put{$\bullet$} at  2 2
\put{$\cdots$} at  3 -1  \put{$\cdots$} at  3 2
\put{$\bullet$} at  4 -1  \put{$\bullet$} at  4 2
\plot 1 2 1 -1 /
\plot 2 2 2 -1 /
\plot 4 2 4 -1 /
\endpicture}_{\ell}
\underbrace{\beginpicture
\setcoordinatesystem units <0.4cm,0.12cm>        
\setplotarea x from .5 to 4.5, y from -1 to 2
\put{$\bullet$} at  1 -1  \put{$\bullet$} at  1 2
\put{$\bullet$} at  2 -1  \put{$\bullet$} at  2 2
\put{$\cdots$} at  3 -1  \put{$\cdots$} at  3 2
\put{$\bullet$} at  4 -1  \put{$\bullet$} at  4 2
\endpicture}_{k-\ell}, \qquad 0 \le \ell \le k.
\end{equation}
\noindent The diagram ${\bf 1}_{\ell,k} $ acts by 0 on all $\C_k^{(r)}$ with $r > \ell$, since ${\bf 1}_{\ell,k} $ lowers the rank of any diagram of rank greater than $\ell$, and ${\bf 1}_{\ell,k}$ is
nonzero on $\C_k^{(\ell)}$.  Thus, the diagrams ${\bf 1}_{\ell,k}, \ 0 \le \ell \le k$, are sufficient to distinguish one $\C_k^{(r)}$ from another.

To show that we have constructed all of the irreducible $\M_k(x)$-modules, we sum the squares of the dimensions.  Since
$\C_k^{(r)}$ has a basis consisting of the Motzkin paths of rank $r$, it has dimension $\mathsf{m}_{k,r}$.  But then by \eqref{eq:sumofsqs},
$\sum_{r=0}^k \dim(\C_k^{(r)})^2 = \sum_{r=0}^k \mathsf{m}_{k,r}^2 =\Mtwok,$ which is the dimension of $\M_k(x)$.  Thus, by the general Wedderburn theory for semisimple algebras, we have found all of the irreducible $\M_k(x)$-representations. 
\end{proof}

\begin{remark}{\rm  Let $\mathcal {T}_k^r$ denote the set of all paths in $\cP_k^r$ which
have no zeros, so $k-r$ is a nonnegative even integer.    The definition of the action of diagrams on paths shows that
the Temperley-Lieb diagrams (those Motzkin diagrams having no unconnected vertices) act on these
paths.  The same argument as in the proof of this theorem shows that 
the $\KK$-span $\mathsf{T}_k^{(r)}$ of $\mathcal T_k^r$ is an indecomposable module for the Temperley-Lieb
algebra $\mathsf{TL}_k(x)$, and this module is irreducible when $\mathsf{TL}_k(x)$ is semisimple.
These modules give a complete set of nonisomorphic irreducible modules for the
Temperley-Lieb algebra in the semisimple case. }  \end{remark}

As an immediate consequence of Theorem \ref{thm:irredmod}, we can explicitly derive the restriction rule for $\M_{k-1}(x) \subseteq \M_k(x)$, where this embedding is given by adding to a Motzkin $(k-1)$-diagram two 
vertices $k$,$k'$ on the right and connecting them with a vertical edge.   For
$p = (a_1, \ldots, a_k)$,   set  $p' = (a_1, \ldots, a_{k-1})$.   Then 
$$\rank(p') = \begin{cases}  \rank(p)-1 \ \ \hbox{\rm and} \ \  p' \in \cP_{k-1}^{r-1} \ \ \hbox{\rm  if} \ \  a_k = 1, \\
\rank(p) \ \ \hbox{\rm and} \ \ p' \in \cP_{k-1}^{r} \  \ \hbox{\rm  if} \ \ a_k = 0, \\
\rank(p)+1\ \ \hbox{\rm and} \ \ p' \in \cP_{k-1}^{r+1} \ \ \hbox{\rm  if} \ \ a_k = -1. \end{cases}$$
Thus, the map $p \to p'$ realizes the decomposition,
\begin{equation}\label{eq:restrict}
\Res^{\M_k(x)}_{\M_{k-1}(x)} ( \C_k^{(r)} ) = \C_{k-1}^{(r-1)} \oplus \C_{k-1}^{(r)} \oplus \C_{k-1}^{(r+1)},
\end{equation}
where we define $\C_{k-1}^{(\ell)} = 0$ if $\ell > k-1$ or $\ell < 0.$
This relation is to be expected from \eqref{eq:TensorRule} and the Bratteli diagram (Figure 1), and it is a representation-theoretic interpretation of the identity
\begin{equation}\label{eq:restrictionrule}
\mathsf{m}_{k,r} = \mathsf{m}_{k-1,r-1} +  \mathsf{m}_{k-1,r} + \mathsf{m}_{k-1,r+1},
\end{equation}
which appears in \cite{DS}.   

 \begin{subsection}{Cellularity} \label{subsec:cellular}

In this section, we show that the Motzkin algebra $\M_k(x)$ over
any unital commutative ring $\KK$ is {\em cellular}  in the sense of Graham and Lehrer \cite{GL} and that the modules $\C_k^{(r)}$,  which have a basis  the Motzkin paths in 
$\cP_k^r$,  are its left cell modules.  

According to \cite[Defn.~1.1]{GL}, a unital associative  algebra $\mathsf A$ over a unital commutative ring $\KK$ is \,{\em cellular}\, if it has cell data 
$(\Lambda, \mathcal P, \mathsf{d}, \ast)$ satisfying the following requirements:
\begin{itemize} 
\item[{\rm (C1)}]   {\it  $\Lambda$ is a partially ordered set.  For each $\lambda \in \Lambda$ there is a  finite 
 set $\mathcal P^\lambda$, and $\mathsf{d}: \amalg_{\lambda \in \Lambda}  \mathcal P^\lambda \times \mathcal P^\lambda \to \mathsf{A}$,  $(p,q) \mapsto \mathsf{d}_p^q$, is an injective map which  has as its image a $\KK$-basis of $\mathsf A$.}
 \item[{\rm (C2)}] {\it $\ast$ is  a $\KK$-linear involution of $\mathsf{A}$  such that  $(\mathsf{d}_p^q)^* = \mathsf{d}_q^p$ for all $p,q \in \mathcal P^\lambda$ and all $\lambda \in \Lambda$.}
\item[{\rm (C3)}]  {For $\lambda \in \Lambda$, \  $p,q \in \mathcal P^\lambda$, \  and any $a \in \mathsf{A}$,}  
 
\begin{equation}\label{cell.representation}
 a \mathsf{d}^q_p \equiv \sum_{q' \in \cP^\lambda}  \mu_a(q',q) \mathsf{d}^{q'}_{p} \quad \mod \mathsf{A}(< \lambda),
\end{equation}
{\it where $\mu_a(q',q) \in \KK$ is independent of $p$,  and $\mathsf{A}(< \lambda)$ is the $\KK$-submodule of $\mathsf{A}$
spanned by $\{ \mathsf{d}_{r}^s \mid r,s \in \mathcal P^\mu$ for all $\mu < \lambda$ in $\Lambda\}$.}  
\end{itemize}

Being cellular implies the existence of   cell representations (Definition (2.1) in \cite{GL}).
In particular, for  a cellular algebra $\mathsf{A}$ with cell data $(\Lambda, \cP, \mathsf{d}, *)$, there corresponds to each $\lambda \in \cP$,
a (left) cell representation  $\N^{\lambda}$ having a  free $\KK$-module basis $\{ \cb_q \ \vert\ q\in \cP^\lambda\}$ such that for all $a \in\mathsf{A}$   
$$a \cb_q = \sum_{q' \in \cP^\lambda} \mu_a(q',q) \cb_{q'},$$
 where $\mu_a(q',q)$ is as in \eqref{cell.representation}.  \smallskip

For the Motzkin algebra $\M_{k}(x)$,  the cell data is $(\Lambda_k, \mathcal{P}_k, d, \ast)$ where  
  \begin{itemize} 
\item[{\rm (C1)}]  $\Lambda_k =  \{0, 1, \ldots, k\}$  under the usual ordering of integers.  For $r \in \Lambda_k$, the set $\mathcal P_k^r$ is
the set of Motzkin paths of rank $r$,  and  the map $d:  \amalg_{r\in \Lambda_k}  \cP_k^r \times \cP_k^r \to \M_k(x)$   sends $(p,q)$ to the
diagram  $d_p^q$. 
 The map $d$ is injective  (compare \eqref{eigenfunction}), and clearly its image is the basis of all Motzkin $k$-diagrams.  
\item[{\rm (C2)}]   As observed  in Section \ref{subsec:generators},  the involution $\ast$ on $\M_k(x)$, given by reflecting a diagram over its horizontal axis, satisfies $(d_1 d_2)^\ast = d_2^\ast d_1^\ast$ and $(d^\ast)^\ast = d$.  In particular, $(d_p^q)^\ast = d_q^p$.

\item[{\rm (C3)}] As in Section \ref{subsec:motzkinalgebra}, we let $\J_r  \subseteq \M_k(x)$ be $\KK$-span  of the diagrams $d$ with $\rank(d) \le r$.   The $\J_r$ are ideals satisfying 
$$
\J_0 \subseteq \J_1 \subseteq \cdots \subseteq  \J_k = \M_k(x).
$$
 Let $p,q$ be Motzkin paths of rank $r$. Then for any diagram $d$,  either $\rank(d d_p^q) < \rank(d_p^q)$, or  $\rank(d d_p^q)  =  \rank(d_p^q)$ and $d d_p^q = x^\ell d_{p}^{q'}$, for some nonnegative integer $\ell$.  In the second event, the bottom path $p$ is unchanged.  It follows that for $a \in \M_k(x)$, 
\begin{equation}\label{cell.representation}
 a d^q_p \equiv \sum_{q' \in \cP_k^r}  \mu_a(q',q) d^{q'}_{p} \quad \mod \J_{r-1}.
\end{equation}
where $\mu_a(q',q) \in \KK$ is independent of $p$.  
\end{itemize}

This  demonstrates that the Motzkin algebra $\M_k(x)$ is cellular. 
\smallskip

 For each $r \in \Lambda_k$, the (left) cell representation $\N_k^{(r)}$ for $\M_k(x)$  is the free $\KK$-module with basis $\{ \cb_q \ \vert\ q\in \cP_k^r\}$ such that for all $a \in\M_k(x)$  we have
$$a \cb_q = \sum_{q' \in \cP_k^r} \mu_a(q',q) \cb_{q'},$$
 where $\mu_a(q',q)$ is as in \eqref{cell.representation}.  
The action of $a$ on $d^q_p \mod \mathsf{J}_{r-1}$ is independent of the bottom 1-factor $p$, and when we replace the endpoints of the  horizontal edges in $d^q_p$ with white vertices
and remove the horizontal edges in the bottom row, then the action of $a$ on $d^q_p \mod \mathsf{J}_{r-1}$ is precisely the same as the action of $a$ on $q \mod \W_k^{(r-1)}$, (compare  \eqref{eq:Wkrdef}).   Thus,  the action of $a$ on $q$ is the same as that of $a$ on $\cb_q$, so we have
\begin{equation}
\N_k^{(r)} \cong \C_k^{(r)} \qquad\text{as $\M_k(x)$-modules}.  \end{equation}

The cell representations $\C_k^{(r)}$ come equipped with a $\KK$-bilinear form $\langle \cdot ,\cdot \rangle: \C_k^{(r)} \times \C_k^{(r)}\to \KK$ defined for $p,q \in \cP_k^r$  by the equation\begin{equation}\label{eq:bilinearform}
d_p^p d^q_q \equiv \langle p, q \rangle  d^p_q \quad  \mod \mathsf{J}_{r-1}.
\end{equation}
Observe that 
$$(d_p^p d^q_q)^* \equiv \langle p, q \rangle  (d^p_q)^* \quad  \mod \mathsf{J}_{r-1},$$
which implies that  $\langle p,q \rangle d_p^q  \equiv   d_q^q d_p^p   \equiv   \langle q, p \rangle  d_p^q 
\mod \mathsf{J}_{r-1}$, from which we may deduce that $\langle q,p \rangle =
\langle p,q \rangle$ for all $p,q \in \cP_k^r$.   Moreover, since 
$$\langle a p, q \rangle = \langle p,a^*q\rangle$$ for all $a \in \M_k(x)$, $p,q \in \cP_k^r$ 
(see \cite[Prop. 2.4]{GL}), 
the radical $\mathsf{rad}_k^{(r)}$ of the bilinear form  is an $\M_k(x)$-submodule of $\C_k^{(r)}$.
As a consequence of 
\cite[Prop.~3.2, Thm.~3.8]{GL},  we know (i),(ii) and (iii) of the following:

\begin{thm} \label{thm:cellmod}  The  Motzkin algebra $\M_k(x)$ 
over a unital commutative ring  $\KK$ is cellular and the modules $\C_k^{(r)}$ of \eqref{eq:Cmoddef} are the left cell modules for $\M_k(x)$.   Moreover, when $\KK$ is a field, the following hold:
\begin{itemize}
\item[{\rm (i)}]   If $I_k^{(r)}:= \C_k^{(r)}/ \mathsf{rad}_k^{(r)}$ is nonzero, then it is
an absolutely irreducible $\M_k(x)$-module; that is, it is irreducible over any extension field of $\KK$.
\item[{\rm (ii)}]  $\{I_k^{(r)}  \neq (0) \mid r = 0,1,\dots, k\}$ is a complete set of representatives
of isomorphism classes of absolutely irreducible $\M_k(x)$-modules.  
\item[{\rm (iii)}]  The following are equivalent:
\begin{itemize}
\item[{\rm (a)}]   the algebra $\M_k(x)$ is semisimple;
\item[{\rm (b)}]   for each $r=0,1,\dots, k$, the cell module $\C_k^{(r)}$
is absolutely irreducible;
\item[{\rm (c)}]   the bilinear form 
$\KK$-bilinear form $\langle \cdot , \cdot \rangle: \C_k^{(r)} \times \C_k^{(r)}  \to \KK$ is nondegenerate 
for each $r=0,1,\dots, k$. \end{itemize}  \end{itemize}
\end{thm} 
\end{subsection}

\begin{subsection}{Characters of $\C_k^{(r)}$} 

Recall that a character $\chi$ of $\M_k(x)$ on a module $\W$  is the 
map $\chi: \M_k(x) \rightarrow \KK$ given by the trace on $\W$.
Thus  $\chi$ is $\KK$-linear, and  for any two elements $a,b \in  \M_k(x)$, we have 
$\chi(ab) = \chi(ba)$.    In this section,  we compute the characters of the modules
$\C_k^{(r)}$.  
 
\begin{prop}  Any character $\chi$ of $\M_k(x)$ is completely determined by its values on the diagrams  \ ${\bf 1}_{\ell,k}, \ 0 \le \ell \le k$,  in \eqref{eq:1ellk}. 
\end{prop}

\begin{proof}  
This follows by the Jones basic construction and \cite[Lem.~2.8]{HR}.
  However, we include the explicit calculation since it is useful 
  to see the recursive algorithm for converting a general diagram to its representative of the form ${\bf 1}_{\ell,k}.$
By the linearity of $\chi$,  it is sufficient to compute the values of $\chi$  on the basis of diagrams.   If the rank of a diagram $d$ is $k$, then 
$d = {\bf 1}_{k,k}$.  If the rank of $d$ is less than $k$, then there exist diagrams $d_1, d_2$ such that $d= d_1 p_k d_2$; this can easily be seen by drawing diagrams.   
Thus $\chi(d) = \chi(d_1 p_k d_2) = \chi(d_1 p_k p_k d_2) = \chi(p_k d_2 d_1 p_k ) = x^{- \kappa(d_1,d_2)} \chi(\delta(d_2 d_1) p_k )$, where $\delta(d_2 d_1)$ is a diagram in $\M_{k-1}(x)$ and $\kappa(d_1,d_2)$ is the number of closed cycles produced in the product $d_1 d_2$.  Now, continue this argument on the diagram $\delta(d_2 d_1) \in \M_{k-1}(x)$ until we reach the identity diagram of rank $\ell$ in $\M_\ell(x)$ for some $\ell$.  \end{proof}

\begin{prop}\label{prop:charac}  If $\chi_k^{(r)}$ is the character of the module  $\C_k^{(r)}$, then 
$$
\chi_k^{(r)}( {\bf 1}_{\ell,k}) = \begin{cases}
\dim(\C_\ell^{(r)}) = \mathsf{m}_{\ell,r}, & \text{ if $r \le \ell$} \\
0, & \text{ if $r > \ell$}. \\
\end{cases}
$$
\end{prop}

\begin{proof}  We compute the trace of $d =  {\bf 1}_{\ell,k}$ on the basis of Motzkin paths $p \in \cP_k^r$.  When $d = {\bf 1}_{\ell,k}$ acts on the path $p$ (see (\ref{ActionOnPaths2})),  we see that $d\cdot p = x^{\kappa(d,p)} q$ modulo the span of paths of lower rank.  To contribute to the trace, we must have $p = q$ (an entry on the main diagonal).  If $dp$ has lower rank than $p$,  then $d\cdot p=0$, so we assume that  $\rank(dp) = \rank(p)$. By the structure of $d =  {\bf 1}_{\ell,k}$, the 1-factor $p$ must have isolated vertices in the last $k-\ell$ positions, so for the product to be nonzero we need $\ell \ge r$.  Letting $\tilde p$ be the diagram obtained by dropping the last $k - \ell$  isolated  vertices of $p$, we see that $\tilde p$ is a 1-factor in $\C_\ell^{(r)}$ and that $d$ acts as identity on $\C_\ell^{(r)}$,  so its trace is the dimension $\dim(\C_\ell^{(r)})$, which is $\mathsf{m}_{\ell,r}$.
\end{proof}

\end{subsection} 

\end{section}

 \begin{section}{Gram matrices, Chebyshev Polynomials, and \newline Semisimplicity}
 
 Throughout this section, we assume that $\KK$ is a field, and  we study the Gram matrix $\G_k^{(r)}$ of the cell module $\C_k^{(r)}$ with basis consisting of the Motzkin paths in $\cP_k^r$.  By explicitly changing the basis, we block diagonalize $\G_k^{(r)}$ enabling us to give a precise formula for the determinant $\det(\G_k^{(r)})$ in terms of certain Chebyshev polynomials $u_j$.  We then show that the Motzkin algebra $\M_k(x)$ is semisimple if and only if  the parameter $x$ satisfies   
 $u_j(x-1) \neq 0$ for $1 \leq j \leq k-1$. 
 
 \begin{subsection}{Chebyshev polynomials}\label{sec:CP}
 
For $k \ge 0$, the Chebyshev polynomials of the second kind $U_k(x)$ are defined by $U_0(x) =1, U_1(x) = 2x$, and the three-term recurrence relation 
$$
U_k(x) = 2x U_{k-1}(x) - U_{k-2}(x),  \qquad \hbox{ for } k \geq 2.  
$$
Setting $u_k(x) = U_k(x/2)$,  we have the recurrence
\begin{equation}\label{chebyshevrecursion}
u_k(x) = x u_{k-1}(x) - u_{k-2}(x),  \quad u_0(x) =1, \ \ u_1(x) = x.
\end{equation}
The polynomials $u_k(x-1)$ appear in the determinant of the Gram matrices in this section. The first few of them are
\begin{equation*}
\begin{array}{lll  c  lll}
u_0(x-1) &=& 1, & \hskip.5in & u_3(x-1) &=&  (x-1)(x^2-2x-1), \\
u_1(x-1) &=& x-1,  && u_4(x-1) &=& (x^2-3x+1) (x^2-x-1),\\
u_2(x-1) &=& x (x-2),  && u_5(x-1) &=& x (x-1) (x-2)  (x^2-2x-2).
\end{array}
\end{equation*} 
The polynomial $u_k(x-1)$ has degree $k$ and its roots are given by (see \cite{R}, p. 229),
\begin{equation}
\theta_m = e^{i \pi m /(k+1)} + e^{- i \pi m /(k+1)} + 1 = 2 \cos\left(\frac{\pi m}{k+1}\right) + 1,
\qquad m = 1, \ldots, k.
\end{equation}

\end{subsection}

 \begin{subsection}{Gram matrices}
 
 Let $\cP_k^r$ denote the Motzkin paths (or, equivalently, 1-factors) of length $k$ and rank $r$.  For $p, q \in \cP_k^r$, let $d_q^p$  be the unique Motzkin diagram whose top row is $p$ and whose bottom row is $q$.   Let $\C_k^{(r)}$ be the cell module for $\M_k(x)$ with basis given by $\cP_k^r$.  Recall that for $p, q \in \cP_k^r$,   the bilinear form $\langle \cdot , \cdot \rangle$ is defined on $\C_k^{(r)}$ by  $ d_p^p d_q^q \equiv \langle p,q\rangle d_q^p \mod \J_{r-1}.$
 In particular,  $\langle p, q \rangle = 0$ if $\rank(d_q^p) < r$ and, otherwise, $\langle p, q \rangle = x^{\kappa(p,q)}$, where $\kappa(p,q)$ is the number of loops removed from the middle row in the product $d_p^p d_q^q$.

 Let $\G_k^{(r)}$ be the Gram matrix of $\langle \cdot, \cdot\rangle$ on $\C_k^{(r)}$ so that   $(\G_k^{(r)})_{p,q} = \langle p, q\rangle$ for $p,q \in \cP_k^r$.   The module $\C_k^{(k)}$ has dimension 1 and is spanned by the 1-factor $q$ consisting of all white vertices.  The inner product is $\langle q,q \rangle = 1$.  The module $\C_k^{(k-1)}$ is $k$-dimensional and is spanned by the 1-factors $q_i$ having $k-1$ white vertices and 1 black vertex in position $i$.   It is easy to check by multiplying diagrams that $\langle q_i, q_j \rangle = \delta_{i,j}$,  the Kronecker delta.  Thus,
\begin{equation}
\G_k^{(k)}= (1) \qquad\hbox{and}\qquad \G_k^{(k-1)} = \mathsf{I}_k, \quad\hbox{the $k \times k$ identity matrix,} 
\end{equation}
and $\det(\G_k^{(k)}) = \det(\G_k^{(k-1)}) = 1$.

Order the Motzkin paths in $\cP_k^r$ recursively  as follows.  Let $p = (a_1, a_2, \ldots, a_k)$ and  $q = (b_1, b_2, \ldots, b_k)$ be two Motzkin paths of rank $r$.  Then $p < q$ if and only if
\begin{equation}\label{ordering}
a_k > b_k \quad\text{ or } \quad 
 a_k = b_k \text{ and }  (a_1, \ldots, a_{k-1}) < (b_1, \ldots, b_{k-1}).
\end{equation}
Let $\cP_k^{r} = \cP_k^{r,1} \sqcup \cP_k^{r,0}  \sqcup \cP_k^{r,-1},$ where $\cP_k^{r,\ell}$ is the set of Motzkin paths (1-factors) of length $k$ and rank $r$ for which $a_k = \ell$. 
 If $p \in  \cP_k^{r,1}, q \in  \cP_k^{r,0} $, then $\langle p, q \rangle = 0$, and so under this ordering, the Gram matrix decomposes into the following symmetric block form, 
\begin{equation}\label{eq:GramBlocks}
\G_k^{(r)} =  
\left(
\begin{tabular}{c|c|c}
$\G_{k-1}^{(r-1)}$ & ${\bf 0}$ & ${\bf A}^{\mathsf T} \phantom{\Big\vert}$ \\  \hline
${\bf 0}$ & $\G_{k-1}^{(r)}$ & ${\bf B}^{\mathsf T}\phantom{\Big\vert}$ \\  \hline
${\bf A}\phantom{\Big\vert}$ & ${\bf B}$ & ${\bf H}_{k-1}^{(r+1)}$
\end{tabular} \right),
\end{equation}
where $\mathsf{T}$ denotes the transpose, and ${\bf H}_{k-1}^{(r+1)}$ is a matrix consisting of the inner products $\langle p,q \rangle$ where $p,q$ are Motzkin paths ending in $-1$.    We will change the elements  of $\cP_k^{r,-1}$ to get a basis for
$\C_k^{(r)}$ which  block diagonalizes the Gram matrix in \eqref{eq:GramBlocks}.

\end{subsection}

 \begin{subsection}{Basis change}
 \label{sec:BasisChange}

Let $k\ge2$ and  $0 \le r \le k -2$.  If $p \in  \cP_k^{r,-1}$,  
then $p$ has  a  horizontal  edge connecting vertex $k$ to some vertex $i < k$.  We refer to this horizontal edge as the \emph{pivot edge},  and we denote it with a dashed line in the 1-factor $p$ and in the diagram $d_p^p$.  If $p \in \cP_k^{r,-1}$, then we define $p_\bullet \in  \cP_k^{r,0}$ to be the 1-factor $p$ with the pivot edge deleted, we let  $p_+ \in \cP_k^{r+1,1}$ be the 1-factor such that the pivot edge is replaced with two white vertices, and we define $p^{(1)}  \in  \cP_k^{r,1}$ to be the 1-factor created from $p$ by removing the pivot edge and then adding a new horizontal edge connecting the $i$th vertex with the rightmost white vertex $j < i$,   thus making it black.
We refer to the added horizontal edge as the pivot edge of $p^{(1)}$ and denote it with a dashed line.   If there is no white vertex $j < i$, then $p^{(1)} = 0$.   Here are four examples from $\cP_6$:
$$
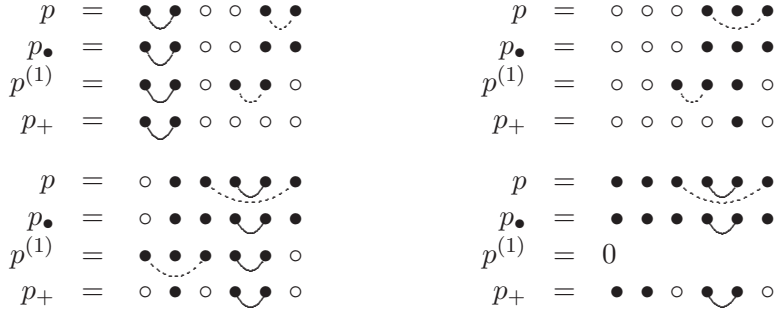
\begin{array} {rcl}
p &  = &
{ \beginpicture
 \setcoordinatesystem units <0.4cm, 0.15cm>        
\setplotarea x from .5 to 6.5, y from -1 to 1
\setquadratic
\setdashes  <.4mm,.6mm>
\plot 5 .5      5.5  -1      6  .5 /
\setsolid
\plot 1 .5      1.5  -1      2  .5 /
\put{$\bullet $} at 1 .5
\put{$\bullet $} at 2 .5
\put{$\whitebull $} at 3 .5
\put{$\whitebull $} at 4 .5
\put{$\bullet$} at 5 .5
\put{$\bullet$} at 6 .5
\endpicture}  \\
p_\bullet  &  = & 
{ \beginpicture
 \setcoordinatesystem units <0.4cm, 0.15cm>        
\setplotarea x from .5 to 6.5, y from -1 to 1
\setquadratic
\plot 1 .5      1.5  -1      2  .5 /
\put{$\bullet $} at 1 .5
\put{$\bullet $} at 2 .5
\put{$\whitebull $} at 3 .5
\put{$\whitebull $} at 4 .5
\put{$\bullet$} at 5 .5
\put{$\bullet$} at 6 .5
\endpicture}  \\
p^{(1)}&  = & 
{ \beginpicture
 \setcoordinatesystem units <0.4cm, 0.15cm>        
\setplotarea x from .5 to 6.5, y from -1 to 1
\setquadratic
\setdashes  <.4mm,.6mm>
\plot 4 .5      4.5  -1      5  .5 /
\setsolid
\plot 1 .5      1.5  -1      2  .5 /
\put{$\bullet $} at 1 .5
\put{$\bullet $} at 2 .5
\put{$\whitebull $} at 3 .5
\put{$\bullet $} at 4 .5
\put{$\bullet$} at 5 .5
\put{$\whitebull $} at 6 .5
\endpicture}  \\
p_+&  = & 
{ \beginpicture
 \setcoordinatesystem units <0.4cm, 0.15cm>        
\setplotarea x from .5 to 6.5, y from -1 to 1
\setquadratic
\setdashes  <.4mm,.6mm>\setsolid
\plot 1 .5      1.5  -1      2  .5 /
\put{$\bullet $} at 1 .5
\put{$\bullet $} at 2 .5
\put{$\whitebull $} at 3 .5
\put{$\whitebull $} at 4 .5
\put{$\whitebull$} at 5 .5
\put{$\whitebull$} at 6 .5
\endpicture}  \\
\end{array}
\qquad \qquad\quad
\begin{array} {rcl}
p & =  &
{ \beginpicture
 \setcoordinatesystem units <0.4cm, 0.15cm>        
\setplotarea x from .5 to 6.5, y from -1 to 1
\setquadratic
\setdashes  <.4mm,.6mm>\plot 4 .5      5  -1      6  .5 /
\setsolid
\put{$\whitebull $} at 1 .5
\put{$\whitebull $} at 2 .5
\put{$\whitebull $} at 3 .5
\put{$\bullet $} at 4 .5
\put{$\bullet$} at 5 .5
\put{$\bullet$} at 6 .5
\endpicture}  \\
p_\bullet &  = & 
{ \beginpicture
 \setcoordinatesystem units <0.4cm, 0.15cm>        
\setplotarea x from .5 to 6.5, y from -1 to 1
\setquadratic
\put{$\whitebull $} at 1 .5
\put{$\whitebull $} at 2 .5
\put{$\whitebull $} at 3 .5
\put{$\bullet $} at 4 .5
\put{$\bullet$} at 5 .5
\put{$\bullet$} at 6 .5
\endpicture}  \\
p^{(1)} &  = & 
{ \beginpicture
 \setcoordinatesystem units <0.4cm, 0.15cm>        
\setplotarea x from .5 to 6.5, y from -1 to 1
\setquadratic
\setdashes  <.4mm,.6mm>\plot 3 .5      3.5 -1      4  .5 /
\setsolid
\put{$\whitebull $} at 1 .5
\put{$\whitebull $} at 2 .5
\put{$\bullet $} at 3 .5
\put{$\bullet $} at 4 .5
\put{$\bullet$} at 5 .5
\put{$\whitebull $} at 6 .5
\endpicture}  \\
p_+ & =  &
{ \beginpicture
 \setcoordinatesystem units <0.4cm, 0.15cm>        
\setplotarea x from .5 to 6.5, y from -1 to 1
\setquadratic
\put{$\whitebull $} at 1 .5
\put{$\whitebull $} at 2 .5
\put{$\whitebull $} at 3 .5
\put{$\whitebull $} at 4 .5
\put{$\bullet$} at 5 .5
\put{$\whitebull$} at 6 .5
\endpicture} 
\end{array}
$$
$$
\begin{array} {rcl}
p &  = &  
{ \beginpicture
 \setcoordinatesystem units <0.4cm, 0.15cm>        
\setplotarea x from .5 to 6.5, y from -1 to 1
\setquadratic
\setdashes  <.4mm,.6mm>\plot 3 .5      4.5  -1.25      6  .5 /
\setsolid
\plot 4 .5      4.5  -.75      5  .5 /
\put{$\whitebull $} at 1 .5
\put{$\bullet $} at 2 .5
\put{$\bullet $} at 3 .5
\put{$\bullet $} at 4 .5
\put{$\bullet$} at 5 .5
\put{$\bullet$} at 6 .5
\endpicture}  \\
p_\bullet  &  = & 
{ \beginpicture
 \setcoordinatesystem units <0.4cm, 0.15cm>        
\setplotarea x from .5 to 6.5, y from -1 to 1
\setquadratic
\plot 4 .5      4.5  -.75      5  .5 /
\put{$\whitebull $} at 1 .5
\put{$\bullet $} at 2 .5
\put{$\bullet $} at 3 .5
\put{$\bullet $} at 4 .5
\put{$\bullet$} at 5 .5
\put{$\bullet$} at 6 .5
\endpicture}  \\
p^{(1)}&  = & 
{ \beginpicture
 \setcoordinatesystem units <0.4cm, 0.15cm>        
\setplotarea x from .5 to 6.5, y from -1 to 1
\setquadratic
\setdashes  <.4mm,.6mm>\plot 1 .5      2  -1.25      3  .5 /
\setsolid
\plot 4 .5      4.5  -.75      5  .5 /
\put{$\bullet $} at 1 .5
\put{$\bullet $} at 2 .5
\put{$\bullet $} at 3 .5
\put{$\bullet $} at 4 .5
\put{$\bullet$} at 5 .5
\put{$\whitebull$} at 6 .5
\endpicture}  \\
p_+&  = & 
{ \beginpicture
 \setcoordinatesystem units <0.4cm, 0.15cm>        
\setplotarea x from .5 to 6.5, y from -1 to 1
\setquadratic
\plot 4 .5      4.5  -.75      5  .5 /
\put{$\whitebull $} at 1 .5
\put{$\bullet $} at 2 .5
\put{$\whitebull $} at 3 .5
\put{$\bullet $} at 4 .5
\put{$\bullet$} at 5 .5
\put{$\whitebull$} at 6 .5
\endpicture}  \\
\end{array}
\qquad \qquad\quad
\begin{array} {rcl}
p &  = &  
{ \beginpicture
 \setcoordinatesystem units <0.4cm, 0.15cm>        
\setplotarea x from .5 to 6.5, y from -1 to 1
\setquadratic
\setdashes  <.4mm,.6mm>\plot 3 .5      4.5  -1.25      6  .5 /
\setsolid
\plot 4 .5      4.5  -.75      5  .5 /
\put{$\bullet $} at 1 .5
\put{$\bullet $} at 2 .5
\put{$\bullet $} at 3 .5
\put{$\bullet $} at 4 .5
\put{$\bullet$} at 5 .5
\put{$\bullet$} at 6 .5
\endpicture}  \\
p_\bullet  &  = & 
{ \beginpicture
 \setcoordinatesystem units <0.4cm, 0.15cm>        
\setplotarea x from .5 to 6.5, y from -1 to 1
\setquadratic
\plot 4 .5      4.5  -.75      5  .5 /
\put{$\bullet $} at 1 .5
\put{$\bullet $} at 2 .5
\put{$\bullet $} at 3 .5
\put{$\bullet $} at 4 .5
\put{$\bullet$} at 5 .5
\put{$\bullet$} at 6 .5
\endpicture}  \\
p^{(1)}&  = & 
0 \\
p_+&  = & 
{ \beginpicture
 \setcoordinatesystem units <0.4cm, 0.15cm>        
\setplotarea x from .5 to 6.5, y from -1 to 1
\setquadratic
\plot 4 .5      4.5  -.75      5  .5 /
\put{$\bullet $} at 1 .5
\put{$\bullet $} at 2 .5
\put{$\whitebull $} at 3 .5
\put{$\bullet $} at 4 .5
\put{$\bullet$} at 5 .5
\put{$\whitebull$} at 6 .5
\endpicture}  \\
\end{array}$$

Now we define recursively
\begin{eqnarray}\label{basischange}
[p] &=& p - p_\bullet - \frac{u_{r-1}(x-1)}{u_{r}(x-1)}[p^{(1)}]   \\
{[p^{(1)}]} &=&  p^{(1)} - \big(p^{(1)}\big)_\bullet - \frac{u_{r-2}(x-1)}{u_{r-1}(x-1)}[p^{(2)}]  \\
\vdots  & & \vdots  \nonumber
\end{eqnarray} 
where $p^{(2)} = \left( p^{(1)}\right)^{(1)}$.    For example, applying these steps recursively to the Motzkin path  $p = { \beginpicture
 \setcoordinatesystem units <0.3cm, 0.15cm>        
\setplotarea x from 1 to 5, y from -1 to 1
\setquadratic
\setdashes  <.4mm,.6mm>\plot 3 .5      4 -1      5  .5 /
\setsolid
\put{$\whitebull $} at 1 .5
\put{$\whitebull $} at 2 .5
\put{$\bullet $} at 3 .5
\put{$\bullet $} at 4 .5
\put{$\bullet$} at 5 .5
\endpicture} 
\in \cP_5^2$, and omitting the argument  $x-1$ to simplify the display, gives
\begin{equation*}
[{ \beginpicture
 \setcoordinatesystem units <0.3cm, 0.15cm>        
\setplotarea x from 1 to 5, y from -1 to 1
\setquadratic
\setdashes  <.4mm,.6mm>\plot 3 .5      4 -1      5  .5 /\setsolid
\put{$\whitebull $} at 1 .5
\put{$\whitebull $} at 2 .5
\put{$\bullet $} at 3 .5
\put{$\bullet $} at 4 .5
\put{$\bullet$} at 5 .5
\endpicture} ] = 
{ \beginpicture
 \setcoordinatesystem units <0.3cm, 0.15cm>        
\setplotarea x from 1 to 5, y from -1 to 1
\setquadratic
\setdashes  <.4mm,.6mm>\plot 3 .5      4 -1      5  .5 / \setsolid
\put{$\whitebull $} at 1 .5
\put{$\whitebull $} at 2 .5
\put{$\bullet $} at 3 .5
\put{$\bullet $} at 4 .5
\put{$\bullet$} at 5 .5
\endpicture}  - { \beginpicture
 \setcoordinatesystem units <0.3cm, 0.15cm>        
\setplotarea x from 1 to 5, y from -1 to 1
\setquadratic
\put{$\whitebull $} at 1 .5
\put{$\whitebull $} at 2 .5
\put{$\bullet $} at 3 .5
\put{$\bullet $} at 4 .5
\put{$\bullet$} at 5 .5
\endpicture} 
- \frac{u_1}{u_2} \left(
{ \beginpicture
 \setcoordinatesystem units <0.3cm, 0.15cm>        
\setplotarea x from 1 to 5, y from -1 to 1
\setquadratic
\setdashes  <.4mm,.6mm>\plot 2 .5      2.5 -1      3  .5 /
\setsolid
\put{$\whitebull $} at 1 .5
\put{$\bullet $} at 2 .5
\put{$\bullet $} at 3 .5
\put{$\bullet $} at 4 .5
\put{$\whitebull $} at 5 .5
\endpicture} 
-
{ \beginpicture
 \setcoordinatesystem units <0.3cm, 0.15cm>        
\setplotarea x from 1 to 5, y from -1 to 1
\setquadratic
\put{$\whitebull $} at 1 .5
\put{$\bullet $} at 2 .5
\put{$\bullet $} at 3 .5
\put{$\bullet $} at 4 .5
\put{$\whitebull $} at 5 .5
\endpicture} \right)
+ \frac{u_0}{u_2} \left(
{ \beginpicture
 \setcoordinatesystem units <0.3cm, 0.15cm>        
\setplotarea x from 1 to 5, y from -1 to 1
\setquadratic
\setdashes  <.4mm,.6mm>\plot 1 .5      1.5 -1      2  .5 /
\setsolid
\put{$\bullet $} at 1 .5
\put{$\bullet $} at 2 .5
\put{$\whitebull $} at 3 .5
\put{$\bullet $} at 4 .5
\put{$\whitebull $} at 5 .5
\endpicture} 
-
{ \beginpicture
 \setcoordinatesystem units <0.3cm, 0.15cm>        
\setplotarea x from 1 to 5, y from -1 to 1
\setquadratic
\put{$\bullet $} at 1 .5
\put{$\bullet $} at 2 .5
\put{$\whitebull $} at 3 .5
\put{$\bullet $} at 4 .5
\put{$\whitebull $} at 5 .5
\endpicture} \right).
\end{equation*}
The set ${\tilde \cP_k^{r}} := \cP_k^{r,1} \sqcup \cP_k^{r,0} \sqcup \{\, [p]\, |\, p  \in \cP_k^{r,-1} \,\}$ is a basis for $\C_k^{(r)}$,  since the change of basis matrix between it
and $\cP_k^{r}$ is  unitriangular.  The next  lemma will help us compute the determinant of the Gram matrix with respect to this new basis.  Observe that the pivot edge of a diagram need not be connected to the $k$th vertex, as it moves to the left when the definition is applied recursively.

\begin{lemma}  \label{GramMatrix:PivotLemma}
Let $p, q \in \cP_k^r$ be such that $p$ has a pivot edge and there are $s$ vertical edges to the left of the pivot edge in $d_p^p$.  Then
$\langle [p], q \rangle = \frac{u_{s+1}(x-1)}{u_{s}(x-1)}  \langle p_\bullet, q \rangle$, if the pivot edge is part of an inner loop in the product $d_p^p d_q^q$. Otherwise, $\langle [p], q \rangle = 0$.
\end{lemma}

\begin{proof} {\it Throughout the proof, we  omit the argument $x-1$ from the polynomials $u_s$ to simplify the notation.}
We prove the result in four cases depending on the form of the path that contains the pivot edge in the product $d_p^p d_q^q$.

\medskip
\noindent{\bf Case 1}:  The pivot edge of $p$ is part of a loop in the middle row of $d_p^p d_q^q$.
\medskip

 First assume that there are $s \ge 2$ vertical edges to the left of the pivot edge. Then the products $d_p^p d_q^q$ and $d_{p_\bullet}^{p_\bullet} d_q^q$ take the following  forms, respectively,
$$
d_p^p d_q^q = 
\begin{array}{c}
 \begin{array}{c}{ \beginpicture
 \setcoordinatesystem units <0.35cm, 0.4cm>        
\setplotarea x from -2 to 10, y from -1 to 1
\plot -2 -1 10 -1 10 1 -2 1 -2 -1 /
\plot -.25 1 -.25 -1 /
\plot 1.25 1 1.25 -1 /
\setquadratic
\setdashes  <.4mm,.6mm>\plot 4 1   5 .5   6 1 /
\plot 4 -1   5 -.5   6 -1 /
\setsolid
\plot 7 -1  7.5 -.5   8 -1 /
\plot 7 1  7.5 .5   8 1 /
\plot 2 -1  2.5 -.5   3 -1 /
\plot 2 1  2.5 .5   3 1 /
\put{$\cdots$} at -1 0
\put{$\cdots$} at .6 0
\put{$\cdots$} at 9 0
\endpicture}\end{array} \\
\begin{array}{c}{ \beginpicture
 \setcoordinatesystem units <0.35cm, 0.4cm>        
\setplotarea x from -2 to 10, y from -1 to 1.25
\plot -2 -1 10 -1 10 1 -2 1 -2 -1 /
\setquadratic
\plot 6 -1  6.5 -.75   7 -1 /
\plot 6 1  6.5 .75   7 1 /
\plot 3 -1  3.5 -.75   4 -1 /
\plot 3 1  3.5 .75   4 1 /
\plot 2 -1  5.0 -.25   8 -1 /
\plot 2 1  5.0 .25   8 1 /
\put{$\cdots$} at -1 0
\put{$\cdots$} at .5 0
\put{$\cdots$} at 9 0
\endpicture}\end{array}
\end{array}
\quad \hbox{ and } \quad
d_{p_\bullet}^{p_\bullet} d_q^q = 
\begin{array}{c}
\begin{array}{c}{ \beginpicture
 \setcoordinatesystem units <0.35cm, 0.4cm>        
\setplotarea x from -2 to 10, y from -1 to 1
\plot -2 -1 10 -1 10 1 -2 1 -2 -1 /
\plot -.25 1 -.25 -1 /
\plot 1.25 1 1.25 -1 /
\setquadratic
\setdashes  <.4mm,.6mm>\setsolid
\plot 7 -1  7.5 -.5   8 -1 /
\plot 7 1  7.5 .5   8 1 /
\plot 2 -1  2.5 -.5   3 -1 /
\plot 2 1  2.5 .5   3 1 /
\put{$\cdots$} at -1 0
\put{$\cdots$} at .6 0
\put{$\cdots$} at 9 0
\endpicture}\end{array} \\
\begin{array}{c}{ \beginpicture
 \setcoordinatesystem units <0.35cm, 0.4cm>        
\setplotarea x from -2 to 10, y from -1 to 1.25
\plot -2 -1 10 -1 10 1 -2 1 -2 -1 /
\setquadratic
\plot 6 -1  6.5 -.75   7 -1 /
\plot 6 1  6.5 .75   7 1 /
\plot 3 -1  3.5 -.75   4 -1 /
\plot 3 1  3.5 .75   4 1 /
\plot 2 -1  5.0 -.25   8 -1 /
\plot 2 1  5.0 .25   8 1 /
\put{$\cdots$} at -1 0
\put{$\cdots$} at .5 0
\put{$\cdots$} at 9 0
\endpicture}\end{array}
\end{array}
$$
so we have $\langle p, q \rangle = x \langle p_\bullet, q \rangle$. The products $d_{p^{(1)}}^{p^{(1)}} d_q^q$ and  $d_{p^{(2)}}^{p^{(2)}} d_q^q$ take the following forms, respectively,
$$
d_{p^{(1)}}^{p^{(1)}} d_q^q =
\begin{array}{c}
 \begin{array}{c}{ \beginpicture
 \setcoordinatesystem units <0.4cm, 0.4cm>        
\setplotarea x from -2 to 10, y from -1 to 1
\plot -2 -1 10 -1 10 1 -2 1 -2 -1 /
\plot -.25 1 -.25 -1 /
\plot 6 -1 6 1 /
\setquadratic
\setdashes  <.4mm,.6mm>\plot 1.25 1   2.625 .25   4 1 /
\plot 1.25 -1   2.625 -.25   4 -1 /
\setsolid
\plot 7 -1  7.5 -.5   8 -1 /
\plot 7 1  7.5 .5   8 1 /
\plot 2 -1  2.5 -.5   3 -1 /
\plot 2 1  2.5 .5   3 1 /
\put{$\cdots$} at -1 0
\put{$\cdots$} at 9 0
\endpicture}\end{array} \\
\begin{array}{c}{ \beginpicture
 \setcoordinatesystem units <0.4cm, 0.4cm>        
\setplotarea x from -2 to 10, y from -1 to 1.25
\plot -2 -1 10 -1 10 1 -2 1 -2 -1 /
\setquadratic
\plot 6 -1  6.5 -.75   7 -1 /
\plot 6 1  6.5 .75   7 1 /
\plot 3 -1  3.5 -.75   4 -1 /
\plot 3 1  3.5 .75   4 1 /
\plot 2 -1  5.0 -.25   8 -1 /
\plot 2 1  5.0 .25   8 1 /
\put{$\cdots$} at -1 0
\put{$\cdots$} at .5 0
\put{$\cdots$} at 9 0
\endpicture}\end{array}
\end{array}
\hbox{ and } \quad
d_{p^{(2)}}^{p^{(2)}}  d_q^q = 
\begin{array}{c}
\begin{array}{c}{ \beginpicture
 \setcoordinatesystem units <0.4cm, 0.4cm>        
\setplotarea x from -2 to 10, y from -1 to 1
\plot -2 -1 10 -1 10 1 -2 1 -2 -1 /
\plot 6 -1 6 1 /
\plot 4 -1 4 1 /
\setquadratic
\setdashes  <.4mm,.6mm>\plot -.25 1   .5 .25   1.25 1 /
\plot -.25 -1   .5 -.25   1.25 -1 /
\setsolid
\plot 7 -1  7.5 -.5   8 -1 /
\plot 7 1  7.5 .5   8 1 /
\plot 2 -1  2.5 -.5   3 -1 /
\plot 2 1  2.5 .5   3 1 /
\put{$\cdots$} at -1 0
\put{$\cdots$} at 9 0
\endpicture}\end{array} \\
\begin{array}{c}{ \beginpicture
 \setcoordinatesystem units <0.4cm, 0.4cm>        
\setplotarea x from -2 to 10, y from -1 to 1.25
\plot -2 -1 10 -1 10 1 -2 1 -2 -1 /
\setquadratic
\plot 6 -1  6.5 -.75   7 -1 /
\plot 6 1  6.5 .75   7 1 /
\plot 3 -1  3.5 -.75   4 -1 /
\plot 3 1  3.5 .75   4 1 /
\plot 2 -1  5.0 -.25   8 -1 /
\plot 2 1  5.0 .25   8 1 /
\put{$\cdots$} at -1 0
\put{$\cdots$} at .5 0
\put{$\cdots$} at 9 0
\endpicture}\end{array}
\end{array}.
$$
First, note that $\langle p^{(1)}, q \rangle = \langle p_\bullet, q \rangle$.
Second, observe that the rank goes down in the product $d_{p^{(2)}}^{p^{(2)}} d_q^q$. In fact, since the pivot edges move to the left, this will be true for all  $d_{p^{(i)}}^{p^{(i)}} d_q^q$ for $i \ge 2.$  Thus $\langle [p^{(2)}],q\rangle = 0$.  Furthermore, $\langle p^{(1)}_\bullet, q\rangle = 0$, since as can be seen in the product $d_{p^{(1)}}^{p^{(1)}} d_q^q$, it is not possible for the rightmost edge of $d_{p^{(1)}}^{p^{(1)}}$ to propagate from bottom to top when the pivot edge of $p^{(1)}$ is removed.   It follows that
\begin{align*}
\langle [p], q \rangle &= \langle p, q \rangle - \langle p_\bullet, q\rangle - \frac{u_{s-1}}{u_s} \left[ \langle p^{(1)},q\rangle-
\langle p^{(1)}_\bullet,q\rangle - \frac{u_{s-2}}{u_{s-1}} \langle [p^{(2)},q\rangle \right]
\\
&= (x-1) \langle p_\bullet, q\rangle - \frac{u_{s-1}}{u_s} \langle p_\bullet,q\rangle 
= \frac{u_{s+1}}{u_s} \langle p_\bullet,q\rangle,
\end{align*}
using the recursion $(x-1) u_s - u_{s-1} = u_{s+1}$.

When $s = 1$, there is only one vertical edge to the  left of the pivot, so $p^{(2)} = 0$ and the proof follows as above, since we still have  $\langle p^{(2)}, q \rangle = 0$.  If $s = 0$, then $p^{(1)} = p^{(2)} = 0$, and the above calculation shows that $\langle [p],q] \rangle = (x-1) \langle p_\bullet,q \rangle = \frac{u_1}{u_0}  \langle p_\bullet,q \rangle$, as desired.

\medskip
\noindent{\bf Case 2}:  The pivot edge of $p$ is part of a path that propagates from bottom to top   in $d_p^p d_q^q$.
\medskip

 In this case, we consider  products of the form 
$$
d_p^p d_q^q = 
\begin{array}{c}
 \begin{array}{c}{ \beginpicture
 \setcoordinatesystem units <0.35cm, 0.4cm>        
\setplotarea x from 0 to 11, y from -1 to 1
\plot 0 -1 11 -1 11 1 0 1 0 -1 /
\plot 2 -1 2 1 /
\setquadratic
\setdashes  <.4mm,.6mm>\plot 6 1   7.375 .25   8.75 1 /
\plot 6 -1   7.375 -.25   8.75 -1 /
\setsolid
\plot 7 -1  7.5 -.5   8 -1 /
\plot 7 1  7.5 .5   8 1 /
\plot 3 -1  4 -.5   5 -1 /
\plot 3 1  4 .5    5 1 /
\put{$\cdots$} at 1 0
\put{$\cdots$} at 10 0
\endpicture}\end{array} \\
\begin{array}{c}{ \beginpicture
 \setcoordinatesystem units <0.35cm, 0.4cm>        
\setplotarea x from 0 to 11, y from -1 to 1.25
\plot 0 -1 11 -1 11 1 0 1 0 -1 /
\plot 7 -1 7 1 /
\setquadratic
\plot 2 -1  2.5 -.25   3 -1 /
\plot 2 1  2.5 .25   3 1 /
\plot 5 -1  5.5 -.5   6 -1 /
\plot 5 1  5.5 .5   6 1 /
\plot 8 -1  8.375 -.5   8.75 -1 /
\plot 8 1  8.375 .5   8.75 1 /
\put{$\cdots$} at 1 0
\put{$\cdots$} at 10 0
\endpicture}\end{array}
\end{array}
\quad \hbox{ and } \quad
d_{p^{(1)}}^{p^{(1)}} d_{q}^{q} = 
\begin{array}{c}
 \begin{array}{c}{ \beginpicture
 \setcoordinatesystem units <0.35cm, 0.4cm>        
\setplotarea x from 0 to 11, y from -1 to 1
\plot 0 -1 11 -1 11 1 0 1 0 -1 /
\plot 8.75 1 8.75 -1 /
\setquadratic
\setdashes  <.4mm,.6mm>\plot 2 1   4 .25   6 1 /
\plot 2 -1   4 -.25   6 -1 /
\setsolid
\plot 7 -1  7.5 -.5   8 -1 /
\plot 7 1  7.5 .5   8 1 /
\plot 3 -1  4 -.5   5 -1 /
\plot 3 1  4 .5    5 1 /
\put{$\cdots$} at 1 0
\put{$\cdots$} at 10 0
\endpicture}\end{array} \\
\begin{array}{c}{ \beginpicture
 \setcoordinatesystem units <0.35cm, 0.4cm>        
\setplotarea x from 0 to 11, y from -1 to 1.25
\plot 0 -1 11 -1 11 1 0 1 0 -1 /
\plot 7 -1 7 1 /
\setquadratic
\plot 2 -1  2.5 -.25   3 -1 /
\plot 2 1  2.5 .25   3 1 /
\plot 5 -1  5.5 -.5   6 -1 /
\plot 5 1  5.5 .5   6 1 /
\plot 8 -1  8.375 -.5   8.75 -1 /
\plot 8 1  8.375 .5   8.75 1 /
\put{$\cdots$} at 1 0
\put{$\cdots$} at 10 0
\endpicture}\end{array}
\end{array}.
$$
Now, $\langle p_\bullet, q \rangle = 0$, since the path no longer propagates  when the pivot edge is removed.   Furthermore, Case 1 tells us that $\langle [p^{(1)}], q \rangle = \frac{u_s}{u_{s-1}}\langle p^{(1)}_\bullet, q \rangle$, since $d_{p^{(1)}}^{p^{(1)}}$ now has $s-1$ vertical edges to the left of the pivot edge.  Also, by comparing the above products of diagrams (with the pivot edge removed in the second product), we see that $\langle p,q \rangle = \langle p^{(1)}_\bullet, q \rangle$.   Thus,
$$
\langle [p], q \rangle = \langle p,q \rangle - \langle p_\bullet, q \rangle - \frac{u_{s-1}}{u_{s}} \langle [p^{(1)}], q \rangle 
= \langle p,q \rangle - \frac{u_{s-1}}{u_{s}} \frac{u_{s}}{u_{s-1}}  \langle p^{(1)}_\bullet,q \rangle 
 = \langle p,q \rangle - \langle p,q\rangle = 0.
$$

\medskip
\noindent{\bf Case 3}:  In the product $d_p^p d_q^q$, the pivot edge of $p$ is not in a closed loop and is not in a path that contains a vertical edge of $q$.
\medskip

Consider products of the form shown below, where either of the dashed vertical edges may or may not be present, 
$$
d_p^p d_q^q = 
\begin{array}{c}
 \begin{array}{c}{ \beginpicture
 \setcoordinatesystem units <0.35cm, 0.4cm>        
\setplotarea x from 0 to 11, y from -1 to 1.
\plot 0 -1 11 -1 11 1 0 1 0 -1 /
\setdashes  <.4mm,.6mm>\plot 2 -1 2 1 /
\plot 3 -1 3 1 /
\setsolid
\setquadratic
\setdashes  <.4mm,.6mm>\plot 8 1   8.5 .5   9 1 /
\plot 8 -1   8.5 -.5   9 -1 /
\setsolid
\plot 5 -1  5.5 -.5   6 -1 /
\plot 5 1  5.5 .5   6 1 /
\plot 4 -1  5.5 -.25   7 -1 /
\plot 4 1  5.5 .25    7 1 /
\put{$\cdots$} at 1 0
\put{$\cdots$} at 10 0
\endpicture}\end{array} \\
\begin{array}{c}{ \beginpicture
 \setcoordinatesystem units <0.35cm, 0.4cm>        
\setplotarea x from 0 to 11, y from -1 to 1.25
\plot 0 -1 11 -1 11 1 0 1 0 -1 /
\setquadratic
\plot 2 -1  3.5 -.25   5 -1 /
\plot 2 1  3.5 .25   5 1 /
\plot 3 -1  3.5 -.5   4 -1 /
\plot 3 1  3.5 .5   4 1 /
\plot 7 -1  7.5 -.5   8 -1 /
\plot 7 1  7.5 .5   8 1 /
\plot 6 -1  7.5 -.25   9 -1 /
\plot 6 1  7.5 .25   9 1 /
\put{$\cdots$} at 1 0
\put{$\cdots$} at 10 0
\endpicture}\end{array}
\end{array}
\quad\hbox{ and }\quad
d_{p^{(1)}}^{p^{(1)}} d_{q}^{q} = 
\begin{array}{c}
 \begin{array}{c}{ \beginpicture
 \setcoordinatesystem units <0.35cm, 0.4cm>        
\setplotarea x from 0 to 11, y from -1 to 1
\plot 0 -1 11 -1 11 1 0 1 0 -1 /
\plot 9 -1 9 1 /
\setdashes  <.4mm,.6mm>\plot 2 -1 2 1 /
\plot 3 -1 3 1 /
\setsolid
\setquadratic
\setdashes  <.4mm,.6mm>\plot 3.5 .5   6 .1   8 1 /
\plot 3.5 -.5   6 -.1   8 -1 /
\setsolid
\plot 5 -1  5.5 -.5   6 -1 /
\plot 5 1  5.5 .5   6 1 /
\plot 4 -1  5.5 -.25   7 -1 /
\plot 4 1  5.5 .25    7 1 /
\put{$\cdots$} at 1 0
\put{$\cdots$} at 10 0
\endpicture}\end{array} \\
\begin{array}{c}{ \beginpicture
 \setcoordinatesystem units <0.35cm, 0.4cm>        
\setplotarea x from 0 to 11, y from -1 to 1.25
\plot 0 -1 11 -1 11 1 0 1 0 -1 /
\setquadratic
\plot 2 -1  3.5 -.25   5 -1 /
\plot 2 1  3.5 .25   5 1 /
\plot 3 -1  3.5 -.5   4 -1 /
\plot 3 1  3.5 .5   4 1 /
\plot 7 -1  7.5 -.5   8 -1 /
\plot 7 1  7.5 .5   8 1 /
\plot 6 -1  7.5 -.25   9 -1 /
\plot 6 1  7.5 .25   9 1 /
\put{$\cdots$} at 1 0
\put{$\cdots$} at 10 0
\endpicture}\end{array}
\end{array}.
$$
Observe that $\langle [p^{(1)}], q \rangle = 0$, since the rank goes down in all products $d_{p^{(i)}}^{p^{(i)}} d_q^q$ for $i \ge 1$ (the pivot edge keeps moving left and so the rightmost vertical  edge in $d_{p^{(i)}}^{p^{(i)}}$ never propagates).  Furthermore,  $\langle p, q \rangle = \langle p_\bullet, q \rangle$, since the pivot edge does not form a loop, so $\langle [p],q \rangle = \langle p,q \rangle - \langle p_\bullet,q \rangle - \frac{u_{s-1}}{u_{s}} \langle [p^{(1)}],q \rangle  = 0$. 

\medskip
\noindent{\bf Case 4}:   In the product $d_p^p d_q^q$, the pivot edge of $p$ is in a path that contains a vertical edge of $d_q^q$ but does not does not contain a vertical edge of $p$.
\medskip

First we consider products of the form shown below, where when following the path starting from the vertical edge in $d_q^q$,  we first hit the left endpoint of the pivot edge, 
$$
d_p^p d_q^q = 
\begin{array}{c}
 \begin{array}{c}{ \beginpicture
 \setcoordinatesystem units <0.35cm, 0.4cm>        
\setplotarea x from 0 to 11, y from -1 to 1
\plot 0 -1 11 -1 11 1 0 1 0 -1 /
\plot 0 -1 0 1 /
\plot 2 -1 2 1 /
\setquadratic
\setdashes  <.4mm,.6mm>\plot 6 1   7.5 .25   9 1 /
\plot 6 -1   7.5 -.25   9 -1 /
\setsolid
\plot 7 -1  7.5 -.5   8 -1 /
\plot 7 1  7.5 .5   8 1 /
\plot 4 -1  4.5 -.5   5 -1 /
\plot 4 1  4.5 .5    5 1 /
\put{$\cdots$} at 1 0
\put{$\cdots$} at 10 0
\endpicture}\end{array} \\
\begin{array}{c}{ \beginpicture
 \setcoordinatesystem units <0.35cm, 0.4cm>        
\setplotarea x from 0 to 11, y from -1 to 1.25
\plot 0 -1 11 -1 11 1 0 1 0 -1 /
\plot 4 1 4 -1 /
\setquadratic
\plot 8 -1  8.5 -.5   9 -1 /
\plot 8 1  8.5 .5   9 1 /
\plot 5 -1  5.5 -.5   6 -1 /
\plot 5 1  5.5 .5   6 1 /
\put{$\cdots$} at 1 0
\put{$\cdots$} at 10 0
\endpicture}\end{array}
\end{array}
\quad\hbox{ and }\quad
d_{p^{(1)}}^{p^{(1)}} d_{q}^{q} = 
\begin{array}{c}
 \begin{array}{c}{ \beginpicture
 \setcoordinatesystem units <0.35cm, 0.4cm>        
\setplotarea x from 0 to 11, y from -1 to 1
\plot 0 -1 11 -1 11 1 0 1 0 -1 /
\plot 0 -1 0 1 /
\plot 9 1 9 -1 /
\setquadratic
\setdashes  <.4mm,.6mm>\plot 2 1   4 .25   6 1 /
\plot 2 -1   4 -.25   6 -1 /
\setsolid
\plot 7 -1  7.5 -.5   8 -1 /
\plot 7 1  7.5 .5   8 1 /
\plot 4 -1  4.5 -.5   5 -1 /
\plot 4 1  4.5 .5    5 1 /
\put{$\cdots$} at 1 0
\put{$\cdots$} at 10 0
\endpicture}\end{array} \\
\begin{array}{c}{ \beginpicture
 \setcoordinatesystem units <0.35cm, 0.4cm>        
\setplotarea x from 0 to 11, y from -1 to 1.25
\plot 0 -1 11 -1 11 1 0 1 0 -1 /
\plot 4 1 4 -1 /
\setquadratic
\plot 8 -1  8.5 -.5   9 -1 /
\plot 8 1  8.5 .5   9 1 /
\plot 5 -1  5.5 -.5   6 -1 /
\plot 5 1  5.5 .5   6 1 /
\put{$\cdots$} at 1 0
\put{$\cdots$} at 10 0
\endpicture}\end{array}
\end{array}.
$$
We have $\langle p^{(1)}, q \rangle = 0$, since the rank will go down in all products as the pivot edge moves left.  Furthermore,  $\langle p,q \rangle = \langle p_\bullet, q \rangle = 0$, since no loop is removed,  so $\langle [p],q \rangle = \frac{u_{s-1}}{u_{s}} \langle [p^{(1)}], q \rangle = 0$.

Second we consider products of the form shown below, where the path starting from the vertical edge in $d_q^q$  first hits the right endpoint of the pivot edge,  In this case, we will use induction on $s$ (the number of vertical edges to the left of the pivot edge) to prove that $\langle [p],q \rangle = 0$.  First, let $s \ge 1$,  
$$
d_p^p d_q^q = 
\begin{array}{c}
 \begin{array}{c}{ \beginpicture
 \setcoordinatesystem units <0.35cm, 0.4cm>        
\setplotarea x from 0 to 11, y from -1 to 2.4
\plot 0 -1 11 -1 11 1 0 1 0 -1 /
\plot 0 -1 0 1 /
\plot 2 -1 2 1 /
\setquadratic
\setdashes  <.4mm,.6mm>\plot 6 1   7.5 .25   9 1 /
\plot 6 -1   7.5 -.25   9 -1 /
\setsolid
\plot 7 -1  7.5 -.5   8 -1 /
\plot 7 1  7.5 .5   8 1 /
\plot 4 -1  4.5 -.5   5 -1 /
\plot 4 1  4.5 .5    5 1 /
\put{$\cdots$} at 1 0
\put{$\cdots$} at 10 0
\put{$\overbrace{\phantom{123456}}$} at 1.75 1.25 
\put{{$\scriptstyle{s ~ vertical ~ edges}$}} at 3.2 2.4 
\endpicture}\end{array} \\
\begin{array}{c}{ \beginpicture
 \setcoordinatesystem units <0.35cm, 0.4cm>        
\setplotarea x from 0 to 11, y from -1 to 1.25
\plot 0 -1 11 -1 11 1 0 1 0 -1 /
\plot 7 1 7 -1 /
\setquadratic
\plot 8 -1  8.5 -.5   9 -1 /
\plot 8 1  8.5 .5   9 1 /
\plot 4 -1  5 -.25   6 -1 /
\plot 4 1  5 .25   6 1 /
\put{$\cdots$} at 1 0
\put{$\cdots$} at 10 0
\endpicture}\end{array}
\end{array}
\quad\hbox{ and }\quad
d_{p^{(1)}}^{p^{(1)}} d_{q}^{q} = 
\begin{array}{c}
 \begin{array}{c}{ \beginpicture
 \setcoordinatesystem units <0.35cm, 0.4cm>        
\setplotarea x from 0 to 11, y from -1 to 2.4
\plot 0 -1 11 -1 11 1 0 1 0 -1 /
\plot 0 -1 0 1 /
\plot 9 1 9 -1 /
\setquadratic
\setdashes  <.4mm,.6mm>\plot 2 1   4 .25   6 1 /
\plot 2 -1   4 -.25   6 -1 /
\setsolid
\plot 7 -1  7.5 -.5   8 -1 /
\plot 7 1  7.5 .5   8 1 /
\plot 4 -1  4.5 -.5   5 -1 /
\plot 4 1  4.5 .5    5 1 /
\put{$\cdots$} at 1 0
\put{$\cdots$} at 10 0
\put{$\overbrace{\phantom{1}}$} at 1.0 1.25 
\put{{$\scriptstyle{s-1 ~ vertical ~ edges}$}} at 3.75 2.4 
\endpicture}\end{array} \\
\begin{array}{c}{ \beginpicture
 \setcoordinatesystem units <0.35cm, 0.4cm>        
\setplotarea x from 0 to 11, y from -1 to 1.25
\plot 0 -1 11 -1 11 1 0 1 0 -1 /
\plot 7 1 7 -1 /
\setquadratic
\plot 8 -1  8.5 -.5   9 -1 /
\plot 8 1  8.5 .5   9 1 /
\plot 4 -1  5 -.25   6 -1 /
\plot 4 1  5 .25   6 1 /
\put{$\cdots$} at 1 0
\put{$\cdots$} at 10 0
\endpicture}\end{array}
\end{array}.
$$

Again, $\langle p,q \rangle = \langle p_\bullet, q \rangle = 0$, so $\langle [p],q \rangle = \frac{u_{s-1}}{u_{s}} \langle [p^{(1)}], q \rangle$. If, in $d_{p^{(1)}}^{p^{(1)}} d_q^q$, the path containing the pivot edge does not reach the bottom of $d_q^q$, then we are in Case 3, and $\langle [p^{(1)}], q \rangle = 0.$  If, in  $d_{p^{(1)}}^{p^{(1)}} d_q^q$, the path containing the pivot edge reaches the bottom of $d_q^q$, then since we have reduced the number of vertical edges that are to the left of the pivot edge, we apply induction to conclude that $\langle [p],q \rangle = \frac{u_{s-1}}{u_{s}} \langle [p^{(1)}], q \rangle = 0$.  In the base case, when $s = 0$, we have $p^{(1)} = 0$, so the result follows immediately.
\end{proof}

\begin{prop}  \label{GramMatrixProp}
Let $p, q \in \cP_k^r$ with $p \in \cP_k^{r,-1}$. Then
\begin{enumerate} 
\item[{\rm (i)}] $\langle [p], q \rangle = 0$, if $q \in \cP_k^{r,0} \sqcup \cP_k^{r,1}$,
\item[{\rm (ii)}] $\langle [p], [q] \rangle = \frac{u_{r+1}(x-1)}{u_{r}(x-1)} \langle p_+, q_+ \rangle$, if $q \in \cP_k^{r,-1}$.
\end{enumerate}
\end{prop}

\begin{proof} {\rm (i)} If $p \in \cP_k^{r,-1}$ and  $q \in \cP_k^{r,0} \sqcup \cP_k^{r,1}$, then the path that contains the pivot edge of $p$ in $d_p^p d_q^q$ is not a loop, so by Lemma \ref{GramMatrix:PivotLemma}, we have $\langle[p],q\rangle = 0$.

\medskip

{\rm (ii)} If $p, q \in \cP_k^{r,-1}$  then $q_\bullet \in \cP_k^{r,0}$, so $\langle [p], q_\bullet \rangle = 0$, by part (a).   Furthermore, each summand in $q^{(1)}$ is in $\cP_k^{r,1}$, so  $\langle [p], [q^{(1)}] \rangle = 0$, also by part (a).
Thus,
$$
\langle [p] , [q] \rangle = \langle [p],q\rangle - \langle [p], q_\bullet \rangle - \frac{u_{r-1}(x-1)}{u_r(x-1)} \langle [p],[q^{(1)}] \rangle
= \langle [p],q \rangle.
$$
Let the pivot edge of $p$ connect vertices $i$ and $k$, and let the rightmost horizontal edge of $q$ connect vertex $j$ and $k$ (the right endpoint must be $k$ in each case, since $p, q \in \cP_k^{r,-1}$).  Thus, the products   $d_p^p d_q^q$ and
$d_{p_+}^{p_+} d_{q_+}^{q_+}$ look like,
$$
d_p^p d_q^q = 
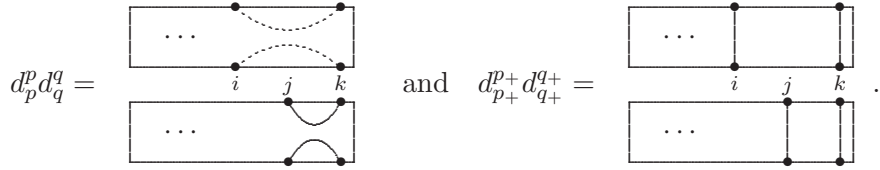
\begin{array}{c}
 \begin{array}{c}{ \beginpicture
 \setcoordinatesystem units <0.35cm, 0.4cm>        
\setplotarea x from 0 to 8.2, y from -1 to 1
\plot 0 -1 8.5 -1 8.5 1 0 1 0 -1 /
\setquadratic
\setdashes  <.4mm,.6mm>\plot 4 1   6 .25   8 1 /
\plot 4 -1   6 -.25   8 -1 /
\setsolid
\put{$\cdots$} at 2 0
\put{$\scriptstyle{\bullet}$} at 4 1 \put{$\scriptstyle{\bullet}$} at 4 -1 
\put{$\scriptstyle{\bullet}$} at 8 1 \put{$\scriptstyle{\bullet}$} at 8 -1 
\put{$\scriptstyle{k}$} at 8 -1.5
\put{$\scriptstyle{i}$} at 4 -1.5
\put{$\scriptstyle{j}$} at 6 -1.6
\endpicture}\end{array} \\
 \begin{array}{c}{ \beginpicture
 \setcoordinatesystem units <0.35cm, 0.4cm>        
\setplotarea x from 0 to 8.2, y from -1 to 1.25
\plot 0 -1 8.5 -1 8.5 1 0 1 0 -1 /
\setquadratic
\plot 6 1   7 .25   8 1 /
\plot 6 -1   7 -.25   8 -1 /
\put{$\cdots$} at 2 0
\put{$\scriptstyle{\bullet}$} at 6 1 \put{$\scriptstyle{\bullet}$} at 6 -1 
\put{$\scriptstyle{\bullet}$} at 8 1 \put{$\scriptstyle{\bullet}$} at 8 -1 
\endpicture}\end{array} 
\end{array}
\quad\hbox{and}\quad
d_{p_+}^{p_+} d_{q_+}^{q_+} = 
\begin{array}{c}
 \begin{array}{c}{ \beginpicture
 \setcoordinatesystem units <0.35cm, 0.4cm>        
\setplotarea x from 0 to 8.2, y from -1 to 1
\plot 0 -1 8.5 -1 8.5 1 0 1 0 -1 /
\plot 4 -1 4 1 /
\plot 8 -1 8 1 /
\setquadratic
\put{$\cdots$} at 2 0
\put{$\scriptstyle{\bullet}$} at 4 1 \put{$\scriptstyle{\bullet}$} at 4 -1 
\put{$\scriptstyle{\bullet}$} at 8 1 \put{$\scriptstyle{\bullet}$} at 8 -1 
\put{$\scriptstyle{k}$} at 8 -1.5
\put{$\scriptstyle{i}$} at 4 -1.5
\put{$\scriptstyle{j}$} at 6 -1.6
\endpicture}\end{array} \\
 \begin{array}{c}{ \beginpicture
 \setcoordinatesystem units <0.35cm, 0.4cm>        
\setplotarea x from 0 to 8.2, y from -1 to 1.25
\plot 0 -1 8.5 -1 8.5 1 0 1 0 -1 /
\plot 6 -1 6 1 /
\plot 8 -1 8 1 /
\put{$\cdots$} at 2 0
\put{$\scriptstyle{\bullet}$} at 6 1 \put{$\scriptstyle{\bullet}$} at 6 -1 
\put{$\scriptstyle{\bullet}$} at 8 1 \put{$\scriptstyle{\bullet}$} at 8 -1 
\endpicture}\end{array} 
\end{array}.
$$

If the pivot edge of $p$ is contained in a loop in the middle row in the product $d_p^p d_q^q$ then there is a path in this middle row from vertex $i$ to vertex $j$.   Furthermore, by Lemma \ref{GramMatrix:PivotLemma}, we have $\langle [p], q \rangle = \frac{u_{r-1}(x-1)}{u_r(x-1)} \langle p_\bullet, q \rangle$.   By comparing diagrams, we see $\langle p_\bullet, q \rangle = \langle p_+, q_+ \rangle$, and so the result holds in this case. 

If the pivot edge of $p$ is not contained in a loop in the middle row in the product $d_p^p d_q^q$, then by Lemma \ref{GramMatrix:PivotLemma}, we have $\langle [p], q \rangle = 0$. There is no  path from $i$ to $j$ in the middle row, and so it is not possible for the vertical edge at $j$ in $d_{q_+}^{q_+}$ to propagate through to the vertical edge at $i$ in $d_{p_+}^{p_+}$. Thus $\langle p_+, q_+ \rangle = 0 = \langle [p], q \rangle$.
\end{proof}

\begin{thm}\label{GramMatrixDiagonalized} With respect to the basis ${\tilde \cP_k^{r}} = \cP_k^{r,1} \sqcup \cP_k^{r,0} \sqcup \left\{\ [p]\ \big\vert\ p \in \cP^{r,-1}_k\ \right\}$, the Gram matrix $\G_k^{(r)}$ takes the following block-diagonal form
$$
\G_k^{(r)} =
\left(
\begin{tabular}{c|c|c}
$\G_{k-1}^{(r-1)}$ & ${\bf 0}$ & ${\bf 0}$ \\  \hline
${\bf 0}$ & $\G_{k-1}^{(r)}$ & ${\bf 0}$ \\  \hline
${\bf 0}$ &${\bf 0}$ & $\frac{u_{r+1}(x-1)}{u_{r}(x-1)} \G_{k-1}^{(r+1)}$
\end{tabular} \right),
$$
where $\G_{j}^{(i)}= \emptyset$ if $j < 0$ or $i > j$.   
\end{thm}

\begin{proof}  For $p \in \cP_k^r$, let $p'$ be the Motzkin path given by deleting the $k$th entry in $p$.  Then  $p \mapsto p'$ gives a bijection between $\cP_k^{r,1}$ and $\cP_{k-1}^{r-1}$ and a bijection between $\cP_k^{r,0}$ and $\cP_{k-1}^{r}$.   If $p \in \cP_k^{r,-1}$, then $p'= (p_+)'$, so $p \mapsto  (p_+)'$ is a bijection between $\cP_k^{r,-1}$ and $\cP^{r+1}_{k-1}$.  Since we are simply deleting the last entry, each bijection preserves the ordering of Motzkin paths defined in \eqref{ordering}.   
Now, if $p,q \in \cP_k^{r,1}$ or $p,q \in \cP_k^{r,0}$, then $\langle p, q \rangle = \langle p', q' \rangle$.  If $p,q \in \cP_k^{r,-1}$, then by Proposition \ref{GramMatrixProp}(b), $\langle [p], [q] \rangle = \frac{u_{r+1}(x-1)}{u_r (x-1)} \langle p_+, q_+ \rangle =  \frac{u_{r+1}(x-1)}{u_r (x-1)} \langle p', q' \rangle$.  All other inner products are 0 by Proposition \ref{GramMatrixProp}(a) and the fact that if $p \in  \cP_k^{r,1}, q \in  \cP_k^{r,0} $, then $\langle p, q \rangle = 0$.
\end{proof}

\end{subsection}

 \begin{subsection}{Determinant of the Gram matrix} 
\begin{thm}  \label{thm:GramDet}
Recall from Section \ref{subsec:MotzkinPaths} that $\mathsf{m}_{k,r+2s}$ is the number of Motzkin paths of length $k$ and rank $r + 2s$.   Then, for each $k>0$ and $0 \le r \le k$, we have
\begin{equation}\label{eq:GramDet}
\det( \G_k^{(r)}) = \prod_{s=1}^{\lfloor \frac{k-r}{2} \rfloor} \left(\frac{u_{s+r}(x-1)}{u_{s-1}(x-1)}\right)^{\mathsf {m}_{k,r+2s}}.
\end{equation}
 \end{thm}
 
 \def\m{{\mathsf{m}}}

\begin{proof}  We use induction on $k$, Theorem \ref{GramMatrixDiagonalized}, and the fact that $\mathsf{m}_{k+1,r} = \mathsf{m}_{k,r-1} +\mathsf{m}_{k,r} + \mathsf{m}_{k,r+1}$ as in \eqref{eq:restrictionrule}.  The base case is when $k = 1$ and $r = 0$ or $1$. We have $\G_{1}^{(1)} = \G_{1}^{(0)} = (1)$, which trivially satisfy \eqref{eq:GramDet}.

From Theorem \ref{GramMatrixDiagonalized},  (again omitting the $(x-1)$'s in the expressions) we have 
$$
\det(\G_{k+1}^{(r)}) = \det(\G_{k}^{(r-1)}) \det(\G_k^{(r)}) \det(\G_{k}^{(r+1)}) \left( \frac{u_{r+1}}{u_r} \right)^{\mathsf{m}_{k,r+1}},
$$
and so by induction $\det(\G_{k+1}^{(r)})$ equals
$$
\prod_{i=1}^{\lfloor \frac{k-(r-1)}{2} \rfloor} \left(\frac{u_{i+r-1}}{u_{i-1}}\right)^{\mathsf{m}_{k,r-1+2i}}
\prod_{j=1}^{\lfloor \frac{k-r}{2} \rfloor} \left(\frac{u_{j+r}}{u_{j-1}}\right)^{\mathsf{m}_{k,r+2j}}
\prod_{\ell=1}^{\lfloor \frac{k-(r+1)}{2} \rfloor} \left(\frac{u_{\ell+r+1}}{u_{\ell-1}}\right)^{\mathsf{m}_{k,r+1+2\ell}}
 \left( \frac{u_{r+1}}{u_r} \right)^{\m_{k,r+1}}.
$$
If $s \not= r, r+1$, then the exponent of $u_{s-1}$ in the denominator of this product is
$$
\mathsf{m}_{k,r-1 + 2 s}+ \mathsf{m}_{k,r+ 2 s} + \mathsf{m}_{k,r+1 + 2 s} = \mathsf{m}_{k+1, r+ 2 s},
$$
and the exponent of $u_{s+r}$ in the numerator of this product is
$$
\mathsf{m}_{k,r-1 + 2 (s+1)} + \mathsf{m}_{k,r + 2 s} + \mathsf{m}_{k,r+1 + 2 (s-1)} = 
\mathsf{m}_{k,r + 2 s+1} + \mathsf{m}_{k,r + 2 s} + \mathsf{m}_{k,r + 2 s-1} = \mathsf{m}_{k+1, r+ 2 s}.
$$
The term $u_r$ appears in the  the numerator of the first product when $i = 1$ and its exponent is $\mathsf{m}_{k,r+1}$, which exactly cancels with the last term.  The term $u_{r+1}$ appears when $i = 2$, with exponent $\mathsf{m}_{k,r+3}$, and when $j= 1$ with exponent $\mathsf{m}_{k, r+2}$. It does not appear in the third term, but does appear in the fourth term with exponent $\mathsf{m}_{k,r+1}$.  These exponents sum to $\mathsf{m}_{k,r+3} + \mathsf{m}_{k,r+2} + \mathsf{m}_{k,r+1} = \mathsf{m}_{k+1,r+2}$.  These are exactly the exponents expected in \eqref{eq:GramDet} for $\det(\G_{k+1}^{(r)})$.
\end{proof}

\begin{example} 
\label{ex:DetSpecialCases} The following special cases are immediate consequences of  \eqref{eq:GramDet},
\begin{enumerate}
\item[{\rm (i)}]  $\det(\G_{k}^{(k)}) = \det(\G_{k}^{(k-1)}) = 1$,
\item[{\rm (ii)}]  $\det(\G_{k}^{(k-2)}) = u_{k-1}(x-1),$
\item[{\rm (iii)}]  $\det(\G_{k}^{(k-3)}) = \left(u_{k-2}(x-1)\right)^{k}.$
\end{enumerate}
\end{example}

 As a result of Theorem  \ref{thm:GramDet},  we get the next theorem which gives the precise conditions for  when  $\M_k(x)$ is a semisimple algebra over a field $\KK$ in terms of the Chebyshev
 polynomials in Section \ref{sec:CP}.

\begin{thm}\label{thm:cheby}   As an algebra over the field $\KK$,  the Motzkin algebra $\M_k(x)$ is semisimple if and only if $x \in \KK$ satisfies  $u_j(x-1) \neq 0$ for $1\leq j \leq k-1$. 
\end{thm}

\begin{proof}  
By \cite[Thm.~3.8]{GL}, $\M_k(x)$ is semisimple if and only if the form $\langle \cdot, \cdot \rangle$ is nondegenerate on $\C_k^{(r)}$ for each $0 \le r \le k$. 
This is the equivalent to $\det(\G_k^{(r)}) \not= 0$ for each $0 \le r \le k$.   We proceed by induction on $k$. If $k = 1$, then $\G_1^{(1)} = \G_1^{(0)} = (1)$ and $u_0(x-1) = 1.$
If $k > 1$, then by Theorem \ref{GramMatrixDiagonalized}, $\M_k(x)$ is semisimple if and only if $x$ avoids the zeros of $\det(\G_{k-1}^{(r-1)}), \det(\G_{k-1}^{(r)})$, and $\frac{u_{r+1}(x-1)}{u_r(x-1)}\det(\G_{k-1}^{(r+1)})$.  By induction,  $u_j(x-1) \neq 0$ for
$1 \leq j \leq k-2$, and   $\frac{u_{r+1}(x-1)}{u_r(x-1)} u_j(x-1) \neq 0$ for $1 \le j \le k-2, \  0 \leq r \le k -2$.  The maximum subscript appearing in these expressions is $k-1$,  and it is attained when $r = j= k -2$.  In that case,  $\frac{u_{k-1}(x-1)}{u_{k-2}(x-2)} u_{k-2}(x-2) = u_{k-1}(x-1)$.
\end{proof}
\end{subsection}

\begin{remark} \label{rem:CMPX}{\rm
One can verify that the Motzkin algebras $\M_0(x) \subseteq \M_1(x) \subseteq \cdots$ satisfy the six axioms of \cite{CMPX} to be a ``tower of recollment."  (In fact, the properties of the Jones basic construction, verified in Section \ref{Sec:BasicConstruction}, constitute several of those axioms.)  It then follows from the arguments used in \cite[Thm.~1.1 (ii), Cor.~5.1]{CMPX} that $\M_k(x)$ is semisimple  if and only if $x \in \KK$ is chosen such that $\det(\G_{j}^{(i)}) \not=0$ for each $1 \le j \le k$ and each $i \in\{ j-3, j-2,j-1, j\}$.  Since by Example \ref{ex:DetSpecialCases},  $\det(\G_{j}^{(j-3)}) = \big(u_{j-2}(x-1)\big)^j$,  $\det(\G_{j}^{(j-2)}) = u_{j-1}(x-1)$, and $\det(\G_{j}^{(j-1)}) = \det(\G_{j}^{(j)}) = 1$, 
this approach generates exactly the same set of bad values for the parameter $x$.  However, the proof that 
$\det(\G_{j}^{(j-3)}) = \big(u_{j-2}(x-1)\big)^j$
entails  nearly the same amount of work  as our general proof of Theorems \ref{GramMatrixDiagonalized} and  \ref{thm:GramDet}, and the approach adopted here yields the general form of the determinant in Theorem 
\ref{thm:GramDet} for all $k$ and $r$.}

  \end{remark}
  
  \begin{remark} {\rm The change of basis in  \eqref{basischange} should block-diagonalize the Gram matrix for the Temperley-Lieb algebra (compare \cite{W}),  except in the Temperley-Lieb case the diagram $p_\bullet$ does not exist and should be omitted.}
  \end{remark}

\end{section}

\end{document}